\documentclass{article}
\usepackage{amsmath}
\usepackage{amssymb}
\usepackage{amsfonts}
\usepackage{amscd}
\usepackage{amsthm}
\usepackage{mathrsfs}
\usepackage{enumitem}

\newcommand{\CC}{{\rm\bf C}}
\newcommand{\RR}{{\rm\bf R}}
\newcommand{\QQ}{{\rm\bf Q}}
\newcommand{\ZZ}{{\rm\bf Z}}

\newcommand{\FF}{{\rm\bf F}}

\newcommand{\Adeles}{{\rm\bf A}}

\newcommand{\OO}{\mathcal {O}}

\DeclareMathOperator{\FS}{\mathrm{FS}}
\DeclareMathOperator{\Spec}{\mathrm{Spec}}

\DeclareMathOperator{\GL}{\mathrm{GL}}
\DeclareMathOperator{\SL}{\mathrm {SL}}

\DeclareMathOperator{\Oo}{\mathrm {O}}
\DeclareMathOperator{\U}{\mathrm {U}}
\DeclareMathOperator{\SO}{\mathrm {SO}}
\DeclareMathOperator{\SU}{\mathrm {SU}}

\DeclareMathOperator{\Inn}{\mathrm {Inn}}
\DeclareMathOperator{\Aut}{\mathrm {Aut}}
\DeclareMathOperator{\End}{\mathrm {End}}
\DeclareMathOperator{\Hom}{\mathrm {Hom}}
\DeclareMathOperator{\Ext}{\mathrm {Ext}}

\DeclareMathOperator{\Gal}{\mathrm {Gal}}
\DeclareMathOperator{\ind}{\mathrm{Ind}}
\DeclareMathOperator{\pro}{\mathrm{Pro}}
\DeclareMathOperator{\res}{\mathrm{Res}}

\DeclareMathOperator{\Lie}{\mathrm{Lie}}

\newcommand{\liea}{{\mathfrak {a}}}
\newcommand{\lieb}{{\mathfrak {b}}}
\newcommand{\liec}{{\mathfrak {c}}}
\newcommand{\lied}{{\mathfrak {d}}}

\newcommand{\lieg}{{\mathfrak {g}}}
\newcommand{\lieh}{{\mathfrak {h}}}

\newcommand{\liek}{{\mathfrak {k}}}
\newcommand{\liel}{{\mathfrak {l}}}

\newcommand{\lien}{{\mathfrak {n}}}

\newcommand{\liep}{{\mathfrak {p}}}
\newcommand{\lieq}{{\mathfrak {q}}}

\newcommand{\lieu}{{\mathfrak {u}}}

\newcommand{\liez}{{\mathfrak {z}}}

\theoremstyle{Theorem}
\newtheorem{introconjecture}{Conjecture}

\newtheorem{introtheorem}[introconjecture]{Theorem}

\theoremstyle{plain}
\newtheorem{theorem}{Theorem}[section]
\newtheorem{lemma}[theorem]{Lemma}
\newtheorem{corollary}[theorem]{Corollary}
\newtheorem{proposition}[theorem]{Proposition}
\newtheorem{conjecture}[theorem]{Conjecture}

\theoremstyle{remark}
\newtheorem{remark}[theorem]{Remark}
\newtheorem{definition}[theorem]{Definition}

\begin{document}

\title{Rational Structures on\\Automorphic Representations}
\author{Fabian Januszewski}
\date{\today}

\maketitle

\begin{abstract}
This paper proves the existence of global rational structures on spaces of cusp forms of general reductive groups. We identify cases where the constructed rational structures are optimal, which includes the case of $\GL(n)$. As an application, we deduce the existence of a natural set of periods attached to cuspidal automorphic representations of $\GL(n)$. This has consequences for the arithmetic of special values of $L$-functions that we discuss in \cite{januszewski2015pre,januszewski2015pre2}.

In the course of proving our results, we lay the foundations for a general theory of Harish-Chandra modules over arbitrary fields of characteristic $0$. In this context, a rational character theory, translation functors and an equivariant theory of cohomological induction are developed. We also study descent problems for Harish-Chandra modules in quadratic extensions, where we obtain a complete theory over number fields.
\end{abstract}

{\tableofcontents}

\addcontentsline{toc}{section}{Introduction}
\section*{Introduction}

By Langlands' philosophy and generalized Taniyama-Shimura type modularity conjectures, we expect a close relationship between motives over number fields and certain automorphic representations. To each motive $M$ one may, assuming Grothendieck's Standard Conjectures, attach a field of coeffients $E$ in a natural way, which is a number field. This field plays a fundamental role in arithmetic questions: The coefficients of the $L$-function attached to $M$ all lie in $E$. A famous conjecture of Deligne \cite{deligne1979} predicts that the special values of this $L$-function --- a priori complex numbers --- also lie in $E$.

If we believe that we may attach to $M$ an automorphic representation $\Pi_M$ of $\GL(n)$, it is not obvious how the inherently transcendental object $\Pi_M$ reflects these rationality properties. In \cite{clozel1990}, Clozel went a long way to single out the class of (regular) algebraic automorphic representations of $\GL(n)$ which should correspond to absolutely irreducible motives of rank $n$ (cf.\ loc.\ cit.\ Conjecture 4.16). Clozel went on to show that the finite part of a regular algebraic representation $\Pi$ of $\GL(n)$ is defined over a number field $\QQ(\Pi)$ (cf.\ loc.\ cit.\ Th\'eor\`eme 3.13). This implies that the coefficients of the standard $L$-function of $\Pi$ all lie in $\QQ(\Pi)$. In these cases, Well known results on special values of automorphic $L$-functions show the validity of the automorphic analog of Deligne's Conjecture with $\QQ(\Pi)$ replacing $E$.

An essential feature of Clozel's rational structure is that it is optimal: It exists over the field of rationality of the finite part of $\Pi$ (cf.\ Proposition 3.1 of loc.\ cit.). In his treatment Clozel circumvents the problem of rationality of the archimedean part of $\Pi$ by studying an ad hoc action of the Galois group on the infinity type. This is legitimate since a well known result of Matsushima relates the occurence of an automorphic representation in Betti cohomology to its type at infinity, and Clozel's rational structure originates in the cohomology of arithmetic groups.

Yet the existence of global rational structures on automorphic representations remains an open problem for $\GL(n)$, even more so for more general reductive groups. Until now this was a major obstacle in the study of arithmetic properties of automorphic representations and their $L$-functions, especially when it comes to the finer structure of cohomologically defined periods, as working in the finite part of the representation turns out to be not enough. The raison d'\^etre of this article is to settle this problem once and for all for arbitrary reductive groups.

In the case of $\GL(n)$ our results are optimal. For general reductive groups the situation is more involved. We also discuss global rational structures for not necessarily cohomological representations of groups of Hermitian types, where we allow discrete and non-degenerate limits of discrete series at infinity.

\addcontentsline{toc}{subsection}{Results for the general linear group}
\subsection*{Results for the general linear group}

The global formulation of our main result in the case of $\GL(n)$ is the following (cf.\ Theorems \ref{thm:globalglnrationality} and \ref{thm:globalqrationality}). Fix a number field $F/\QQ$ and a character $\omega: Z(\lieg)\to\CC$ of the center of the universal enveloping algebra of the Lie algebra of $G:=\res_{F/\QQ}\GL_n$, which we assume \lq{}regular and $C$-algebraic\rq{} in the sense that $\omega$ occurs as infinitesimal character of an absolutely irreducible rational representation $M$ of $G$. Furthermore, $M$ is assumed to be essentially conjugate self-dual over $\QQ$, i.e.\ $M^{\vee,\rm c}\cong M\otimes\eta$ for some $\QQ$-rational character $\eta$ of $G$. Consider the space
$$
L_0^2(\GL_n(F)Z(\RR)^0\backslash{}\GL_n(\Adeles_F);\omega)
$$
of automorphic cusp forms which lie in the closure of the smooth forms transforming under $Z(\lieg)\otimes\RR$ according to $\omega$ (with respect to the derived action of $G(\RR)$). The character $\omega$ may be considered as defined over its field of rationality $\QQ(\omega)$, a number field, which agrees with the field of rationality of $M$. Consider the space
\begin{equation}
L_0^2(\GL_n(F)Z(\RR)^0\backslash{}\GL_n(\Adeles_F);\res_{\QQ(\omega)/\QQ}\omega)
\;=\;
\label{eq:l2cuspforms}
\end{equation}
$$
\quad\bigoplus_{\tau:\QQ(\omega)\to\CC}
L_0^2(\GL_n(F)Z(\RR)^0\backslash{}\GL_n(\Adeles_F);\omega^\tau),
$$
where $\tau$ on the right hand side runs through the embeddings of $\QQ(\omega)$ into $\CC$. We may think of the left hand side as a space of vector valued cusp forms.

Assume further given a $\QQ$-rational model $K\subseteq G$ of a maximal compact subgroup of $G(\RR)$, which is semi-admissible in the sense of section \ref{sec:admissiblegroups} and quasi-split over an imaginary quadratic extension of $\QQ$. Such a $\QQ$-rational model exists whenever $F$ is totally real or a CM field (cf.\ section \ref{sec:explicitmodels}). Otherwise we may replace $\QQ$ with a finite extension $\QQ_K/\QQ$  where $K$ has a model and depart from there. For simplicity we discuss only the case $\QQ_K=\QQ$ here. Write $\lieg$ for the $\QQ$-Lie algebra of $G$.

\begin{introtheorem}\label{thm:introglobal}
As a representation of $\GL_n(\Adeles_F)$, the space \eqref{eq:l2cuspforms} is defined over $\QQ$, i.e.\ there is a basis of the space \eqref{eq:l2cuspforms} which generates a $\QQ$-subspace
\begin{equation}
L_0^2(\GL_n(F)Z(\RR)^0\backslash{}\GL_n(\Adeles_F);\res_{\QQ(\omega)/\QQ}\omega)_{\QQ},
\label{eq:l2cuspformsqbar}
\end{equation}
stable under the natural action of $(\lieg,K)\times\GL_n(\Adeles_F^{(\infty)})$. This rational structure has the following properties:
\begin{itemize}
\item[(a)] The complexification
$$
L_0^2(\GL_n(F)Z(\RR)^0\backslash{}\GL_n(\Adeles_F);\res_{\QQ(\omega)/\QQ}\omega)_{\QQ}\otimes\CC
$$
is naturally identified with the subspace of smooth $K$-finite vectors in \eqref{eq:l2cuspforms}.
\item[(b)] To each irreducible subquotient $\Pi$ of the space \eqref{eq:l2cuspforms} corresponds a unique irreducible subquotient $\Pi_{\overline{\QQ}}$ of the $\overline{\QQ}$-rational structure induced by \eqref{eq:l2cuspformsqbar}, and vice versa.
\item[(c)] The module $\Pi_{\overline{\QQ}}$ is defined over its field of rationality $\QQ(\Pi_{\overline{\QQ}})$, which agrees with Clozel's field of rationality $\QQ(\Pi)$ and is a number field.
\item[(d)] As an abstract $\QQ(\Pi)$-rational structure on $\Pi$, $\Pi_{\QQ(\Pi)}\subseteq\Pi$ from (b) is unique up to complex homotheties.
\item[(e)] Taking $(\lieg,K)$-cohomology sends the $\QQ$-structure \eqref{eq:l2cuspformsqbar} to the natural $\QQ$-structure on cuspidal cohomology
\begin{equation}
H_{\rm cusp}^{\bullet}(\GL_n(F)\backslash{}\GL_n(\Adeles_F)/K_\infty;\res_{\QQ(\omega)/\QQ}M^\vee).
\label{eq:hcusp}
\end{equation}
\end{itemize}
\end{introtheorem}


For a more precise statement, we refer to Theorem \ref{thm:globalqrationality} in the text.

The notion of $(\lieg,K)$-cohomology in (e) is a rational variant of the classical one, which is introduced in section \ref{sec:homologyandcohomology}. In particular, the action of $\Aut(\CC/\QQ)$ induced on complex $K$-finite cusp forms by \eqref{eq:l2cuspformsqbar} is compatible with the action of $\Aut(\CC/\QQ)$ on cuspidal cohomology in (e).

The normalization of the rational structure we obtain on \eqref{eq:l2cuspforms} depends on the degree in cohomology chosen in (e) (and odd finite order twists, cf.\ Theorem \ref{thm:globalqrationality}). Since in a degree $q$ lying in the interior of the cuspidal range (i.e.\ $t_0<q<q_0$ where $t_0$ and $q_0$ denote the minimal and maximal degrees where $\Pi$ contributes to cohomology), the multiplicity of an irreducible $\Pi$ in cuspidal cohomology of degree $q$ is an integer $m > 1$, we need to make precise the sense of the normalization we have in mind. Such a normalization may be obtained via appropriate normalization of period matrices.

More concretely Theorem \ref{thm:introglobal} implies the existence of a natural set of periods attached to $\Pi$ (cf.\ Theorem \ref{thm:periods}):

\begin{introtheorem}\label{thm:introperiods}
For each irreducible cuspidal regular algebraic $\Pi$ occuring in \eqref{eq:hcusp} and each $t_0\leq q\leq q_0$ in the cuspidal range and each rational structure $\iota:\Pi_{\QQ(\Pi)}\to\Pi$ on $\Pi$ there is a period matrix $\Omega_q(\Pi,\iota)\in\GL_{m_q}(\CC)$ with the following properties:
\begin{itemize}
\item[(a)] $\Omega_q(\Pi,\iota)$ is the transformation matrix transforming the rational structure $H^q(\lieg,K_\infty; \iota)$ on $(\lieg,K)$-cohomology into the natural $\QQ(\Pi)$-structure on cuspidal cohomology.
\item[(b)] The double coset $\GL_{m_q}(\QQ(\Pi))\Omega_q(\Pi,\iota)\GL_{m_q}(\QQ(\Pi))$ depends only on the pair $(\Pi,\iota)$ and the degree $q$.
\item[(c)] For each $c\in\CC^\times$ we have the relation
$$
\Omega_q(\Pi,c\cdot\iota)\;=\;c\cdot\Omega_q(\Pi,\iota).
$$
\item[(d)] The ratio
$$
\frac{1}
{\Omega_{t_0}(\Pi,\iota)}\cdot
\Omega_q(\Pi,\iota)\;\in\;\GL_{m_q}(\CC)
$$
is independent of $\iota$.
\end{itemize}
\end{introtheorem}

In \cite{januszewski2015pre} and \cite{januszewski2015pre2} the author applies the rational theory developed here to prove period relations for special values of automorphic $L$-functions outside the context of Shimura varieties. This puts the extremal periods $\Omega_{t_0}(\Pi,\iota)$ and $\Omega_{q_0}(\Pi,\iota)$ in Theorem \ref{thm:introperiods} into the context of Deligne's Conjecture \cite{deligne1979}.

\addcontentsline{toc}{subsection}{Results for reductive groups}
\subsection*{Results for reductive groups}

For a general connected linear reductive group $G$ over a number field $F$, Buzzard and Gee \cite{buzzardgee2011} generalized Clozel's notion of algebraicity for $\GL(n)$ to $G$ in order to formulate corresponding conjectures about the existence of Galois representations attached to automorphic representations $\Pi$ of $G(\Adeles_F)$. Buzzard and Gee introduce in fact {\em two} notions of algebraicity for $\Pi$: $C$-algebraicity and $L$-algebraicity, where the first notion generalizes that of a {\em cohomological} representation, and the second notion conjecturally corresponds to those representations which admit Galois representations. Buzzard and Gee show that while these two notions on $G$ do in general not agree, passing from $G$ to a suitable covering group $\widetilde{G}$ enables them to twist $C$-algebraic representations into $L$-algebraic representations and vice versa (cf.\ section 5.2 of loc.\ cit.).

By definition, $C$- and $L$-algebraicity are local properties of the archimedean component $\Pi_\infty$ of an automorphic representation $\Pi$, and a consequence of the existence of motivic Galois representations attached to $\Pi$ would be that the Satake parameters of unramified components $\Pi_v$ at finite places $v$ are algebraic and generate a finite extension $E$ of $\QQ$. A representation with the latter properties is called $L$-arithmetic in loc.\ cit., and the corresponding notion of $C$-arithmeticity means that the unramified representations $\Pi_v$ at finite places $v$ are, as Hecke modules, defined over a number field $E$.

Let $G$ be a connected linear reductive group over $\QQ$. The case of a general number field reduces to this case via restriction of scalars. Shin and Templier showed in \cite{shintemplier2014} that for a cohomological cuspidal automorphic representation $\Pi$ of $G(\Adeles)$ the field of rationality $\QQ(\Pi)$ of $\Pi^{(\infty)}$ is a number field. However they did not construct a rational structure on the latter space defined over a number field. The authors of loc.\ cit.\ exhibit only a $\overline{\QQ}$-rational structure. Such a $\Pi$ is always $C$-algebraic in the sense of \cite{buzzardgee2011}. In particular, $\Pi$ is expected to be $C$-arithmetic, cf.\ Conjecture 3.1.6 of loc.\ cit.

We show that this is indeed the case.

\begin{introtheorem}\label{introthm:finitedefinition}
Let $\Pi$ be any irreducible cohomological cuspidal automorphic representation of $G(\Adeles)$. Then
\begin{itemize}
\item[(a)] The finite part $\Pi_{(\infty)}$ is defined over a finite extension $\QQ(\Pi)^{\rm rat}/\QQ(\Pi)$.
\item[(b)] Any $\QQ(\Pi)^{\rm rat}$-rational structure on $\Pi_{(\infty)}$ is unique up to complex homotheties.
\end{itemize}
\end{introtheorem}


It is not clear if $\QQ(\Pi)^{\rm rat}=\QQ(\Pi)$, because multiplicity one is known to fail in general. We show that the degree of the field of definition over $\QQ(\Pi)$ divides the multiplicity and is therefore bounded (cf.\ Theorem \ref{thm:finitedefinition}, (c)). Corollaries \ref{cor:cuspidalrationaldecomposition} and \ref{cor:qbarcuspidalcohomology} show the existence of a rational decomposition of cuspidal cohomology into irreducibles. 

To state the global extension of Theorem \ref{introthm:finitedefinition}, assume that $K\subseteq G$ is a model of a maximal compact subgroup of $G(\RR)$ defined over a number field $\QQ_K'$, over which every $K(\RR)$-conjugacy class of $\theta$-stable parabolic subalgebras of $\lieg_\CC$ admits a $\QQ_K'$-rational representative.

We prove (cf.\ Theorem \ref{thm:gglobalrational}),

\begin{introtheorem}\label{introthm:gglobalrational}
If $\Pi$ is an irreducible cohomological cuspidal automorphic representation of $G(\Adeles)$ contributing to cohomology with coefficients in an absolutely irreducible rational $G$-module $M$, then the $(\lieg,K)\times G(\Adeles^{(\infty)})$-module $\Pi^{(K)}$ of $K$-finite vectors in $\Pi$ admits a model over the number field $\QQ_K'\QQ_M(\Pi)^{\rm rat}$.
\end{introtheorem}

Under suitable hypotheses on $K(\RR)$ and the model $K$, the field $\QQ_K'$ may be replaced by a smaller field, which yields optimal fields of definition, cf.\ Theorem \ref{thm:realstandardmodules} in the text.

Unlike in the case of $\GL(n)$, we do not obtain a global $\QQ$-structure on natural spaces of cusp forms as in Theorem \ref{thm:introglobal}, due to the lack of control over the field of definition of the finite part. However, we do prove the existence of a $\overline{\QQ}$-rational structure on spaces of cusp forms in the general case (cf.\ Theorem \ref{thm:l2qbarstructure}).

For every factorizable automorphic representation $\Pi$ of $G$, our framework allows us to define for every automorphism $\sigma\in\Aut(\CC/\QQ_K)$ a twisted representation $\Pi^\sigma$. We conjecture that this operation preserves automorphy of cuspidal representations (cf.\ Conjecture \ref{conj:automorphictwisting}). We also formulate the stronger Conjecture \ref{conj:twistingmultiplicities} which predicts that twisting with $\sigma$ preserves multiplicities. Assuming the latter, we determine the field of rationality of the given spaces of cusp forms in Proposition \ref{prop:twistingmultiplicities}.

For $G=\res_{F/\QQ}\GL(n)$, Conjecture \ref{conj:twistingmultiplicities} is a Theorem due to Clozel \cite{clozel1990}, and for discrete and non-degenerate limits of discrete series representations of groups of Hermitian type Blasius, Harris and Ramakrishnan provide in \cite{blasiusharrisramakrishnan1994} more evidence for our Conjectures.

We did not attempt to formulate an even stronger Conjecture which would predict the existence of a rational structure on spaces of cusp forms defined over a {\em number field}. Such a statement would be far from present techniques, whereas Conjectures \ref{conj:automorphictwisting} and \ref{conj:twistingmultiplicities} may be within reach of the Arthur-Selberg trace formula.

In general, the size of the field of rationality is related to a suitable Galois action on the corresponding $L$-packets, which suggests the existence of a {\em rational} Langlands classification, generalizing \cite{langlands1973}. Such a rational classification could reveal fundamental arithmetic patterns in the representation theory of reductive groups, local and global. As already hinted in the previous sections certain rationality patterns reflect motivic arithmetic structure as demonstrated in \cite{januszewski2015pre,januszewski2015pre2}.

\addcontentsline{toc}{subsection}{Results for groups of Hermitian type}
\subsection*{Results for groups of Hermitian type}

For a group $G$ of Hermitian type satisfying Deligne's axioms \cite{deligne1979b} Blasius, Harris and Ramakrishnan showed in \cite{blasiusharrisramakrishnan1994} the existence of a rational structure on the finite part of cuspidal representations $\Pi$ of $G(\Adeles)$ whose infinity component is either a discrete series or a non-degenerate limit of discrete series representation. We globalize this result (cf.\ Theorem \ref{thm:hermitianrationality}).

\begin{introtheorem}\label{introthm:hermitianrationality}
Let $\Pi$ be an irreducible cuspidal automorphic representation of $G(\Adeles)$ with infinity component $\Pi_\infty$. Assume that $\Pi_\infty$ belongs to the discrete series, or is a non-degenerate limit of discrete series. Then the $(\lieg,K)\times G(\Adeles^{(\infty)})$-module $\Pi^{(K)}$ admits a model over a number field $F$.
\end{introtheorem}

Again Theorem \ref{thm:realstandardmodules} provides a list of cases where we find optimal rational structures at infinity over the fields of rationaliy.

\addcontentsline{toc}{subsection}{Methods of proof and outline of the paper}
\subsection*{Methods of proof and outline of the paper}

In order to construct the global rational structures in Theorems \ref{introthm:finitedefinition}, \ref{introthm:gglobalrational} and \ref{introthm:hermitianrationality}, we are naturally led to define rational structures on the infinite-dimensional archimedean part $\Pi_\infty$ of each automorphic representation $\Pi$ under consideration. Our main local result is that $\Pi_\infty$ is always defined over a number field, which in many cases agrees with the field of rationality of $\Pi_\infty$ (cf.\ Proposition \ref{prop:standardmodules}). A substantial part of our work follows our pursuit of proving optimal rationality results (cf.\ Theorem \ref{thm:realstandardmodules}).

In order to have a useful notion of rationality for $G(\RR)$-representations, we need to set up an appropriate theory of $(\lieg,K)$-modules over general fields of characteristic zero.

In the first four sections we lay the foundation for such a theory. After introducing the abstract notion of a {\em pair} in this setting we introduce the appropriate categories of modules for pairs. In that context we emphasize that the elementary Proposition \ref{prop:hombasechange} has fundamental consequences for all aspects of the theory. We set up the necessary homological machinery which allows us to define and study related rationality questions. Our first main result is the Homological Base Change Theorem (cf.\ Theorem \ref{thm:derivedbasechange}), which has many important consequences in the sequel.

A fundamental question in the theory of $(\lieg,K)$-modules over arbitrary fields $F/\QQ$ is which properties are {\em geometric}, in the sense that a property holds for a module $X$ over $F$ if and only if it holds over one resp.\ over all extensions $F'/F$. Section \ref{sec:geometric} is dedicated to the study of geometric properties. We show that admissibility and finite length are both geometric properties. As we are dealing with infinite-dimensional modules and possibly infinite field extensions these are non-trivial facts.

In order to construct models of modules we define rational Zuckerman functors and show that they commute with base change, and in particular are Galois equivariant. We also sketch a rational algebraic character theory, generalizing the algebraic characters defined by the author in \cite{januszewski2011}, which is Galois equivariant as well (cf.\ Theorem \ref{thm:characters} and Proposition \ref{prop:cqgallois}). This implies the existence of an equivariant global character theory for automorphic representations.

This setup yields rational models of Harish-Chandra modules of interest over $\overline{\QQ}$ resp.\ number fields where the corresponding $\theta$-stable parabolics and their characters are defined, since in our principal applications to automorphic representations the modules of interest are cohomologically induced standard modules $A_\lieq(\lambda)$.

To obtain optimal rationality results for $A_\lieq(\lambda)$, we depart from a standard module $A_\lieq(0)$ with trivial infinitesimal character occuring in the same coherent family as $A_\lieq(\lambda)$. Then $A_\lieq(0)$ is defined over an imaginary quadratic extension $\QQ_K'/\QQ_K$, and we descend it to $\QQ_K$ whenever this is possible. We refer to Theorem \ref{thm:realstandardmodules} for a non-exhaustive list of cases where we can guarantee models over $\QQ_K$. Finally, we apply rational translation functors to obtain suitable models of $A_\lieq(\lambda)$ for general $\lambda$ (supposed to be rational in a suitable sense).

This descent problem leads us to a natural generalization of the classical notion of Frobenius-Schur indicator \cite{frobeniusschur1906} to $(\lieg,K)$-modules over arbitrary fields. We pursue this in section \ref{sec:frobeniusschur} where we show that the general descent problem in quadratic extensions over number fields satisfies a local-global principle (cf.\ Theorem \ref{thm:localglobal}).

For infinitesimally unitary representations such as $A_\lieq(0)$, the obstruction for descent at the ramified archimedean places is controlled by the symmetry properties of invariant bilinear forms, as is the case classically. This is established in Theorem \ref{thm:fsunitary}. An important consequence is that the indicator for one ramified archimedean place determines the indicators at all archimedean places.

For elementary reasons the indicator for $A_\lieq(0)$ at a ramified archimedean prime is determined by the indicator of the bottom layer. More generally, we show in Proposition \ref{prop:minimalktypedescent} that this reduction always works, locally and globally, regardless of any unitarity assumptions. This then allows us to exploit properties of the chosen model $K$ of the maximal compact group to lay hand on the local obstructions for descent at non-archimedean places.

Section \ref{sec:admissiblegroups} identifies the relevant natural classes of groups for this purpose. We say that a model $K$ is $F'/F$-admissible for a quadratic extension $F'/F$ if the only obstruction to descent for arbitrary absolutely irreducibles is the field of rationality. Not all compact groups admit admissible models. Therefore, we also introduce the class of $F'/F$-semi-admissible models consisting of models $K$ where the global obstruction is controlled by a single archimedean place. Propositions \ref{prop:necessaryadmissiblegroups}, \ref{prop:unconditionaladmissiblegroups} and \ref{prop:unitarygroups} give necessary and sufficient criteria for the existence of admissible models.

For example we show that every model of $\U(n)$ is admissible, whereas $\SU(n)$ admits admissible models only when $n$ is odd. For even $n$ there are $\QQ(\sqrt{-1})/\QQ$-admissible models of $\SU(n)$. In general it seems to be a delicate question if semi-admissible models exist, because we usually want to realize $K$ as a subgroup of a given group $G$ over $\QQ$.

In order to define rational translation functors, we first investigate the Galois action on infinitesimal characters (cf.\ Propositions \ref{prop:centerrationality} and \ref{prop:hcgalois}). It is no surprise that in terms of Harish-Chandra's parametrization of characters by weights, this action specializes to the one studied by Borel and Tits in \cite{boreltits1965} in the finite-dimensional case. This then allows us to define rational translation functors accordingly. As their definition is straightforward, we only treat them implicitly in the proof of Proposition \ref{prop:standardmodules}.

As already indicated, we do not develop a full blown theory of $(\lieg,K)$-modules over any field here, as this would certainly yield to a monograph of size comparable to \cite{book_knappvogan1995}. Therefore we do not discuss rational Hecke algebras, and omit the treatment of Bernstein functors and hard duality. These topics are taken up in \cite{januszewski2015pre,januszewski2015pre2} in order to study the rationality properties of certain zeta integrals, which has applications to Deligne's Conjecture on special values of $L$-functions \cite{deligne1979}.

\addcontentsline{toc}{subsection}{Related work}
\subsection*{Related work}

The beginnings of a rational theory of Harish-Chandra modules had already been sketched in the author's manuscript \cite{januszewskipreprint}. Independently Michael Harris sketched a variant of Beilinson-Bernstein Localization over $\QQ$ in the context of Shimura varieties and hinted at applications to periods of automorphic representations in \cite{harris2012}. Harris emphasizes the Galois-equivariant viewpoint, i.e.\ he considered modules over $\overline{\QQ}$ together with a Galois action, and Beilinson-Bernstein Localization already puts strict finiteness conditions on the modules under consideration. Our results show that this is not a serious loss, since finite length modules satisfy those conditions automatically (cf.\ Theorem \ref{thm:finitelength}).

Harris obtains another proof of Theorem \ref{introthm:hermitianrationality} in \cite{harris2012}. However, we remark that the exposition in loc.\ cit.\ is very terse. For example, the existence of a rational Beilinson-Bernstein correspondence, which is the key ingredient to the proof of Theorem \ref{introthm:hermitianrationality} in loc.\ cit., is claimed but not proved. Furthermore, certain rationality statements in loc.\ cit.\ are too optimistic. We convinced ourselves that such a theory indeed exists and has the claimed properties. Meanwhile, Michael Harris is working on an erratum for loc.\ cit.


Further motivation for our work comes from related work of Michael Harris and his joint work with Harald Grobner, as well as work of G\"unter Harder and A.\ Raghuram. In their joint work \cite{grobnerharris2013}, Harris and Grobner observed a weaker variant of Theorem \ref{thm:glnrationality} for $\GL(n,F\otimes_\QQ\RR)$ for imaginary quadratic fields $F$, which is an essential ingredient in their work.

In their recent investigation of special values of Rankin-Selberg $L$-functions, Harder and Raghuram implicitly compared rational structures on Harish-Chandra modules in \cite{harderraghuram2014preprint}. In this context Harder observed in \cite{harder2014preprint}, independently from us, that cohomological modules for $\GL(n,\RR)$ are defined over $\QQ$, which is a special case of Theorem \ref{thm:glnrationality}. Harder even went further and gave models over $\ZZ$. He first gives an ad hoc construction of models for $\GL(2,\RR)$-modules and then invokes an explicit algebraic variant of parabolic induction in order to produce models for $\GL(n,\RR)$-modules without reference to ambient categories of modules.

Harder pursues rationality and integrality properties of intertwining operators, which yield applications in \cite{harderraghuram2014preprint}. In this context questions of integrality of inverses of intertwining operators boil down to combinatorial identities \cite{harder2010}, which have been taken up by Don Zagier in \cite{zagier2010} in a special but non-trivial case.


Frobenius-Schur indicators classifying invariant bilinear forms had been previously studied by Prasad and Ramakrishnan \cite{prasadramakrishnan2012} and by Adams \cite{adams2014}. Adams' results allow us to expand the list in Theorem \ref{thm:realstandardmodules}. In their theses, Adams' students Robert McLean and Ran Cui studied Frobenius-Schur indicators for self-dual irreducible Langlands quotients of the principal series and the real-quaternionic indicator respectively \cite{mcleanthesis,cuithesis,cuipreprint2016}. Their results fit into our framework and provide even more optimal cases along the line of Theorem \ref{thm:realstandardmodules}.

\addcontentsline{toc}{subsection}{Acknowledgements}
{\bf Acknowledgements.}
The author thanks Binyong Sun for his hospitality and fruitful discussions, Jacques Tilouine for pointing out that the field of definition of Harish-Chandra modules is related to the $L$-packet, and G\"unter Harder for sharing and explaining his work in \cite{harder2014preprint}. The author thanks Jeff Adams for sharing Robert McLean's and Ran Cui's work. The author thanks Claus-G\"unther Schmidt and Anton Deitmar for their comments and remarks on a preliminary version of this paper. Last but not least the author thanks the referee for helpful remarks and the suggestion to include the general reductive case.

\addcontentsline{toc}{section}{Notation and Conventions}
\section*{Notation and Conventions}

Throughout the paper all fields we consider are of characteristic $0$. The reader may well assume that all fields are contained in $\CC$ in light of an appropriate Lefschetz Principle (see also \cite{lepowsky1976}). We denote fields by $F$, $F'$, ..., linear algebraic groups by $B,G,K$, ..., their Lie algebras by the corresponding gothic letters $\lieb,\lieg,\liek$, ..., and by $G'$, resp.\ $\lieg'$ their base change in an extension $F'/F$. Reductive linear algebraic groups are not assumed to be connected, but their connected components are assumed to be geometrically connected. We follow the customary convention that a rational representation of an algebraic group $G$ is a homomorphism of algebraic groups $\rho:G\to \Aut(V)$ for a (finite-dimensional) vector space $V$. Such a representation is said to be $F$-rational if all these data are given over a field $F$. $U(\lieg)$ denotes the universal enveloping algebra of a Lie algebra $\lieg$ over $F$ and $Z(\lieg)$ its center. Note that $Z(\lieg)$ is defined over $F$, cf.\ Proposition \ref{prop:centerrationality}. For an algebraic group or a real Lie group $G$, the superscript $\cdot^0$ denotes the connected component of the identity, and $\pi_0(G)=G/G^0$.

The derived group of a group $G$ is denoted by ${\mathscr D}(G)$, and the adjoint group of a reductive group is $G^{\rm ad}$. We extend the notion of isogeny to compact Lie groups in the obvious way.

We say that a connected reductive group $G$ over a number field or non-archimedean local field $F$ has good reduction at a place $v$ of $F$ if there is a group scheme $\mathcal G$ over the valuation ring $\OO_v\subseteq F$ of $v$ which admits $G$ as generic fiber and whose special fiber is smooth and connected.

If $F$ is a number field, we write $\Adeles_F$ for the ring of ad\`eles of $F$. We set $\Adeles:=\Adeles_\QQ$ for notational simplicity. We say that an archimedean prime $v$ of $F$ ramifies in an extension $F'/F$ if the extension of the corresponding completions $F_v'/F_v$ is non-trivial (here $v$ always denotes the unique extension of $v$ to $F'$). We say that $F'/F$ ramifies at infinity if there at least one archimedean place of $F$ ramifies in $F'/F$.

If $F'/F$ is a quadratic extension of fields, we say that a quadratic extension $L'/L$ {\em dominates} $F'/F$, if $F\subseteq L$, $F'\subseteq L'$, $L'=F'L$ and $F=L\cap F'$.

\section{Rational Pairs and Modules}

Let $F$ be a field of characteristic $0$. A {\em pair} $(\liea,B)$ over $F$, or equivalently an $F$-rational pair $(\liea,B)$, consists of a reductive linear algebraic group $B$ over $F$ (not necessarily connected), and a Lie algebra $\liea$ over $F$, and the following additional data:
\begin{itemize}
\item[(i)] A Lie algebra monomorphism
$$
\iota_B:\quad \lieb:=\Lie(B)\;\to\;\liea.
$$
\item[(ii)] An action of $B$ on $\liea$, whose differential is the action of $\lieb$ on $\liea$.
\end{itemize}

We consider the category $\mathcal C_{\rm fd}(B)$ of finite-dimensional rational $B$-modules over $F$. This is an $F$-linear tensor category, with a natural forgetful functor
$$
\mathcal F^B:\quad\mathcal C_{\rm fd}(B)\to \mathcal C_{\rm fd}(1),
$$
which turns $(\mathcal C_{\rm fd}(B),\mathcal F^B)$ into a Tannakian category, with Tannakian dual canonically isomorphic to $B$.

We remark that by (i) and (ii), $\liea$ is a Lie algebra object in $\mathcal C_{\rm fd}(B)$, i.e.\ we have a commutative diagram
$$
\begin{CD}
\liea\otimes_F\liea @>[\cdot,\cdot]>>\liea\\
@A\Delta{}AA @AAA\\
\liea@>>> 0
\end{CD}
$$
in $\mathcal C_{\rm fd}(B)$, and a similar one reflecting the Jacobi identity.

The category of finite-dimensional $(\liea,B)$-modules is defined as the category $\mathcal C_{\rm fd}(\liea,B)$ of $\liea$-module objects in $\mathcal C_{\rm fd}(B)$ with the property that the action of $\lieb\subseteq\liea$ coincides with the derivative of the action of $B$. Considering the category $\mathcal C(B)$ of ind-objects in $\mathcal C_{\rm fd}(B)$, we define mutatis mutandis the category of $(\liea,B)$-modules $\mathcal C(\liea,B)$ as the category of $\liea$-module objects in $\mathcal C(B)$, again with the action of $\lieb\subseteq\liea$ being the derivative of the action of $B$. All these categories are equipped with natural forgetful faithful exact functors
$$
\mathcal C(\liea,B)\to\mathcal C(1)
$$
into the category of $F$-vector spaces.

We call an $(\liea,B)$-module $X$ {\em admissible}, if for each finite-dimensional $B$-module $V$
\begin{equation}
\dim_F\Hom_B(V,X)\;<\;\infty.
\label{eq:admissibledefinition}
\end{equation}

\subsection{Base change}

Let $F'/F$ be a field extension. Then every pair $(\liea,B)$ over $F$ gives rise to a pair
$$
(\liea',B')\;:=\;
(\liea,B)\otimes_F F'\;:=\;
(\liea\otimes_F F',B\times_{\Spec F} \Spec F')
$$
over $F'$. Similarly every $(\liea,B)$-module $X$ gives rise to an $(\liea',B')$-module $X\otimes_F F'$, and this construction extends to an exact faithful functor
$$
-\otimes_FF':\quad \mathcal C(\liea,B)\to\mathcal C(\liea',B'),
$$
which is left adjoint to the exact faithful forgetful functor
$$
\cdot|_F:\quad \mathcal C(\liea',B')\to\mathcal C(\liea,B).
$$
The latter also admits an exact right adjoint
$$
-\otimes_F^BF':\quad \mathcal C(\liea,B)\to\mathcal C(\liea',B'),\quad X\mapsto\Hom_{F}(F',X)_{B-{\rm finite}},
$$
where the $F'$-vector space structure of $\Hom_{F}(F',X)$ is inherited from the domain.
\begin{proposition}\label{prop:hombasechange}
For any $(\liea,B)$-modules $X,Y$ we have a canonical monomorphism
\begin{equation}
\Hom_{(\liea,B)}(X,Y)
\otimes_F F'\;\to\;
\Hom_{(\liea',B')}(X\otimes_F F',Y\otimes_F F').
\label{eq:hombasechangemono}
\end{equation}
which extends to a natural isomorphism
\begin{equation}
\Hom_{(\liea,B)}(X,Y)\otimes_{F}^1F'
\;\cong\;
\Hom_{(\liea',B')}(X\otimes_F F',Y\otimes_F^B F').
\label{eq:hombasechangeiso}
\end{equation}
The map \eqref{eq:hombasechangemono} is an isomorphism in the following cases:
\begin{enumerate}
  \item[(i)] $F'/F$ is finite,
  \item[(ii)] $X$ is finitely generated,
  \item[(iii)] $$\dim_{F'}\Hom_{(\liea',B')}(X\otimes_F F',Y\otimes_F F')=1.$$
\end{enumerate}
\end{proposition}

\begin{proof}
  The existence of the monomorphism \eqref{eq:hombasechangemono} is clear, and the existence of the canonical isomorphism \eqref{eq:hombasechangeiso} is a straighforward application of the adjointness relations. For $F'/F$ finite, the two functors $-\otimes_FF'$ and $-\otimes_F^BF'$ agree, this shows statement (i). The proof of statement (ii) is standard.


For (iii) we observe that under the given hypothesis, $\Hom_{(\liea',B')}(X\otimes_F F',Y\otimes_F^B F')\neq 0$, hence $\Hom_{(\liea,B)}(X,Y)\neq 0$ by \eqref{eq:hombasechangeiso}, which implies that \eqref{eq:hombasechangemono} must be an isomorphism.
\end{proof}

We will generalize Proposition \ref{prop:hombasechange} to arbitrary Ext groups in section \ref{sec:homology} (cf.\ Corollary \ref{cor:extbasechange}).

\subsection{Restriction of scalars}\label{sec:restrictionofscalars}

If $F'/F$ is finite, and if $(\liea',B')$ is a pair over $F'$, we may define the restriction of scalars
$$
(\liea'',B'')\;:=\;
\res_{F'/F} (\liea',B')\;:=\;
(\res_{F'/F}\liea',\res_{F'/F}B'),
$$
where on the right hand side $\res_{F'/F}$ denotes restriction of scalars \`a la Weil \cite{book_weil1961}. We have similarly a functor
$$
\res_{F'/F}:\quad \mathcal C(\liea',B')\to\mathcal C(\liea'',B''),
$$
naively given by sending an $(\liea',B')$-module $X'$ to $X'$ considered as an $F$-vector space $\res_{F'/F}X'$. We have the straightforward
\begin{proposition}
The functor $\res_{F'/F}$ is fully faithful.
\end{proposition}

\begin{proof}
It is well known that
$$
\res_{F'/F}:\quad \mathcal C_{\rm fd}(B')\to\mathcal C_{\rm fd}(B'')
$$
is fully faithful. This naturally extends to the categories of ind-objects, and as the image of $\liea'$ under $\res_{F'/F}$ is $\liea''$, the claim follows.
\end{proof}

In the sequel, depending on the context, we also consider $\res_{F'/F}$ also as a functor
$$
\res_{F'/F}:\quad \mathcal C_{\rm fd}(\liea',B')\to\mathcal C_{\rm fd}(\liea,B),
$$
i.e.\ we implicitly compose the restriction of scalars with the forgetful functor along the unit map $(\liea,B)\to(\liea'',B'')$ of the adjunction.


\subsection{Associated pairs}

We depart from an $F$-rational pair $(\liea,B)$. Let $\sigma:F\to\CC$ be an embedding and denote
$$
(\liea^\sigma,B^\sigma)\;:=\;(\liea,B)\otimes_{F,\sigma}\CC
$$
the corresponding base change. Then we consider $B^\sigma(\CC)$ as a real Lie group and fix a maximal compact subgroup $K^\sigma$. It is unique up to conjugation by an element of $B^\sigma(\CC)$. The inclusion
$$
K^\sigma\;\to\;B^\sigma(\CC)
$$
induces an equivalence
\begin{equation}
\mathcal C_{\rm fd}(B\otimes_{F,\sigma}\CC)\;\to\;
\mathcal C_{\rm fd}(K^\sigma)
\label{eq:BKequivalence}
\end{equation}
of the categories of finite-dimensional representations. Then $(\liea^\sigma,K^\sigma)$ constitutes a classical pair that we call {\em associated} to $(\liea,B)$ (with respect to $\sigma$).

For associated pairs we have
\begin{proposition}\label{prop:classicaliso}
The category $\mathcal C(\liea^\sigma,B^\sigma)$ of $\CC$-rational modules is naturally equivalent to the category $\mathcal C(\liea^\sigma,K^\sigma)$ of classical $(\liea^\sigma,K^\sigma)$-modules. This equivalence induces an equivalence of the corresponding categories of finite-dimensional modules.
\end{proposition}

\begin{proof}
First observe that in the case $\liea=\lieb$ both categories agree thanks to the equivalence \eqref{eq:BKequivalence}, and are given by the category of rational $B^\sigma$-modules, which is the same as locally finite $K^\sigma$-modules. The general case reduces to this case by the observation that the corresponding Lie algebra objects $\liea$ in both categories are mapped onto each other (up to isomorphism), whence the resulting categories of $(\liea,B)$- resp.\ $(\liea^\sigma,K^\sigma)$-module objects are equivalent.
\end{proof}

\subsection{Reductive pairs}\label{sec:reductivepairs}

A pair $(\lieg,K)$ over $F$ is {\em reductive} if $\lieg$ is reductive, $\liek$ is the space of fixed points under an $F$-linear involution $\theta$ of $\lieg$, and we are given a non-degenerate invariant bilinear form $\langle\cdot,\cdot\rangle:\lieg\to F$, where invariance is understood with respect to the adjoint actions of $\lieg$ and $K$ on $\lieg$.

On the categories of reductive pairs and their modules we have natural base change and restriction of scalars functors as before. Recall that restriction of scalars on pairs is only defined for finite extensions.

By the classification of reductive algebraic groups we know that each reductive pair $(\lieg,K)$ that corresponds to a linear reductive real Lie group $G$, has an $F$-rational model $(\lieg_F,K_F)$ over a {\em number field} $F\subseteq\CC$, i.e.\ $F/\QQ$ is finite. However in general such a model is not unique, and the resulting notion of rationality for modules depends on the choice of model.

We call a parabolic subalgebra $\lieq\subseteq\lieg$, defined over $F$, {\em $F$-germane} if it has a Levi decomposition $\liel+\lieu$ over $F$ with a $\theta$-stable Levi factor $\liel$. Write $\tilde{L}\cap K$ for the maximal subgroup of $K$ whose Lie algebra is contained in $\liel$. We assume it to be defined over $F$ and set
$$
L\cap K\;:=\;N_{\tilde{L}\cap K}(\lieq)\cap N_{\tilde{L}\cap K}(\theta\lieq).
$$
Then $(\liel,L\cap K)$ is another reductive pair over $F$. In the sequel we implicitly assume Levi factors of $F$-germane parabolic subalgebras to be always of this specific form, except when explicitly stated otherwise (in which case we pass to a subgroup of finite index of $L\cap K$).

\subsection{Galois actions}\label{sec:galoisactions}

Let $F'/F$ be a not necessarily finite Galois extension with Galois group $\Gal(F'/F)$. If $(\liea,B)$ is a pair over $F$ with base change $(\liea',B')$ to $F'$, $(\liec,D)\subseteq (\liea_{F'},B_{F'})$ a subpair over $F'$ and $\tau\in\Gal(F'/F)$, we have the Galois twists
$$
\liec^\tau\;:=\;\tau(\liec)\;\subseteq\;\liea'=\liea\otimes_F F'
$$
and similarly the subgroup
$$
D^\tau\;:=\;\tau(D)\;\subseteq\;B'\;=\;B\otimes_F F'
$$
Then we have the subpair
$$
(\liec,D)^\tau\;:=\;(\liec^\tau,D^\tau)\;\subseteq\;(\liea',B').
$$
We remark that mutatis mutandis $\tau$ acts on the universal enveloping algebra $U(\liea')$.

\begin{proposition}\label{prop:centerrationality}
Let $F'$ be an algebraic closure of $F$, $(\liea,B)$ a pair with base change $(\liea',B')$ to $F'$. Write $Z(\liea')$ for the center of the universal enveloping algebra $U(\liea')$. Then $Z(\liea')$ is defined over $F$, i.e.
$$
Z(\liea')\;=\;Z(\liea)\otimes_F F'.
$$
\end{proposition}

\begin{proof}
An element $a\in U(\liea')$ lies in $Z(\liea')$ if and only if one (hence all) its Galois twists $a^\tau$ lie in $Z(\liea')$ as well, $\tau\in\Gal(F'/F)$. Hence by Galois descent for vector subspaces of $U(\liea')$, $Z(\liea')$ is defined over $F$.
\end{proof}

Let $F'/F$ be a Galois extension as before and let $X$ be an $(\liea',B')$-module. Then for any $\tau\in\Gal(F'/F)$ we have on
$$
X^\tau\;:=\;X\otimes_{F',\tau} F'
$$
a natural $F'$-linear action of $\liea$, induced by the action of $\liea$ on $X$. This action extends uniquely to an action of
$$
\liea'=\liea\otimes_{F} F'.
$$
Similarly we have a unique action of $B'$ on $X^\tau$, extending the natural $F'$-linear action of $B$ on $X$. If $X^\tau$ is defined over $F$, this action coincides with the usual Galois-twisted action.

We remark that unless $X$ is defined over $F$, $X^\tau$ is only defined up to unique isomorphism. To be more precise, the twisted module $X^\tau$ comes with a natural $\sigma$-linear isomorphim
$$
\iota_\tau:\quad X\;\to\;X^\tau.
$$
Then $(X^\tau,\iota_\tau)$ is unique up to unique isomorphism.

In their investigation of rationality questions of rational representations of reductive groups, Borel and Tits introduced in \cite[Section 6]{boreltits1965} the following Galois action on weights. Let $\lieg$ be a reductive Lie algebra defined over $F$, $F'/F$ a Galois extension, over which the base change $\lieg'$ of $\lieg$ to $F'$ splits, and $\lieh\subseteq\lieg'$ a split Cartan subalgebra. Denote by $\Delta(\lieg',\lieh)$ the set of roots and fix a positive system $\Delta^+\subseteq\Delta(\lieg',\lieh)$, giving rise to a Borel subalgebra $\lieb=\lieh+\lien$ with nilpotent radical $\lien$. We write $W(\lieg,\lieh)$ for the corresponding Weyl group and $X(\lieh)=\lieh^*$ for the space of linear characters of $\lieh$, and $\rho\in X(\lieh)$ for the half sum of positive roots. Then for every $\sigma\in F'$ we find a unique inner automorphism $\alpha\in\Inn(\lieg)\subseteq\Aut(\lieg)$ sending $\lieh^\tau$ to $\lieh$ and $\lieb^\tau$ to $\lieb$. Then a weight $\lambda\in X(\lieh)$ is sent via $\tau$ to ${}_\Delta\tau(\lambda)$, which is characterized by the property
\begin{equation}
{}_\Delta\tau(\lambda)(\alpha(h^\tau))\;=\;\lambda(h)^\tau,\quad h\in\lieh.
\label{eq:deltagaloisaction}
\end{equation}
This action is well defined, i.e.\ independent of the choice of $\alpha$, and sends dominant weights to dominant weights. More concretely, if $\lambda\in X(\lieh)$ is the highest weight of an irreducible finite-dimensional $\lieg'$-module $V(\lambda)$, then $V(\lambda)^\tau$ is irreducible of highest weight ${}_\Delta\tau(\lambda)$. Observe that ${}_\Delta\tau(\rho)=\rho$.

We remark that for $h\in U(\lieh)^{W(\lieg',\lieh)}$, also
$$
\alpha(h^\tau)\;\in\;U(\lieh)^{W(\lieg',\lieh)},
$$
and this element is independent of the choice of $\alpha$. Indeed, if $\alpha'\in\Inn(\lieg')$ sends $\lieh^\tau$ to $\lieh$ and $\lien^\tau$ to $\lien$, then the element
$$
\alpha'\circ\alpha^{-1}\;\in\;N(\lieh)
$$
sends the positive system $\Delta^+$ to $\Delta^+$, thus lies in the centralizer of $\lieh$.

Recall that a $\lieg'$-moduile $V$ is called {\em quasi-simple} if the center $Z(\lieg')$ of the universal enveloping algebra acts on $V$ via scalars.

\begin{proposition}\label{prop:hcgalois}
If $V$ is a quasi-simple $\lieg'$-module on which $Z(\lieg')$ acts via the infinitesimal character $\lambda+\rho$, then $V^\tau$ has infinitesimal character
$$
{}_\Delta\tau(\lambda+\rho)\;=\;
{}_\Delta\tau(\lambda)+\rho.
$$
\end{proposition}

\begin{proof}
In light of Proposition \ref{prop:centerrationality}, the Galois action on infinitesimal characters
$$
\eta:\quad Z(\lieg')\;\to\;F',
$$
is given by
$$
\eta^\tau\;=\;\tau\circ \eta\circ\tau^{-1}.
$$
The same formula applies to the action of $\lieg$ on $V^\tau$. Recall the definition of the Harish-Chandra map
$$
\gamma:Z(\lieg')\to U(\lieh)^{W(\lieg',\lieh)}.
$$
It is given by the composition of the projection
$$
p_\lien:\quad
U(\lieg')=U(\lieh)\oplus(\lien^- U(\lieg')+U(\lieg')\lien)
\;\to\; U(\lieh)
$$
with the algebra map
$$
\rho_{\lien}:\quad U(\lieh)\to U(\lieh),
$$
which is induced by the map
$$
\lieh\to \lieh,\quad h\mapsto h-\rho(h)\cdot 1_{U(\lieh)}.
$$
Let $\tau\in\Gal(F'/F)$ and where $\alpha\in\Inn(\lieg)$ as before. Then
\begin{equation}
\alpha(\lien^\tau)\;=\;\lien,
\label{eq:lientau}
\end{equation}
and the same formula holds for $\lien^-$.

For any $g\in U(\lieg')$ with decomposition
$$
g\;=\;p_{\lien}(g)+r_{\lien}(g),\quad r_{\lien}(g)\in(\lien^- U(\lieg')+U(\lieg')\lien),
$$
we deduce the decomposition
$$
\alpha(g^\tau)\;=\;\alpha(p_{\lien}(g)^\tau)+
\alpha(r_{\lien}(g)^\tau).
$$
Hence, by \eqref{eq:lientau},
$$
p_{\lien}(\alpha(g^\tau))\;=\;\alpha(p_{\lien}(g)^\tau).
$$
Similarly we deduce from relation \eqref{eq:deltagaloisaction}, the relation
$$
\rho(\alpha(h^\tau))\;=\;{}_\Delta\tau(\rho)(\alpha(h^\tau))\;=\;\rho(h)^\tau,
$$
and
$$
\alpha(1_{U(\lieh)}^\tau)\;=\;1_{U(\alpha(\lieh^\tau))}\;=\;1_{U(\lieh)},
$$
that for any $h\in\lieh$,
$$
\rho_\lien(\alpha(h^\tau))\;=\;\alpha(h^\tau)-\rho(h)^\tau\cdot 1_{U(\lieh)}.
$$
In conclusion we obtain for any $z\in Z(\lieg)$,
\begin{equation}
\gamma(\alpha(z^\tau))\;=\;\alpha(p_{\lien}(z)^\tau)-\rho(h)^\tau\cdot 1_{U(\lieh)}.
\label{eq:gammaalphatau}
\end{equation}
Let $\eta:Z(\lieg')\to F'$ be the infinitesimal character parametrized by $\lambda\in\lieh^*\!/W(\lieg',\lieh)$. Applying $\lambda$ to the identity \eqref{eq:gammaalphatau} proves the claim.
\end{proof}

\section{Homological Base Change Theorems}\label{sec:homology}

Let $(\liea,B)$ be any $F$-rational pair. Since $B$ is reductive, the category $\mathcal C_{\rm fd}(B)$ of finite-dimensional rational representations of $B$ is semisimple, hence all objects in $\mathcal C_{\rm fd}(B)$ are injective and projective, and the same remains valid in the ind-category $\mathcal C(B)$.

The forgetful functor $\mathcal F_{B}^{\liea,B}$ sending $(\liea,B)$-modules to $B$-modules has a left adjoint
\begin{equation}
\ind_{B}^{\liea,B}:\;\;\;M\;\mapsto\; U(\liea)\otimes_{U(\lieb)}M
\label{eq:ind}
\end{equation}
and a right adjoint
\begin{equation}
\pro_{B}^{\liea,B}:\;\;\;M\;\mapsto\; \Hom_{\lieb}(U(\liea),M)_{B-\text{finite}}.
\label{eq:pro}
\end{equation}
Then $\ind_{B}^{\liea,B}$ sends projectives to projectives and $\pro_{B}^{\liea,B}$ sends injectives to injectives and both functors commute with base change. As all $B$-modules are injective and projective, we see that $\mathcal C(\liea,B)$ has enough injectives and enough projectives. Therefore standard methods from homological algebra apply.

\subsection{The General Homological Base Change Theorem}

We have the fundamental
\begin{theorem}[Homological Base Change]\label{thm:derivedbasechange}
Let $(\liea,B)$ and $(\liec,D)$ be two pairs over $F$, and let
$$
\mathcal F:\quad
\mathcal C(\liea,B)\to\mathcal C(\liec,D)
$$
be a left (resp.\ right) exact functor. Consider for a homomorphism of fields $\tau:F\to F'$ another left (resp.\ right) exact functor
$$
\mathcal F':\quad
\mathcal C(\liea\otimes_{F,\tau} F',B\otimes_{F,\tau} F')\to\mathcal C(\liec\otimes_{F,\tau} F',D\otimes_{F,\tau} F')
$$
which extends $\mathcal F$, i.e.\ there is a natural isomorphism
$$
\iota:\quad
\mathcal F(-)\otimes_{F,\tau} F'\;\to\;\mathcal F'\circ(-\otimes_{F,\tau} F').
$$
Then in the right exact case, the natural isomorphism $\iota$ always extends naturally to a natural isomorphism
$$
(L_q\mathcal F)(-)\otimes_{F,\tau} F'\;\to\;(L_q\mathcal F')\circ(-\otimes_{F,\tau} F')
$$
for all $q$, which are compatible with the associated long exact sequences. In the left exact case, if $F'/F$ is finite, or if $X$ admits an admissible standard resolution, then we have a natural isomorphism
$$
(R^q\mathcal F)(X)\otimes_{F,\tau} F'\;\to\;(R^q\mathcal F')\circ(X\otimes_{F,\tau} F')
$$
in all degrees.
\end{theorem}

\begin{proof}
  The right (resp.\ left) derived functors of $\mathcal F$ and $\mathcal F'$ may be computed via resolutions computed inductively with standard injectives (resp.\ projectives). Since the contruction of resolutions by standard projectives via the functors \eqref{eq:ind} and \eqref{eq:pro} always commutes with base change to $F'$ along $\tau$, the first claim follows. The construction of standard injectives commutes with base change along $\tau$ whenever $F'/F$ is finite or $X$ is admissible, which shows the second claim.
\end{proof}

\begin{corollary}\label{cor:extbasechange}
  Assume that for two $(\liea,B)$-modules $X$ and $Y$ one of the following conditions is satisfied:
  \begin{itemize}
    \item[(i)] $F'/F$ is finite,
    \item[(ii)] $X$ is finitely generated and $Y$ admits an admissible injective standard resolution.
  \end{itemize}
Then, writing $(\liea',B')$ for the base change of $(\liea,B)$ in $F'/F$, we have in every degree $q$ natural isomorphisms
$$
\Ext^q_{\liea,B}(X,Y)\otimes_F F'\;\cong\;
\Ext^q_{\liea',B'}(X\otimes_F F',Y\otimes_F F').
$$
In particular, a short exact sequence
$$
\begin{CD}
0@>>> X @>>> Z @>>> Y @>>> 0
\end{CD}
$$
of $(\liea,B)$-modules splits over $F$ if and only if the short exact sequence
$$
\begin{CD}
0@>>> X\otimes_F F' @>>> Z\otimes_F F' @>>> Y\otimes_F F' @>>> 0
\end{CD}
$$
of $(\liea',B')$-modules splits over $F'$.
\end{corollary}

\begin{proof}
  Strictly speaking Corollary \ref{cor:extbasechange} is not an immediate consequence of Theorem \ref{thm:derivedbasechange}, since we are dealing with a bifunctor here. However, to prove the claim, we may compute $\Ext_{\liea,B}^q(-,-)$ as the left derived functor of $\Hom_{\liea,B}(-,-)$ in the second argument. Using a standard resolution, which commutes by hypothesis (i) or (ii) with base change, and recalling that $\Hom_{\liea,B}(-,-)$ commutes with base change by Proposition \ref{prop:hombasechange} under the given choices of hypotheses as well, the claim follows. By the universality of $\delta$-functors we eventually obtain binatural isomorphisms. Those isomorphisms are independent of the previous choices.

The last statement of the Corollary follows from Baer's classical interpretation of the $\Ext^1$ group as the group of extensions.
\end{proof}

\subsection{Equivariant homology and cohomology}\label{sec:homologyandcohomology}

As before let $(\liea,B)$ be a pair over $F$. Assume that $(\liec,D)$ is a subpair which is normalized by another subpair $(\widetilde{\liea},\tilde{B})$ of $(\liea,B)$ with the property that
\begin{equation}
\liea\;=\;\widetilde{\liea}+\liec.
\label{eq:algsum}
\end{equation}

We obtain a functor
$$
H^0(\liec,D;-):\quad\mathcal C(\liea,B)\to \mathcal C(\widetilde{\liea},\tilde{B}),
$$
$$
X\quad\mapsto\quad
X^{\liec,D},
$$
sending a module to its $(\liec,D)$-invariant subspace, on which $(\widetilde{\liea},\tilde{B})$ acts naturally.

$H^0(\liec,D;-)$ is left exact and the higher right derived functors
$$
H^q(\liec,D;-):=R^qH^0(\liec,D;-):\;\;\;\mathcal C(\liea,B)\to \mathcal C(\widetilde{\liea},\tilde{B})
$$
are the {\em $F$-rational $\liec,D$-cohomology} and may, thanks to \eqref{eq:algsum}, be computed via the usual standard complex
\begin{equation}
\Hom_{D}(\wedge^{\bullet}\liec/\lied,X)
\label{eq:standardcohomologycomplex}
\end{equation}
of $(\widetilde{\liea},\tilde{B})$-modules.

Dually we may define and explicitly compute {\em $F$-rational $\liec,D$-homology} as the left derived functors of the coinvariant functor
$$
H_0(\liec,D;-):\quad\mathcal C(\liea,B)\to \mathcal C(\widetilde{\liea},\tilde{B}),
$$
$$
X\quad\mapsto\quad
X_{\liec,D}.
$$
Again it may be computed via the usual standard complex
\begin{equation}
\left(\wedge^{\bullet}(\liec/\lied)\otimes_F X\right)^D
\label{eq:standardhomologycomplex}
\end{equation}
of $(\widetilde{\liea},\tilde{B})$-modules.

These homology and cohomology theories satisfy the usual Poincar\'e duality relations, K\"unneth formalism, and give rise to $F$-rational Hochschild-Serre spectral sequences \cite{hochschildserre1953}.

\begin{proposition}\label{prop:equicohomology}
For any $F$-rational $(\liea,B)$-module $M$ the cohmology $H^q(\liec,D; X)$ is $F$-rational and for any map $\tau:F\to F'$ of fields we have a natural isomorphism
$$
H^q(\liec,D;X)\otimes_{F,\tau} F'\;\to\;
H^q(\liec\otimes_{F,\tau}F',D\otimes_{F,\tau}F';X\otimes_{F,\tau} F')
$$
of $(\liea,B)\otimes_{F,\tau} F'$-modules. The same statement is true for $\liec,D$-homology, and the duality maps and Hochschild-Serre spectral sequences respect the rational structure.
\end{proposition}

\begin{proof}
This is obvious from the $F$-resp.\ $F'$-rational standard complexes computing cohomology and homology, and also a consequence of Theorem \ref{thm:derivedbasechange}
\end{proof}

\subsection{Equivariant cohomological induction}\label{sec:equizuckerman}

To define $F$-rational cohomological induction, we adapt Zuckerman's original construction as in \cite[Chaper 6]{book_vogan1981}. Assume we are given an $F$-rational reductive subgroup $C\subseteq B$. We start with an $(\liea,C)$-module $M$ and set
$$
\tilde{\Gamma}_0(M)\;:=\;
\;\{m\in M\mid\dim_F U(\lieb)\cdot m\;<\;\infty\}.
$$
and
$$
\Gamma_0(M)\;:=\;
\;\{m\in \tilde{\Gamma}_0(M)\mid \text{the $\lieb$-representation}\; U(\lieb)\cdot m\;\text{lifts to}\;B^0\}.
$$
As in the analytic case the obstruction for a lift to exist is the (algebraic) fundamental group of $B^0$. In particular there is no rationality obstruction, as a representation $N$ of $\lieb$ lifts to $B^0$ if and only if it does so after base change to one (and hence any) extension of $F$. It is easy to see that $\Gamma_0(M)$ is an $(\liea,B^0)$-module.

We define the space of {\em $B^0$-finite vectors} in $M$ as
$$
\Gamma_1(M)\;:=\;
$$
$$
\;\{m\in \Gamma_0(M)\mid \text{the actions of $C$ (on $M$) and $C\cap B^0\subseteq B^0$  agree on $m$}\}.
$$
This is an $(\liea,C\cdot B^0)$-module.
Finally the space of {\em $B$-finite vectors} in $M$ is
$$
\Gamma(M)\;:=\;
\;\pro_{\liea,C\cdot B^0}^{\liea,B}(\Gamma_1(M)).
$$
Remark that this is not a subspace of $M$ in general. By Frobenius reciprocity it comes with a natural map $\Gamma(M)\to M$.

The functor $\Gamma$ is a right adjoint to the forgetful functor along
$$
i:\quad(\liea,C)\;\to\;(\liea,B)
$$
and hence sends injectives to injectives. We obtain the higher Zuckerman functors as the right derived functors
$$
\Gamma^q\;:=\;R^q\Gamma:\;\;\;
\mathcal C(\liea,C)\to \mathcal C(\liea,B).
$$
As in the classical case we can show
\begin{proposition}\label{prop:find}
In every degree $q$ we have a commutative square
$$
\begin{CD}
\mathcal C(\liea,C)@>\Gamma^q>>\mathcal C(\liea,B)\\
@V\mathcal F_{\lieb,C}^{\liea,C}VV
@VV\mathcal F_{\lieb,B}^{\liea,B}V\\
\mathcal C(\lieb,C)@>\Gamma^q>>\mathcal C(\lieb,B)\\
\end{CD}
$$
\end{proposition}

Strictly speaking the commutativity only holds up to natural isomorphism. The natural isomorphisms are unique as it turns out that $\Gamma(-)$ is in both cases the right adjoint of the classical forgetful functor along the corresponding inclusion of pairs. We will not go into this as the commutativity may also easily be deduced via base change from the classical setting. However, for the sake of readability, we decided to ignore such higher categorical aspects in the sequel.

\begin{proof}
For $q=0$ the commutativity is obvious from the explicit construction of the functor $\Gamma$. For $q>0$ this follows from the standard argument that the forgetful functors have an exact left adjoint given by induction along the Lie algebras and hence carry injectives to injectives. Furthermore they are exact, which means that the Grothendieck spectral sequences for the two compositions both degenerate. Therefore the edge morphisms of said spectral sequence yield the commutativity, this proves the claim.
\end{proof}

\begin{theorem}\label{thm:equizuckerman}
  Let $M$ be an $F$-rational $(\liea,C)$-module and $\tau:F\to F'$ a map of fields. Assume that either $F'/F$ is finite or that $M$ admits an admissible injective standard resolution. Then in every degree $q$ we have a natural isomorphism
$$
\Gamma^q(X)\otimes_{F,\tau} F'\;\to\;
\Gamma^q(X\otimes_{F,\tau} F')
$$
of $(\liea,B)\otimes_{F,\tau} F'$-modules. Furthermore these isomorphisms are functorial in $\liea,B,C$, and respect the long exact sequences associated to $\Gamma$.
\end{theorem}

\begin{proof}
By the Homological Base Change Theorem \ref{thm:derivedbasechange}, it suffices to show that the rational Zuckerman functor $\Gamma$ commutes with base change. This follows easily from the above construction of $\Gamma$.
\end{proof}

\begin{remark}
  By Theorem \ref{thm:equizuckerman} the functors $\Gamma^q$ satisfy the usual properties, i.e.\ they vanish for $q > \dim_F\lieb/\liec$, we have a Hochschild-Serre spectral sequence for $B$-types, for parabolic cohomological induction the effect on infinitesimal characters is the same as in the classical setting, etc.
\end{remark}

\section{Geometric Properties of $(\lieg,K)$-modules}\label{sec:geometric}

A natural question to ask is which properties of Harish-Chandra modules are {\em geometric}, in the sense that the property of an $(\liea,B)$-module $M$ holds over $F$, if and only if it holds over one (and hence any) extension $F'$ of $F$. We will see in this section that many classical properties, i.e.\ admissibility, $Z(\lieg)$-finiteness, and finite length are geometric properties. In order to control finite length in extensions we need to appeal to Quillen's generalization of Dixmier's variant of Schur's Lemma.

\subsection{Schur's Lemma and fields of definition}

Let again $F'/F$ be an extension and the pair $(\liea',B')$ be the extension of scalars to $F'$ of a pair $(\liea,B)$ over $F$. We say that an $(\liea',B')$-module $X'$ over $F'$ is {\em defined over $F$}, if there is an $(\liea,B)$-module $X$ satisfying
\begin{equation}
F'\otimes_{F}X\;\cong\;X'
\label{eq:model}
\end{equation}

We say that the pair $(\liea,B)$ {\em satisfies condition (Q)}, if we find a finite subgroup $B_0\subseteq B$ such that the map
\begin{equation}
B_0\;\to\;\pi_0(B\otimes_F\bar{F})
\label{eq:conditionQ}
\end{equation}
is surjective, where $\bar{F}$ denotes an algebraic closure of $F$.

For example condition (Q) is satisfied if $(\liea,B)$ is a reductive pair coming from a connected reductive linear algebraic group $G$.

For later use we first remark the following generalization of Schur's Lemma.
\begin{proposition}\label{prop:schur}
If $(\liea,B)$ satisfies condition (Q) and if $X$ is an irreducible $(\liea,B)$-module, then $\End_{\liea,B}(X)$ is an algebraic division algebra over $F$.
\end{proposition}

An algebraic division algebra over $F$ is a division algebra $A$ over $F$, all of whose elements are algebraic over $F$, i.e.\ for each $a\in A$, $F(a)/F$ is a (necessarily finite) algebraic extension.

\begin{proof}
It suffices to remark that an irreducible $(\liea,B)$-module $X$ remains irreducible after replacing $B$ by an appropriate finite subgroup $B_0$ satisfying \eqref{eq:conditionQ}. Therefore $X$ is an irreducible module over the convolution algebra $U(\liea)*B_0$. Quillen's result in \cite{quillen1969} applies to this case and proves the claim.
\end{proof}

\begin{corollary}\label{cor:finitecenter}
Let $(\lieg,K)$ be a reductive pair over $F$ satisfying condition (Q). If $X$ is an irreducible $(\lieg,K)$-module, then $Z(\lieg)$ acts on $X$ via a finite-dimensional quotient. In particular $X$ is $Z(\lieg)$-finite, and so is every $(\lieg,K)$-module of finite length.
\end{corollary}

\begin{proof}
Fix an extension $F'/F$ over which $\lieg$ splits. Over this extension the center $Z(\lieg)\otimes_FF'$ of $U(\lieg\otimes_F F')$ (cf.\ Proposition \ref{prop:centerrationality}) is noetherian, hence $Z(\lieg)$ is noetherian. Therefore its image in $\End_{\lieg,K}(X)$ is finitely generated, and thus finite-dimensional by Proposition \ref{prop:schur}. Therefore $Z(\lieg)$ acts on $X$ via a finite-dimensional quotient and the corollary follows.
\end{proof}

\begin{definition}
We call an $(\liea,B)$-module $X$ absolutely irreducible if it is irreducible over {\em every} extension $F'/F$.
\end{definition}

\begin{corollary}\label{cor:schur}
If $X$ is an absolutely irreducible $(\liea,B)$-module over $F$, $(\liea,B)$ satisfying condition (Q), then $\End_{\liea,B}(X)=F$.
\end{corollary}

\begin{proof}
If $F$ is algebraically closed or equal to $\CC$, we have
$$
\dim_F\End_{\liea,B}(X)\;=\;1
$$
by Proposition \ref{prop:schur}. Therefore Proposition \ref{prop:hombasechange} implies the claim for arbitrary $F$.
\end{proof}

In general, for a Galois extension $F'/F$, models of $(\liea',B')$-modules $X'$ over $F$ need not be unique. The existence of non-isomorphic models is equivalent to the existence of non-trival $1$-cocycles of $\Gal(F'/F)$ with coefficients in the group $\Aut_{(\liea',B')}(X')$. At least for absolutely irreducible modules satisfying Schur's Lemma, Hilbert's Satz 90 guarantees the uniqueness.

\begin{proposition}\label{prop:uniquemodels}
Let $F'/F$ be a Galois extension $F$, assume that $X'$ is an $(\liea',B')$-module satisfying
\begin{equation}
\End_{\liea',B'}(X')\;=\;F',
\label{eq:schurcondition}
\end{equation}
and $X$ is a model of $X'$ over $F$. Then $X$ is unique up to isomorphism.
\end{proposition}

\begin{proof}
The proof proceeds as in \cite[p. 741]{boreltits1965}.
\end{proof}

\subsection{Geometric properties}

Admissibility is a {\em geometric} property in the following sense.
\begin{proposition}\label{prop:admissible}
Let $X$ be an $(\liea,B)$-module. Then the following are equivalent:
\begin{itemize}
\item[(i)] $X$ is an admissible $(\liea,B)$-module.
\item[(ii)] For some extension $F'/F$, $X\otimes_F F'$ is an admissible $(\liea,B)\otimes_F F'$-module.
\item[(iii)] For every extension $F'/F$, $X\otimes_F F'$ is an admissible $(\liea,B)\otimes_F F'$-module.
\end{itemize}
\end{proposition}

\begin{proof}
Let us show that (i) implies (iii). Let $F'/F$ be any extension. Then by the theory of reductive algebraic groups we know that there is a finite subextension $F''/F$ with the property that the scalar extension functor $-\otimes_{F''}F'$ induces a faithful essentially surjective functor
$$
\mathcal C_{\rm fd}(B'')\;\to\;
\mathcal C_{\rm fd}(B'),
$$
i.e.\ all finite-dimensional representations of $B$ which are defined over $F'$ are already defined over $F''$. This observation naturally extends to the ind-categories
$$
\mathcal C(B'')\;\to\;
\mathcal C(B').
$$
In light of Proposition \ref{prop:hombasechange} this reduces our considerations to the cases where $F'/F$ is finite-dimensional.

Assume (i), and let $V$ be any finite-dimensional $B'$-module over a finite extension $F'/F$. Then
$$
\Hom_{B'}(V,X\otimes_F F')\;=\;
\Hom_{B}(\res_{F'/F}V,X).
$$
Since $\res_{F'/F}V$ is finite-dimensional over $F$, (iii) follows.

Since (iii) implies (ii), we are left to show that (ii) implies (i). So assume that $X\otimes_F F'$ is an admissible $(\liea',B')$-module for an extension $F'/F$. Let $V$ be a finite-dimensional $B$-module. Then by Proposition \ref{prop:hombasechange} we have
$$
\dim_F(V,X)\;=\;\dim_{F'}(V\otimes_FF',X\otimes_F F')\;<\;\infty.
$$
This concludes the proof.
\end{proof}

\begin{theorem}\label{thm:finitelength}
Let $(\lieg,K)$ be a reductive pair over $F$ satisfying condition (Q), and let $X$ be an $F$-rational $(\lieg,K)$-module. Then the following are equivalent:
\begin{itemize}
\item[(i)] $X$ is of finite length.
\item[(ii)] $X$ is $Z(\lieg)$-finite and admissible.
\end{itemize}
\end{theorem}

\begin{proof}
That (ii) implies (i) reduces to the case $F=\CC$ by Proposition \ref{prop:admissible}, since $Z(\lieg)$-finiteness is a geometric property as well. By Proposition \ref{prop:classicaliso} the case $F=\CC$ follows from the classical case, where the statement is well known.

The implication (i) to (ii) reduces to the classical case as follows. Let $X$ be an irreducible $(\lieg,K)$-module over $F$. Then $X$ is $Z(\lieg)$-finite by Corollary \ref{cor:finitecenter}, and also finitely generated. The latter two properties are stable under base change, and to prove the admissibility of $X$, we may by Proposition \ref{prop:admissible} assume that $F=\CC$, in which case, by Proposition \ref{prop:classicaliso}, the result is well known (cf.\ \cite{lepowsky1973}).
\end{proof}

\begin{corollary}\label{cor:finitelength}
Let $X$ be an $F$-rational $(\lieg,K)$-module for a reductive pair $(\lieg,K)$ satisfying condition (Q). Then the following are equivalent:
\begin{itemize}
\item[(i)] $X$ is of finite length.
\item[(ii)] For some extension $F'/F$, $X\otimes_F F'$ is of finite length.
\item[(iii)] For every extension $F'/F$, $X\otimes_F F'$ is of finite length.
\end{itemize}
\end{corollary}

\begin{proof}
By Theorem \ref{thm:finitelength} and Proposition \ref{prop:admissible} it suffices to observe that $Z(\lieg)$-finiteness is a geometric property as well, which is obvious.
\end{proof}

\subsection{$\lieu$-cohomology and constructible parabolic subalgebras}

\begin{proposition}\label{prop:lieuprop}
If $\lieq$ is a $\theta$-stable $F$-germane parabolic subalgebra of a reductive $F$-rational pair $(\lieg,K)$, with Levi decomposition $\lieq=\liel+\lieu$, then for all degrees $q$, the functors
$$
H^q(\lieu; -)\quad\text{and}\quad
H_q(\lieu; -),
$$
preserve admissibility, $Z(\lieg)$-finiteness and if $(\lieg,K)$ satisfies condition (Q), then also finite length.
\end{proposition}

\begin{proof}
The preservation of $Z(\lieg)$-finiteness is proven as in the classical case, which holds in fact for any parabolic subalgebra. The preservation of admissibility also follows mutatis mutandis as in the classical case. With Theorem \ref{thm:finitelength} we conclude that $\lieu$-(co)homology preserves finite length.
\end{proof}

Assume that the field $F$ has a real place, i.e.\ we have
$$
\Hom(F,\RR)\;\neq\;0,
$$
and that $(\lieg,K)$ gives rise to a classical reductive pair $(\lieg_\RR,K_\RR)$ after extension of scalars along an embedding $\iota_\RR:F\to\RR$. Let $F'/F$ be an extension, and $(\lieg',K')$ the usual extension of scalars in the extension $F'/F$. We call an $F'$-germane parabolic subalgebra $\lieq'\subseteq\lieg'$ {\em $F$-constructible}, if there exists a sequence of $F'$-germane parabolic subalgebras
\begin{equation}
\lieq'=\lieq_0'\subseteq\lieq_1'\subseteq\cdots\subseteq\lieq_l'=\lieg'
\label{eq:constructiblesequence}
\end{equation}
with the following property: Inside the Levi pair $(\liel_{i+1}',L_{i+1}'\cap K')$ of the associated parabolic $\lieq_{i+1}'$ the parabolic $\lieq_i'\cap\liel_{i+1}'$ is $\theta$-stable or defined over $F$. This notion generalizes the notion of constructibility introduced in \cite{januszewski2011}. In the case $F=\RR$ and $F'=\CC$ the two notions coincide.

The motivation for this notion is
\begin{proposition}\label{prop:constructiblehomology}
Assume that $(\lieg,K)$ is a reductive pair over a field $F\subseteq\RR$ satisfying condition (Q) and that $\lieq'\subseteq\lieg'$ is $F'$-constructible with Levi decomposition $\lieq'=\liel'+\lieu'$. Then for all degrees $q$, the functors
$$
H^q(\lieu'; -)\quad\text{and}\quad
H_q(\lieu'; -),
$$
preserve finite length.
\end{proposition}

\begin{proof}
Assume we are given a sequence as in \eqref{eq:constructiblesequence} satisfying the defining property of $F'$-constructibility. By the Hochschild-Serre spectral sequence for $\lieu'$-(co)homology it is enough to show that finite length is preserved in the case where $\lieq'$ is $\theta$-stable or defined over $F$. In the $\theta$-stable case the claim follows from Proposition \ref{prop:lieuprop}. The other case reduces by Corollary \ref{cor:finitelength} to the case of a real parabolic $\lieq'$ over $F=\RR$ after base change along $\iota_\RR$, where the statemement is well known.
\end{proof}

\subsection{Restrictions of irreducibles}

\begin{proposition}\label{prop:galoisrestriction}
Assume $F'/F$ to be a finite Galois extension. Let $X'$ be an irreducible $(\liea',B')$-module, then as an $(\liea,B)$-module, $\res_{F'/F}X'$ decomposes into a finite direct sum of irreducible $(\liea,B)$-modules. The number of summands is bounded by the degree $[F':F]$.
\end{proposition}

\begin{proof}
We write $X$ for the $(\liea,B)$-module $\res_{F'/F}X'$. Consider the $(\liea',B')$-module $X'':=X\otimes_F F'$. The identity
$$
F'\otimes_F F'\;=\;
\bigoplus_{\sigma\in\Gal(F'/F)} F'\otimes_{F',\sigma} F'
$$
shows that $X''$ decomposes into a finite direct sum
\begin{equation}
X''\;=\;
X'\otimes_{F'}F'\otimes_F F'\;=\;
\bigoplus_{\sigma\in\Gal(F'/F)} X'\otimes_{F',\sigma} F',
\label{eq:Xdoubleprimesum}
\end{equation}
of irreducible $(\liea',B')$-modules
$$
F_\sigma'\;:=\;X'\otimes_{F',\sigma} F'.
$$
Therefore, by Corollary \ref{cor:extbasechange}, each short exact sequence
$$
\begin{CD}
0@>>> Y @>>> X @>>> Z @>>> 0
\end{CD}
$$
of $(\liea,B)$-modules splits, as it must split over $F'$ by \eqref{eq:Xdoubleprimesum}, and we have isomorphisms
$$
Y\otimes_F F'\;=\;\bigoplus_{\sigma\in\Xi} F_\sigma'
$$
and
$$
Z\otimes_F F'\;=\;\bigoplus_{\sigma\not\in\Xi} F_\sigma'
$$
for some unique subset $\Xi\subseteq\Gal(F'/F)$. Inductively we conclude that $X$ decomposes into a finite sum of irreducible $(\liea,B)$-modules as claimed.
\end{proof}

\begin{corollary}\label{cor:restriction}
Assume $F'/F$ to be a finite extension. Let $X'$ be an absolutely irreducible $(\liea',B')$-module. Then, as an $(\liea,B)$-module, $\res_{F'/F}X'$ decomposes into a finite direct sum of irreducible $(\liea,B)$-modules.
\end{corollary}

\begin{proof}
Since $X'$ is absolutely irreducible, we may replace $F'$ without loss of generality by its normal hull over $F$. Then the claim follows from Proposition \ref{prop:galoisrestriction}.
\end{proof}

\section{Algebraic Characters}\label{sec:algebraiccharacters}

In the context of the theory developed here, it is possible to construct an abstract theory of algebraic characters over any field $F$ of characteristic $0$ following the axiomatic treatment given in \cite{januszewski2011}. Our results about geometric properties of modules may be used to produce non-trivial instances of this theory. To give a concrete example, we sketch the case of finite length modules here. The case of discretely decomposables discussed in \cite{januszewski2011} generalizes to arbitrary base fields of characteristic $0$ along the same lines as well.

\subsection{Rational algebraic characters}

We depart from a reductive pair $(\lieg,K)$ over a field $F$ and fix an $F'$-constructible parabolic subalgebra $\lieq'\subseteq\lieg'$ over an extension $F'/F$ with Levi decomposition $\lieq'=\liel'+\lieu'$. By Proposition \ref{prop:constructiblehomology} we have for each $q\in\ZZ$ a well defined functor
$$
H^q(\lieu',-):\quad
\mathcal C_{\rm fl}(\lieg',K')\to
\mathcal C_{\rm fl}(\liel',L'\cap K')
$$
on the corresponding categories of finite length modules over $F'$. We remark that by Corollary \ref{cor:finitelength} the categories of finite length modules are essentially small. Hence, by the long exact sequence of cohomology, the Euler characteristic of these functors gives rise to a group homomorphism
$$
H_{\lieq'}:\quad
K_{\rm fl}(\lieg',K')\to
K_{\rm fl}(\liel',L'\cap K'),
$$
$$
[X]\;\mapsto\;\sum_{q\in\ZZ}(-1)^q[H^q(\lieu',X)]
$$
of the corresponding Grothendieck groups of the abelian categories $\mathcal C_{\rm fl}(\lieg',K')$ resp.\ $\mathcal C_{\rm fl}(\liel',L'\cap K')$. Here the bracket $[\cdot]$ denotes the class associated to a module. We define the {\em Weyl denominator relative to $\lieq'$} as
$$
W_{\lieq'}\;:=\;H_{\lieq'}({\bf 1}_{\lieg',K'}),
$$
where ${\bf 1}_{\lieg',K'}$ denotes the trivial $(\lieg',K')$-module.

We know that the category $\mathcal C_{\rm fl}(\liel',L'\cap K')$ of finite length modules is closed under tensor products with finite-dimensional representations, again by Corollary \ref{cor:finitelength}, as this is well known over $\CC$, cf. \cite{kostant1975}. In particular $K_{\rm fl}(\liel',L'\cap K')$ is naturally a module over the commutative ring $K_{\rm fd}(\liel',L'\cap K')$ of finite-dimensional modules, the (scalar)multiplication stemming from the tensor product. The relative Weyl denominator $W_{\lieq'}$ lies in the latter ring and we may therefore consider the module-theoretic localization
$$
C_{\lieq'}(\liel',L'\cap K')\;:=\;K_{\rm fl}(\liel',L'\cap K')[W_{\lieq'}^{-1}].
$$
The $\lieq'$-character map is by definition the map
$$
c_\lieq:\quad
K_{\rm fl}(\lieg',K')\to C_{\lieq'}(\liel',L'\cap K'),
$$
$$
[X]\;\mapsto\;\frac{H_{\lieq'}(X)}{W_{\lieq'}}.
$$
The generalization of Theorem 1.4 of loc.\ cit.\ in our context is
\begin{theorem}\label{thm:characters}
The map $c_{\lieq'}$ has the following properties:
\begin{itemize}
\item[(a)] The map $c_{\lieq'}$ is {\em additive}, i.e.\ for all $X,Y\in K_{\rm fl}(\lieg',K')$ we have
$$
c_{\lieq'}(X+Y)\;=\;
c_{\lieq'}(X)+c_{\lieq'}(Y).
$$
\item[(b)] The map $c_{\lieq'}$ is {\em multiplicative}, i.e.\ for all $X\in K_{\rm fl}(\lieg',K')$ and $Y\in K_{\rm fd}(\lieg',K')$ we have
$$
c_{\lieq'}(X\cdot Y)\;=\;
c_{\lieq'}(X)\cdot c_{\lieq'}(Y),
$$
and
$$
c_{\lieq'}({\bf1}_{\lieg',K'})\;=\;{\bf1}_{\liel,L\cap K'}.
$$
\item[(c)] If $\lieq$ is $\theta$-stable then $c_{\lieq'}$ {\em respects admissible duals}, i.e.\ for all $X\in K_{\rm fl}(\lieg',K')$ we have
$$
c_{\lieq'}(X^\vee)\;=\;
c_{\lieq'}(X)^\vee,
$$
where $X^\vee$ denotes the $K$-finite dual.
\item[(d)] If for $X\in K_{\rm fl}(\lieg',K')$ its restriction lies in $K_{\rm fl}(\liel',L'\cap K')$, then we have the formal identity
$$
c_{\lieq'}(X)\;=\;[X]
$$
in $C_{\lieq'}(\liel',L'\cap K')$.
\end{itemize}
\end{theorem}

\begin{proof}
The additivity is clear. The multiplicativity is proven mutatis mutandis as in Theorem 1.4 in \cite{januszewski2011}. This also applies to (d).

To prove (c), we recall the notion of an $(\lieu'$-admissible) pair of categories having duality. In our situation we consider the pair of categories of finite length modules for $(\lieg',K')$ and $(\liel',L'\cap K')$ respectively. By definition this pair has duality if for each finite length $(\lieg',K')$-module $X$ in each degree $q$ the three modules
$$
H^q(\lieu'; X)^\vee,\;H^q(\lieu'; X^*),\;H^q(\lieu'; X^*/X^\vee),
$$
all lie in $\mathcal C_{\rm fl}(\liel',L'\cap K')$, where the superscripts $\cdot^\vee$ and $\cdot^*$ denote the $K'$- resp.\ $(L'\cap K')$-finite duals, and furthermore the Euler characteristic
$$
\sum_{q}(-1)^q[H^q(\lieu';X^*/X^\vee)]\;=\;0
$$
vanishes. Once this is established in our situation, the proof of the analogous statement of Theorem 1.4 in loc.\ cit.\ goes through word by word.

To see that the two categories of finite length modules form a pair with duality, it suffices in our context to see that the natural map
$$
H^q(\lieu'; X^\vee)\to H^q(\lieu'; X^*)
$$
is an isomorphism, which is equivalent to the vanishing of $H^q(\lieu'; X^*/X^\vee)$. This may be proved as in Proposition 1.2 in loc.\ cit., which is stated for admissible modules, but the proof given in loc.\ cit.\ works mutatis mutandis also in the finite length case for arbitrary $F'$.
\end{proof}

As in Proposition 1.6 in \cite{januszewski2011} we see that our algebraic characters over $F'$ are transitive, and similarly we see as in Proposition 1.7 of loc.\ cit.\ that our characters are compatible with restrictions, which generalizes statement (iv) of Theorem \ref{thm:characters}. The compatibility of $F'$-rational characters with translation functors as discussed in section 2 of loc.\ cit.\ remains valid as well.

Given a map of subfields $\tau:F'\to F''$ of $\CC$, which induces the identity on $F$, we may consider the exact base change functors
$$
\mathcal C_{?}(\lieg',K')\;\to\;\mathcal C_{?}(\lieg'',K''),
$$
for $?\in\{\rm fd,\rm fl\}$, and similarly
$$
\mathcal C_{?}(\liel',L'\cap K')\;\to\;\mathcal C_{?}(\liel'',L''\cap K''),
$$
where double prime denotes the base change to $F''$ along $\tau$. These induce maps
$$
-\otimes_{F',\tau}F'':\quad
K_{\rm ?}(\lieg',K')\to
K_{\rm ?}(\lieg'',K''),
$$
on the level of Grothendieck groups, and mutatis mutandis for the modules over the Levi factor of $\lieq'$. These maps are additive, multiplicative, and respect duals in the sense of Theorem \ref{thm:characters}.

Now $\lieq$ itself gives rise to a parabolic subalgebra $\lieq''\subseteq\lieg''$, which is easily seen to be constructible, and we have the relation
$$
W_{\lieq'}\otimes_{F',\tau}F''\;=\;W_{\lieq''}.
$$
As we have already seen in Proposition \ref{prop:equicohomology}, by the very definition of $\lieu$- resp.\ $\lieu''$-cohomology via standard complexes (or the Homological Base Change Theorem), we see that
$$
(-\otimes_{F',\tau}F'')\circ H_{\lieq'}\;=\;H_{\lieq''}(-\otimes_{F',\tau}F'').
$$
Since base change commutes with tensor products, we just proved
\begin{proposition}\label{prop:cqgallois}
For each constructible subalgebra $\lieq'\subseteq\lieg'$ and each $F$-linear map of fields $\tau:F'\to F''\subseteq\CC$, $\lieq''=\lieq'\otimes_{F',\tau}F''$ is constructible again and we have the commuting square
$$
\begin{CD}
K_{\rm fl}(\lieg',K')@>-\otimes_{F',\tau}F''>> K_{\rm fl}(\lieg'',K'')\\
@Vc_{\lieq'}VV @VVc_{\lieq''}V\\
C_{\lieq'}(\liel',L'\cap K')@>>-\otimes_{F',\tau}F''> C_{\lieq''}(\liel'',K'')
\end{CD}
$$
\end{proposition}

If $G$ is a reductive group over a number field $F$, the distribution characters of admissible $G(\Adeles_F^{(\infty)})$-modules preserve rationality and commute with base change, as Hecke operators act via operators of finite rank. Thus for factorizable automorphic representations of $G(\Adeles_F)$, this obserservation together with Proposition \ref{prop:cqgallois} implies the existence of algebraic global characters which preserve rationality and commute with base change.

\section{Frobenius-Schur Indicators}\label{sec:frobeniusschur}

In this section we introduce and study Frobenius-Schur indicators for $(\liea,B)$-modules over arbitrary fields and discuss their relation to the descent problem in quadratic extensions $F'/F$. Contrary to common terminology, we define the indicator as obstruction to descent, regardless of existence or non-existence of invariant bilinear forms. Our choice in Definition \ref{def:indicator} below is a notational convenience and simplifies the formulation of many statements. Conceptually it would be preferrable to define the indicator as an element in the Brauer group of $F$ as in the proofs of Propositions \ref{prop:necessaryadmissiblegroups} and \ref{prop:unconditionaladmissiblegroups}.

The main results of section \ref{sec:fsgeneralnotion} are a local global principle for quadratic descent (cf.\ Theorem \ref{thm:localglobal}) and Proposition \ref{prop:minimalktypedescent}, which reduces the descent problem for $(\liea,B)$-modules to descent of suitable $B$-types.

In section \ref{sec:fsunitary} we show that the infinitesimally unitary case over the complex numbers, the indicator classifies irreducibles and their invariant bilinear complex forms accordingly. Here the situation turns out to be completely analogous to the classical case treated by Frobenius and Schur in \cite{frobeniusschur1906}.

Adams determined classical Frobenius-Schur indicators classifying invariant bilinear forms on complex representations of real reductive groups with the property that all irreducible representations are self-dual in \cite{adams2014}. By our results in section \ref{sec:fsunitary} Adams' results fit into our framework and allow us to extend the list of cases in which we obtain optimal results (cf.\ the proof of Theorem \ref{thm:realstandardmodules}).

\subsection{The general notion}\label{sec:fsgeneralnotion}

In this section we let $(\liea,B)$ denote a pair over a field $F$ of characteristic $0$ and we fix a quadratic extension $F'/F$.

We write $(\liea',B')$ for the base change of $(\liea,B)$ to $F'$. Let $X$ be an irreducible $(\liea',B')$-module and denote by $Y:=\res_{F'/F}X$ its restriction of scalars to $F$, which we consider as an $(\liea,B)$-module over $F$. We write $\tau:F'\to F'$ for the non-trivial automorphism of $F'/F$ and set
$$
\overline{X}\;:=\;X\otimes_{F',\tau}F'.
$$
We remark that the action of $(\liea_{F'},B_{F'})$ on $\overline{X}$ depends on the real form $(\liea,B)$ (cf.\ section \ref{sec:galoisactions}).

We have the decomposition
\begin{equation}
Y'\;:=\;Y\otimes_FF'\;=\;X\oplus\overline{X},
\label{eq:xplusxbar}
\end{equation}
as $(\liea',B')$-module.

The module $Y$ is either irreducible or, by Proposition \ref{prop:galoisrestriction}, decomposes into a direct sum
\begin{equation}
Y\;\cong\;Y_1\oplus Y_2
\label{eq:x1plusx2}
\end{equation}
of two irreducible $(\liea,B)$-modules, where this decomposition may be assumed compatible with the decomposition \eqref{eq:xplusxbar}.

\begin{proposition}\label{prop:fsreal}
An irreducible $(\liea',B')$-module $X$ is defined over $F$ if and only if $Y=\res_{F'/F}X$ is reducible. In this case $Y$ decomposes into two isomorphic copies of the same module, i.e.\ $Y_1\cong Y_2$ in \eqref{eq:x1plusx2}, and $X\cong\overline{X}$. If furthermore $X$ is absolutely irreducible and $(\liea,B)$ satisfies condition (Q), then $Y_i$ is absolutely irreducible, $\End_{\liea,B}(Y_i)=F$, $i=\{1,2\}$, and $\End_{\liea,B}(Y)=F^{2\times 2}$.
\end{proposition}

\begin{proof}
Assume $X$ is defined over $F$, i.e.\ there exists an $(\liea,B)$-module $X_0$ with
$$
X\;\cong\;X_0\otimes_FF'.
$$
Then, as $(\liea,B)$-modules, we have
$$
Y\;=\;X_0\otimes_F\res_{F'/F}F',
$$
which clearly decomposes into two copies of the irreducible module $X_0$.

If $X$ is defined over $F$ and absolutely irreducible, $Y_i\cong X_0$ is absolutely irreducible as well. By Corollary \ref{cor:schur} we know $\End_{\liea,B}(X_0)=F$, which implies $\End_{\liea,B}(Y)=F^{2\times 2}$, and the claim about endomorphism rings follows.

Assume conversely that $Y$ is reducible. We claim that $Y_1$ is a model of $X$ over $F$. Indeed, consider the map
\begin{equation}
Y_1\otimes_FF'\;\to\;X
\label{eq:xitox}
\end{equation}
of $(\liea',B')$-modules induced by the inclusion
$$
Y_1\to Y
$$
by the adjointness relation of restriction of scalars. One the one hand, the map \eqref{eq:xitox} is non-zero, and therefore an epimorphism, $X$ being irreducible. On the other hand $Y_1\otimes_F F'$ is a non-trivial submodule of \eqref{eq:xplusxbar}, hence irreducible as well. Therefore \eqref{eq:xitox} is an isomorphism and the claim follows.
\end{proof}

\begin{proposition}\label{prop:fscomplexquaternionic}
Assume that $(\liea,B)$ satisfies condition (Q). Let $X$ be an absolutely irreducible $(\liea',B')$-module and assume that $Y=\res_{F'/F}X$ is an irreducible $(\liea,B)$-module. Then either
\begin{itemize}
\item[(i)] $X\cong\overline{X}$ and $\End_{\liea,B}(Y)$ is a quaternion division algebra over $F$, or
\item[(ii)] $\End_{\liea,B}(Y)=F'$ otherwise.
\end{itemize}
\end{proposition}

\begin{proof}
We first observe that by absolute irreducibility
$$
F'\;=\;\End_{\liea',B'}(X)\;\subseteq\;\End_{\liea,B}(Y),
$$
and by Proposition \ref{prop:hombasechange},
$$
\dim_F\End_{\liea,B}(Y)=\dim_{F'}\End_{\liea',B'}(Y')\;\in\;\{2,4\},
$$
by \eqref{eq:xplusxbar}. Therefore case (ii) applies if and only if $X\not\cong\overline{X}$ and $\End_{\liea,B}(Y)$ is a four-dimensional division algebra over $F$ by Proposition \ref{prop:fsreal} and Schur's Lemma otherwise.
\end{proof}

\begin{definition}\label{def:indicator}
We define the {\em Frobenius-Schur indicator} $\FS_F(X)$ of an absolutely irreducible $(\liea',B')$-module $X$ accordingly as
\begin{eqnarray*}
\FS_F(X) =& 1\quad&\Longleftrightarrow\quad \End_{\liea,B}(Y)=F^{2\times 2},\\
\FS_F(X) =& 0\quad&\Longleftrightarrow\quad \End_{\liea,B}(Y)=F',\\
\FS_F(X) =& -1\quad&\Longleftrightarrow\quad \End_{\liea,B}(Y)\;\text{is a quaternion division algebra},
\end{eqnarray*}
where $Y=\res_{F'/F}X$ as before.
\end{definition}

\begin{lemma}\label{lem:fscomplexbase}
Let $X$ be an absolutely irreducible $(\liea',B')$-module and $L'/L$ a quadratic extension dominating $F'/F$, i.e.\ $F'L=L'$ and $F'\cap L=F$. Then
$$
\FS_F(X)\;=\;0\quad\Longleftrightarrow\quad\FS_{L}(X)\;=\;0.
$$
\end{lemma}

\begin{proof}
It suffices to remark that by Proposition \ref{prop:hombasechange} the dimension of $\End_{\liea,B}(Y)$ over $F$ equals the dimension of $\End_{\liea_L,B_L}(Y_L)$ over $L$ and $Y_L=\res_{L'/L}X_{L'}$.
\end{proof}

If $F'/F$ is a quadratic extension of number fields and $v$ a place of $F$, we denote by $F_v$ the completion of $F$ at $v$ and let $F_v':=F'\cdot F_v$. We let $S_F$ denote the set of places of $F$.

\begin{theorem}\label{thm:localglobal}
Let $F'/F$ be a quadratic extension of number fields, $(\liea,B)$ a pair over $F$ satisfying condition (Q), and $X$ be an absolutely irreducible $(\liea',B')$-module. Then
\begin{itemize}
\item[(a)] If $\FS_F(X)\neq 0$, then for all but an even number of finitely many places $v\in S_F$ the module $X_{F_v'}$ over $F_v'$ admits a model over $F_v$.
\item[(b)] Furthermore, independently of any a priori condition on $\FS_F(X)$, $X$ admits a model over $F$ if and only if for all places $v\in S_F$ the module $X_{F_v'}$ over $F_v'$ admits a model over $F_v$.
\end{itemize}
\end{theorem}

\begin{remark}
Statements (a) and (b) imply that $X$ admits a model over $F$ if for all but one place $v\in S_F$ the module $X_{F_v'}$ over $F_v'$ admits a model over $F_v$.
\end{remark}

Theorem \ref{thm:localglobal} and Lemma \ref{lem:fscomplexbase} may be summarized by the following two formulae:
$$
|\FS_F(X)|\;=\;\prod_{v\in S_F}\FS_{F_v}(X_{F_v'}),
$$
$$
\FS_F(X)\;=\;\min_{v\in S_F}\FS_{F_v}(X_{F_v'}),
$$
with the convention $\FS_{F_v}(X_{F_v'}):=1$ whenever $F_v'=F_v$.

\begin{proof}[Proof of Theorem \ref{thm:localglobal}]
Assume that for all places $v\in S_F$, the module $X_{F_v'}$ over $F_v'$ admits a model over $F_v$. By Lemma \ref{lem:fscomplexbase} this implies $\FS_F(X)\neq 0$.

Let $Y=\res_{F'/F}.$ Then for {\em all places} $v$ of $F$ the quaternion algebra
\begin{equation}
\End_{\liea,B}(Y)\otimes_F F_v\;=\;
\End_{\liea_{F_v},B_{F_v}}(Y_{F_v})
\label{eq:localquaternions}
\end{equation}
is split by Proposition \ref{prop:fsreal}. By the Theorem of Hasse-Minkowski this implies that $\End_{\liea,B}(Y)$ splits, and the claim follows by Propositions \ref{prop:fsreal} and \ref{prop:fscomplexquaternionic}. This proves (b).

In general \eqref{eq:localquaternions} splits for almost all places and number of exceptions is even by the global product formula for the Hilbert symbol, hence (a) follows.
\end{proof}

\begin{remark}\label{rmk:localglobal}
The proof shows that $\End_{\liea,B}(Y)$ is uniquely determined by the collection of local Frobenius-Schur indicators $\FS_{F_v}(X_{F_v'})$.
\end{remark}

\begin{proposition}\label{prop:minimalktypedescent}
Let $(\liea,B)$ be a pair over a field $F$ satisfying condition (Q), $F'/F$ a quadratic extension and let $X$ be an absolutely irreducible $(\liea',B')$-module with $\FS_F(X)\neq 0$. Assume that there exists an absolutely irreducible rational $B'$-module $X_0$ which occurs in $X$ with multiplicity one and for which $\FS_F(X_0)\neq 0$. Then
$$
\FS_F(X)\;=\;\FS_F(X_0).
$$
\end{proposition}

\begin{proof}
Assume that $X$ is defined over $F$. Since $X_0\cong\overline{X_0}$ by our hypothesis and because $X_0$ occurs with multiplicity one in $X$, the subspace $X_0\subseteq X$ is invariant under the action of $\Gal(F'/F)$, and therefore defined over $F$ by Galois descent. This shows
$$
\FS_F(X_0)\;=\;-1\quad\Longrightarrow\quad
\FS_F(X)\;=\;-1.
$$

Assume conversely that $\FS_F(X)=-1$. Let $Y=\res_{F'/F}X$ and $Y_0=\res_{F'/F}X_0$. Then $Y$ is irreducible by Proposition \ref{prop:fsreal}. Consider the restriction map
$$
r:\quad
\End_{\liea,B}(Y)\;\to\;
\End_{B}(Y_0).
$$
This map is well defined by the characterization of $X_0$ as submodule of $X$ and the identity \eqref{eq:xplusxbar}. The irreducibility of $Y$ implies that $r$ is injective, hence an isomorphism because
$$
\dim_F\End_{B}(Y_0)\;\leq\;4.
$$
Therefore $\End_{B}(Y_0)=\End_{B}(Y)$ is a quaternion division algebra and $\FS_F(X_0)=-1$ as claimed.
\end{proof}


\subsection{Frobenius-Schur indicators in the infinitesimally unitary case}\label{sec:fsunitary}

In this section we show that in the infinitesimally unitary case over $\CC$ the previously defined Frobenius-Schur indicators behave the same way as they do classically \cite{frobeniusschur1906}. This has important consequences on the local indicators at archimedean places (cf.\ Corollary \ref{cor:fsunitary}).

Before we proceed with the complex case, we consider the following situation. Let $F\subseteq\RR$ be a subfield $(\liea,B)$, a pair over $F$ satisfying condition (Q), and choose an \lq{}imaginary\rq{} quadratic extension $F'/F$ in the sense that $F'$ comes with a fixed embedding $F'\subseteq\CC$ and $F'\not\subseteq\RR$. In other words we have a commutative square
\begin{equation}
\begin{CD}
F'@>>>\CC\\
@AAA @AAA\\
F@>>>\RR\\
\end{CD}
\label{eq:squareoffields}
\end{equation}
with $\CC=F'\RR$ and $F=F'\cap \RR$.

Since $F$ and $F'$ come with fixed embeddings into $\RR$ and $\CC$ respectively we consider likewise the pairs $(\liea_\RR,B_\RR)$ and $(\liea_\CC,B_\CC)$ over $\RR$ and $\CC$ and also the complexification $X_\CC$ of $X$, and
$$
Y_\RR\;:=\;Y\otimes_F\RR\;=\;\res_{\CC/\RR}X_\CC.
$$

A Hermitian form on $X$ is an $F'$-linear $(\liea',B')$-equivariant map
\begin{equation}
h:\quad X\otimes_{F'}\overline{X}\to F'.
\label{eq:hermitianform}
\end{equation}
satisfying the usual condition
\begin{equation}
h\circ\varepsilon\;=\;\tau\circ h,
\label{eq:hsymmetry}
\end{equation}
of conjugate symmetry, where $\tau:F'\to F'$ denotes the non-trivial Galois automorphism of the extension $F'/F$ as before, $\overline{X}$ is the $\tau$-twist of $X$ and
$$
\varepsilon:\quad X\otimes_{F'}\overline{X}\;\to\; X\otimes_{F'}\overline{X},
$$
$$
v\otimes w\;\mapsto\;\kappa^{-1}(w)\otimes\kappa(v)
$$
is the conjugate linear transposition of the arguments of $h$, with
$$
\kappa:\quad X\to\overline{X},\quad v\mapsto v\otimes 1
$$
denoting the natural conjugate linear isomorphism of $F'$-vector spaces.

We say that $h$ is posititive definite if the complex form
\begin{equation}
h_\CC:\quad X_\CC\otimes_{\CC}\overline{X}_\CC\to \CC
\label{eq:complexhermitianform}
\end{equation}
deduced from $h$ is positive definite. We call $X$ infinitesimally unitary if the form $h$ is positive definite.

\begin{lemma}\label{lem:hermitianbasechange}
An absolutely irreducible admissible $(\liea',B')$-module $X$ admits a non-zero invariant Hermitian form (resp.\ an invariant (anti-)symmetric bilinear form) if and only if $X_\CC$ is admits a non-zero invariant Hermitian form (resp.\ an invariant (anti-)symmetric bilinear form).
\end{lemma}

\begin{proof}
As before we write $X^\vee$ for the $K$-finite and hence admissible dual of $X$. Then any non-zero sesquilinear form $h$ on $X$ corresponds uniquely to a non-zero $(\liea',B')$-linear map
\begin{equation}
h^\vee:\quad \overline{X}\to X^\vee,
\label{eq:hermitianiso}
\end{equation}
$$
x\;\mapsto\;h(-\otimes x).
$$
As absolutely irreducible admissible module $X$ is reflexive, hence $h^\vee$ is always an isomorphism. In particular there is an invariant non-zero sesquilinear form on $X$ if and only if
$$
\Hom_{\liea',B'}(\overline{X}, X^\vee)\;\neq\; 0.
$$
By Proposition \ref{prop:hombasechange} the latter property is independent of the base field. Therefore, $X$ admits a non-zero sesquilinear form if and only if $X_\CC$ does.

Assume $X_\CC$ admits an invariant Hermitian form. We need to show that $X$ admits an invariant form $h$ satisfying \eqref{eq:hsymmetry} as well (the converse being tautological).

By the absolute irreducibility and Schur's Lemma, the space of invariant sesquilinear forms is one-dimensional in the cases at hand, and the map
$$
h\;=\;\tau\circ h\circ\varepsilon,
$$
induces an action of the symmetric group $S_2$ on two elements on this space. Therefore there is a non-zero form $h$ satisfying \eqref{eq:hsymmetry}, if and only if the action of $S_2$ is trivial, a property which is independent of the base.

The proof in the case of invariant bilinear forms proceeds mutatis mutandis.
\end{proof}

\begin{corollary}\label{cor:unitaritybasechange}
An absolutely irreducible admissible $(\liea',B')$-module $X$ is infinitesimally unitary if and only if $X_\CC$ is infinitesimally unitary in the usual sense.
\end{corollary}

\begin{corollary}\label{cor:unitarityindicators}
If $F$ is a number field, $X$ an absolutely irreducible admissible $(\liea',B')$-module, and $v$ an archimedean place of $F$ such that $X_{F_v'}$ admits an invariant (anti-)symmetric bilinear form, then for every place $v$ of $F$, $X_{F_v'}$ admits an invariant (anti-)symmetric bilinear form.
\end{corollary}

\begin{proof}
By Lemma \ref{lem:hermitianbasechange} $X$ admits an invariant (anti-)symmetric bilinear form, which implies the existence of invariant (anti-)symmetric forms over all completions $F_v'$.
\end{proof}

For the rest of the section we keep the previous notation with the convention that $F=\RR$ and $F'=\CC$, i.e.\ $(\liea,B)$ is a pair over $\RR$ and the rest retains its meaning as before. We assume $X$ to be an irreducible admissible $(\liea',B')$-module over $\CC$.

\begin{theorem}\label{thm:fsunitary}
Let $(\liea,B)$ be a pair over $\RR$ satisfying condition (Q) and let $X$ be an irreducible admissible infinitesimally unitary $(\liea',B')$-module over $\CC$. Then $\FS_\RR(X)\neq 0$ if and only if $X$ carries a non-zero invariant bilinear form. In this case $X$ is real (quaternionic) if and only if this form is symmetric (resp.\ anti-symmetric).
\end{theorem}

We remark that in the infinitesimally unitary case ($F=\RR$) Theorem \ref{thm:fsunitary} provides another proof for Proposition \ref{prop:minimalktypedescent}.

\begin{proof}
Recall that $Y=\res_{\CC/\RR}X$ and $Y'$ is the base change of $Y$ to $\CC$. In the infinitesimally unitary case the existence of the isomorphism \eqref{eq:hermitianiso} implies that decomposition \eqref{eq:xplusxbar} reads
\begin{equation}
Y'\;\cong\;X\oplus X^\vee.
\label{eq:xplusxvee}
\end{equation}
Therefore, saying that $X$ is real or quaternionic is, equivalent to saying that $X$ is self-dual, which is the same to say that $X$ carries a non-zero invariant bilinear form. This proves the first part of the theorem.

As to the second part, our proof is a mere categorical reinterpretation of the classical proof in \cite{frobeniusschur1906}. From now on we assume $\FS_\RR(X)\neq 0$, which means that we have a $\CC$-linear isomorphism
$$
\iota:\quad \overline{X}\to X
$$
of $(\liea',B')$-modules.

Recall that the symmetry of $h$ is reflected in the identity \eqref{eq:hsymmetry}. A direct computation shows that
$$
h'\;:=\;\tau\circ h\circ(\iota\kappa\otimes \kappa\iota)
$$
is another positive definite Hermitian form on $X$. Therefore, by Schur's Lemma, we find an $\eta\in\RR_{>0}$ satisfying
$$
h'\;=\;\eta\cdot h.
$$
If we replace $\iota$ by $\lambda\cdot\iota$, $\lambda\in \CC^\times$, we obtain
$$
\tau\circ h\circ((\lambda\iota)\kappa\otimes \kappa(\lambda\iota))\;=\;
(N_{\CC/\RR}(\lambda)\cdot \eta)\cdot h.
$$
Therefore we may assume have $\eta=1$ without loss of generality, and
\begin{equation}
h\;=\;\tau\circ h\circ(\iota\kappa\otimes \kappa\iota).
\label{eq:alpharelation}
\end{equation}

Again by Schur's Lemma, the isomorphism $(\iota\kappa)^2:X\to X$ equals
\begin{equation}
(\iota\kappa)^2\;=\;\lambda\cdot {\bf1}_X
\label{eq:lambdaid}
\end{equation}
for a constant $\lambda\in \CC^\times$, and the relation
$$
(\iota\kappa)^2\circ\iota\;=\;\iota\circ (\kappa\iota)^2
$$
implies likewise that
$$
(\kappa\iota)^2\;=\;\lambda\cdot{\bf1}_{\overline{X}}.
$$
We conclude with \eqref{eq:alpharelation} that
\begin{eqnarray*}
\lambda^2\cdot h
&=&
h\circ ((\iota\kappa)^2\otimes (\kappa\iota)^2)\\
&=&
\tau\circ\tau\circ h\circ(\iota\kappa\otimes\kappa\iota)\circ(\iota\kappa\otimes\kappa\iota)\\
&=&
\tau\circ h\circ (\iota\kappa\circ\kappa\iota)
\\
&=&
h,
\end{eqnarray*}
hence $\lambda^2=1$, i.e.\ $\lambda=\pm1$.

Consider the invariant bilinear form
$$
b\;:=\;h\circ({\bf1}_X\otimes \kappa\iota\kappa)
$$
on $X$. Since $h$ is non-degenerate, so is $b$. A direct calculation, exploiting relations \eqref{eq:hsymmetry}, \eqref{eq:alpharelation} and \eqref{eq:lambdaid}, yields
\begin{eqnarray*}
\lambda\cdot
b(w\otimes v)&=&
h(\iota\kappa\iota\kappa(w)\otimes \kappa\iota\kappa(v))\\
&=&
h(\kappa^{-1}(\kappa\iota\kappa\iota\kappa(w))\otimes \kappa(\iota\kappa(v)))\\
&=&
\tau\circ h(\iota\kappa(v)\otimes \kappa\iota\kappa\iota\kappa(w))\\
&=&
h(v\otimes \kappa\iota\kappa(w))\\
&=&
b(v\otimes w).
\end{eqnarray*}
Hence the explicit relation
\begin{equation}
b\circ\delta\;=\;
\lambda\cdot
b,
\label{eq:bsymmetry}
\end{equation}
where
$$
\delta:\quad V\otimes V\to V\otimes V,\quad v\otimes w\mapsto w\otimes v.
$$

If $X$ is of real type, and $X_0$ is a model of $X$ over $\RR$, we may choose $\iota$ as an $\RR$-multiple of the map
$$
\iota':\quad v\otimes \tau(c)\mapsto v\otimes c,
$$
where $v\in X_0$ and $c\in F'$. We conclude that for
$$
0\neq v=v_0\otimes 1\;\in\;X_0\otimes_FF',
$$
we obtain
$$
(\iota\kappa)^2(v)\;=\;v,
$$
and thus $\lambda=1$.

Assume conversely that $\lambda=1$. We consider $\iota\kappa$ as an endomorphism of $Y=\res_{\CC/\RR}X$. Being conjugate $\CC$-linear, $\iota\kappa$ is not a multiple of the identity. Hence by \eqref{eq:lambdaid} its minimal polynomial is given by
$$
\xi^2-1\;=\;(\xi-1)\cdot(\xi+1).
$$
Therefore the endomorphism
$$
{\bf1}_{Y}-\iota\kappa\;\in\;\End_{\liea,B}(Y)
$$
has a kernel satisfying
$$
0\;\neq\;\ker({\bf1}_{Y}-\iota\kappa)\;\subsetneq\; Y,
$$
which implies that $Y$ is reducible, and with Proposition \ref{prop:fsreal} we conclude that $X$ is real.

We saw that $\FS_\RR(X)=\lambda$, which by \eqref{eq:bsymmetry} concludes the proof.
\end{proof}

\begin{remark}
Replaying the proof of Theorem \ref{thm:fsunitary} over the field $F$ in the situation of \eqref{eq:squareoffields} shows that $\lambda\in F^\times$, $\eta=\pm\lambda$ and $b$ is symmetric if and only if $\eta=\lambda$. Furthermore $X$ admits a model over $F$ if and only if $\lambda\in N_{F'/F}F'^\times$. If $F$ is a number field, the local obstruction for $\lambda$ being a local norm at a place $v$ is $\FS_{F_v}(X_{F_v'})$, which for ramified archimedean places is controlled by the symmetry property of the bilinear form $b$ at hand, and Hasse's Norm Principle (due to Hilbert in the case at hand) provides another approach to Theorem \ref{thm:localglobal}.
\end{remark}

\begin{corollary}\label{cor:fsunitary}
Let $F'/F$ be a quadratic extension of number fields which is ramified at at least one archimedean place $v_\infty$ of $F$. Let $(\liea,B)$ a pair satisfying condition (Q) and assume that $X$ is an absolutely irreducible admissible $(\liea',B')$-module, assumed to be unitary at all archimedean places $v$ ramified in $F'/F$. Then for every ramified archimedean prime $v$ of $F$ we have
$$
\FS_{F_v}(X_{F_v'})\;=\;
\FS_{F_{v_\infty}}(X_{F_{v_\infty}'}).
$$
In particular if $X_{F_{v_\infty}'}$ is real, $X_{F_v'}$ admits a model over $F_v$ for every archimedean place $v$ of $F$.
\end{corollary}

\begin{proof}
The Corollary is an immediate consequence of Lemma \ref{lem:fscomplexbase}, Lemma \ref{lem:hermitianbasechange} and Theorem \ref{thm:fsunitary}.
\end{proof}

\section{Admissible Models of Compact Groups}\label{sec:admissiblegroups}

In order to apply the obtained results on rationality and descent in quadratic extensions we are naturally led to the following question. Given a quadratic extension of number fields $F'/F$, ramified at a fixed archimedean place $v_\infty$, and a compact real Lie group $K_\RR$, identify models $K$ of $K_\RR$ over $F$ (with respect to $v_\infty$) satisfying the following condition:
\begin{itemize}
\item[(A)] All absolutely irreducible rational $K'$-modules $X$ over $F'$ are defined over their field of rationality (i.e.\ $\FS_{F}(X)\neq-1$).
\end{itemize}
We call such an $F$-form $F'/F$-{\em admissible} or simply {\em admissible} if the extension $F'/F$ is clear from the context. We emphasize that the notion of admissibility depends on the real place $v_\infty$.

\begin{remark}
Admissible forms do not exist for all groups $K_\RR$, for admissibility implies that $\FS_{F_{v_\infty}}(X\otimes F_{v_\infty}')\neq -1$ for all absolutely irreducibles $X$ defined over $F'$.
\end{remark}

\begin{remark}\label{rmk:adjointderivedgroup}
All models of compact groups with adjoint derived group are admissible (cf.\ \cite{boreltits1965}).
\end{remark}

For some applications the notion of $F'/F$-admissibility might be too strict. therefore we also introduce the larger class of $F$-forms $K$ subject to the weaker condition
\begin{itemize}
\item[(A')] All absolutely irreducible rational $K'$-modules $X$ over $F'$ satisfy
\begin{equation}
\FS_{F}(X)\;=\;\FS_{F_{v_\infty}}(X\otimes F_{v_\infty}').
\label{eq:archimedeanlocalglobal}
\end{equation}
\end{itemize}
Such a form is called $F'/F$-{\em semi-admissible}.

By our results in section \ref{sec:frobeniusschur} it is clear that every admissible form is semi-admissible. Proposition \ref{prop:semitoadmissible} below gives a sufficient condition for the converse to be true.

\subsection{Groups of class $\mathcal S$}\label{sec:groupsofclasss}

\begin{lemma}\label{lem:extendenboreltits}
Let $K$ be a linear reductive group over a field $F$ of characteristic $0$ with the property that $K$ is a semi-direct product of $K^0$ and $\pi_0(K)$, the latter of order at most $2$. Let $F'/F$ be a quadratic extension and $X$ an absolutely irreducible $K'$-module with $\FS_F(X)\neq 0$, and assume that an irreducible ${K'}^0$-module $X_0$ in $X$ satisfies $\FS_F(X_0)\neq-1$. Then either
\begin{itemize}
\item[(i)] $\FS_F(X)=1$, i.e.\ $X$ admits a model over its field of rationality $F$, or
\item[(ii)] $\FS_F(X)=-1$ and $X$ decomposes into two irreducible Galois conjugate ${K'}^0$-modules $X_i$ (i.e.\ $\FS_F(X_i)=0$).
\end{itemize}
The exceptional case (ii) applies if and only if we have $\FS_L(X_{L'})=-1$ for every quadratic extension $L'/L$ dominating $F'/F$. Case (ii) does not occur if $F$ is a number field.
\end{lemma}

\begin{remark}
The hypotheses on $X$ in Lemma \ref{lem:extendenboreltits} are always satisfied if $K^0$ is quasi-split or has adjoint derived group (cf.\ \cite{boreltits1965}).
\end{remark}

\begin{proof}
For $K=K^0$ there is nothing to prove. If $K\neq K^0$ there are two cases to consider:
\begin{itemize}[leftmargin=5.0em]
\item[(I):] $X$ is irreducible as ${K'}^0$-module,
\item[(II):] $X$ is reducible as ${K'}^0$-module.
\end{itemize}

We treat case (I) first. In this case $X\cong\overline{X}$ also as ${K'}^0$-modules, which implies that the field of rationality of the ${K'}^0$-module $X$ is $F$. By our assumption on $X_0$ the ${K'}^0$-module $X$ admits a model $X_0$ over $F$. Fix a section of $K\to \pi_0(K)$ and denote its image by $S_2$. We claim that $X_0$ is stable under the action of $S_2\subseteq K$.

Let $\delta\in S_2$ denote the non-trivial element. We claim that $\delta$ acts as a scalar $\pm 1$ on $X$. Indeed, if $X_\pm\subseteq X$ is the eigen space of the action of $\delta$ for the possible eigen value $\pm 1$, then $X_\pm$ is stable under ${K'}^0$ since $\delta$ normalizes ${K'}^0$. Hence either $X_+=X$ or $X_-=X$, and case (I) falls into case (i).

In case (II) $X$ decomposes into a direct sum
$$
X\;\cong\;X_1\oplus X_2
$$
of irreducible ${K'}^0$-modules $X_i$ and
\begin{equation}
X\;=\;\ind_{{K'}^0}^{K'}(X_i).
\label{eq:inductionidentity}
\end{equation}
There are two subcases.

If $X_1\cong \overline{X_1}$ our hypothesis implies $\FS_F(X_1)=1$, i.e.\ $X_1$ admits a model over $F$, and as an induced module $X$ also admits a model over $F$. This case also falls into (i).

If $X_1\not\cong\overline{X_1}$, we have $X_2\cong\overline{X_1}$ since $X\cong\overline{X}$. Then $Y_i:=\res_{F'/F}X_i$ is a model of the ${K'}^0$-module $X$ over $F$. The modules $Y_i$ become isomorphic over $F'$, and Proposition \ref{prop:hombasechange} implies that $Y_1\cong Y_2$ already over $F$, because every non-zero homomorphism over $F$ is an isomorphism. Therefore, if a model of $X$ over $F$ exists, it is given by $Y_1$.

We need to check when the action of $\delta$ on $X$ commutes with the non-trivial Galois automorphism $\tau\in\Gal(F'/F)$ induced by the model $Y_1$ at hand.

By \eqref{eq:inductionidentity} the action of $\delta$ does not stabilize the spaces $X_i$, but interchanges them. Consider the action of $\tau$ on $X$ defined by the model $Y_1$, i.e.\ we fix an isomorphism
$$
\iota:\quad Y_1\otimes_F F'\to X
$$
and decree that the diagram
$$
\begin{CD}
X@>t>> X\\
@A\iota AA @AA\iota A\\
Y_1\otimes_F F'@>>y\otimes c\mapsto y\otimes \tau(c)>
Y_1\otimes_F F'
\end{CD}
$$
be commutative. Then $t$ is a conjugate-linear isomorphism of ${K'}^0$-modules, which interchanges the submodules $X_i$ as well.

Restricted to $X_1$ and $X_2$ the map $t$ induces conjugate-linear isomorphisms
$$
t_1:\quad X_1\to X_2,\quad t_2:\quad X_2\to X_1,
$$
which are unique up to scalar by the absolute irreducibility of $X_i$ (cf.\ Corollary \ref{cor:schur}). Since $\delta$ is $F$-linear and of order $2$ this shows that there is a constant $\lambda\in F'$ satisfying the two relations
$$
\lambda t_1\;=\;\delta \circ t_2\circ \delta,\quad
\lambda t_2\;=\;\delta \circ t_1\circ \delta.
$$
In particular we deduce $\lambda^2=1$, hence $\lambda=\pm1$ and
$$
\delta \circ t\;=\;\lambda\cdot (t\circ \delta).
$$
Therefore $\delta$ commutes with $\tau$ if and only if $\lambda = 1$, which falls into case (i).

The case $\lambda=-1$ is the exceptional case (ii). If $L'/L$ is a quadratic extension dominating $F'/F$, we know by Lemma \ref{lem:fscomplexbase} that $\FS_L(X_{i,L'})=\FS_F(X_i)=0$ , and the above computation shows that $\lambda$ is independent of the extension $L'/L$. Therefore case (ii) applies for $F'/F$ if and only if it applies for $L'/L$.

If $F$ is a number field, we know by Theorem \ref{thm:localglobal} that there are infinitely many places $v$ of $F$ where $F_v'\neq F_v$ and $\FS_{F_v}(X_{F_v'})=1$, hence case (ii) cannot apply.
\end{proof}

We introduce the following class of groups.
\begin{definition}
A reductive group $K$ over a field $F$ is {\em of class} $\mathcal S$, if $K$ is an isogenous image of a direct product $\prod_{i=1}^rK_i$ where each $K_i$ is a semi-direct product of its identity component $K_i^0$ and a group of order at most two. A pair $(\liea,B)$ is {\em of class} $\mathcal S$ if $B$ is so.
\end{definition}

\begin{remark}Each such pair (or group) satisfies condition (Q).\end{remark}

\begin{corollary}\label{cor:kglobaldescent}
Let $K$ be reductive group over a number field $F$ of class $\mathcal S$. Let $F'/F$ be a quadratic extension and let $X$ be an absolutely irreducible $K'$-module with $\FS_F(X)\neq 0$, Then the following statements are equivalent:
\begin{itemize}
\item[(i)]
$\FS_F(X)=1$.
\item[(ii)] $\FS_{F_v}(X_{F_v})=1$ for every place $v$ of $F$ which is inert or ramified in $F'/F$ and where $K^0$ is not quasi-split, with one possible exception $v_0$.
\end{itemize}
\end{corollary}

\begin{proof}
Obviously statement (i) implies that statement (ii) holds for {\em all} places $v$. We need to show that the local conditions in (ii) imply (i).

For every finite place $v\neq v_0$ where $K^0$ is quasi-split $X_{F_v'}$ admits a model over its field of rationality $F_v$ by Lemma \ref{lem:extendenboreltits}. At places $v\neq v_0$ which are not inert or ramified we have $F_v'=F_v$ and the existence of models is tautological. By hypothesis (ii), we also find models over $F_v$ at all other places $v\neq v_0$ of $F$. Therefore the parity condition in statement (a) of Theorem \ref{thm:localglobal} shows that we find models locally everywhere and statement (b) of said Theorem proves the existence of a model over $F$.
\end{proof}

\subsection{Criteria of semi-admissibility}\label{sec:semiadmissible}

\begin{proposition}\label{prop:semiadmissibility}
Let $K$ be a reductive group over a number field $F$ and $F'/F$ a quadratic extension ramified at infinity. Assume the following conditions are satisfied:
\begin{itemize}
\item[(A)] $K$ is of class $\mathcal S$.
\item[(B)] For each ramified archimedean place $v$, $K^0_{F_{v}}$ is anisotropic.
\item[(C)] $K^0$ is quasi-split at all finite places of $F$, with one possible exception.
\end{itemize}
Then $K$ is an $F'/F$-semi-admissibile model for $K(F_v)$ for every archimedean prime $v$ of $F$ ramified in $F'/F$.
\end{proposition}

\begin{proof}
Let $X$ be any absolutely irreducible rational $B'$-module over $F'$ and fix a ramified archimedean prime $v_\infty$.

By Lemma \ref{lem:fscomplexbase} we know that the condition $\FS_F(X)=0$ is independent of the base field. Hence we may assume $\FS_F(X)\neq 0$.

By hypothesis (B), $K(F_v)$ is compact and therefore the complexification of every absolutely irreducible rational $B$-module $X$ over $F'$ is unitary. Hence Corollary \ref{cor:fsunitary} applies and shows that for every ramified archimedean prime $v$,
$$
\FS_{F_{v}}(X_{F_{v}'})\;=\;\FS_{F_{v_\infty}}(X_{0,F_{v_\infty}'}).
$$
Hypothesis (C) and Corollary \ref{cor:kglobaldescent} then imply the claim.
\end{proof}

\subsection{Criteria of admissibility}

\begin{proposition}\label{prop:semitoadmissible}
Let $K_\RR$ be a compact Lie group with the following properties:
\begin{itemize}
\item[(i)] $K_\RR$ is of class $\mathcal S$ (in the obvious sense).
\item[(ii)] Every complex irreducible $K_\RR^0$-module $X$ satisfies $\FS_\RR(X)\neq-1$.
\end{itemize}
Then, given any quadratic extension of number fields $F'/F$ ramified at infinity, any $F'/F$-semi-admissible model $K$ of $K_\RR$ over $F$ with respect to a ramified archimedean place $v_\infty$ is $F'/F$-admissible.
\end{proposition}

\begin{proof}
The claim follows with Lemma \ref{lem:extendenboreltits}.
\end{proof}

\begin{proposition}\label{prop:necessaryadmissiblegroups}
Let $K_\RR$ be a (non-abelian) simple compact Lie group. Then for every irreducible complex representation $X$ of $K_\RR$ we have $\FS_\RR(X)\neq -1$ if and only if $K_\RR$ is one of the following types:
\begin{itemize}
\item[($A_n$)] 
$n\equiv 0,2,3\pmod{4}$, any compact form;\\$n\equiv 1\pmod{4}$, any cover of odd degree of the adjoint compact form;
\item[($B_n$)] 
$n\equiv 0,1\pmod{4}$, any compact form;\\$n\equiv 2,3\pmod{4}$, the adjoint compact form;
\item[($C_n$)] 
the adjoint compact form;
\item[($D_n$)] 
$n\equiv 0,1,3\pmod{4}$, any compact form;\\$n\equiv 2\pmod{4}$, $\SO(2n)$ and the adjoint compact form;
\item[($E_6$)] all compact forms;
\item[($E_7$)] the adjoint compact form;
\item[]\hspace*{-2.5em}($E_8$, $F_4$, $G_2$)\\
  the unique compact form.
\end{itemize}
\end{proposition}

\begin{proof}
Choose a Cartan subalgebra $\lieh$ in the complexified Lie algebra of $K_\RR$. Write $\Lambda$ for the weight lattice of $K_\RR$ and $\Lambda_0$ for the root lattice. Fixing a basis of the root system, the Galois group $\Gal(\CC/\RR)$ acts on $\Lambda$ via the action ${}_\Delta\tau$ introduced by Borel and Tits \cite{boreltits1965} (cf.\ section \ref{sec:galoisactions}), and this action descends to $\Lambda/\Lambda_0$. In the case at hand the action of complex conjugation is given by the negated longest Weyl element $-w_0$. Since $\Lambda/\Lambda_0$ may be canonically identified with the dual $C_{K_\RR}^*$ of the center of $K_\RR$, we obtain a Galois action on the latter.

The Frobenius-Schur indicator of an irreducible representation $X$ of highest weight $\lambda$ is non-zero if and only if $\lambda\in\Lambda^{\Gal(\CC/\RR)}$ and it is easy to see that the map
$$
\beta_{K_\RR}:\quad (C_{K_\RR}^*)^{\Gal(\CC/\RR)}\;\to\;H^2(\Gal(\CC/\RR);\CC^\times),
$$
$$
\lambda\;\mapsto\;[\End_{K_\RR}(\res_{\CC/\RR}X)]
$$
is well defined and a homomorphism of groups. The second cohomology on the right hand side is canonically isomorphic to the Brauer group of $\RR$ and with that identification $[\cdot]$ denotes the class associated to a central simple $\RR$-algebra.

The map $\beta_{K_\RR}$ is the cohomological incarnation of the Frobenius-Schur indicator: We have $\beta_{K_\RR}(\lambda)=1$ if and only if $\FS_\RR(X)=1$. For a generalization of this approach, valid for arbitrary fields and arbitrary Galois extensions, we refer the reader to \cite{tits1971}.

For any subgroup $C\subseteq C_{K_\RR}$ of the center, the center $C_{K_{\RR}/C}$ of $K_\RR/C$ is naturally identified with the group $C_{K_\RR}/C$, and we have a natural Galois equivariant monomorphism
$$
\iota_C:\quad C_{K_\RR/C}^*\;\to\;C_{K_\RR}^*,
$$
rendering the following diagram commutative
\begin{equation}
\begin{CD}
(C_{K_\RR}^*)^{\Gal(\CC/\RR)}@>\beta_{K_\RR}>>H^2(\Gal(\CC/\RR);\CC^\times)\\
@A\iota_CAA @AA{\bf1}A\\
(C_{K_\RR/C}^*)^{\Gal(\CC/\RR)}@>\beta_{K_\RR/C}>>H^2(\Gal(\CC/\RR);\CC^\times)
\end{CD}
\label{eq:betaisogenies}
\end{equation}
On the level of modules the map $\iota_C$ corresponts to pullback along the canonical projection $K_\RR\to K_\RR/C$.

The map $\beta_{K_\RR/C}$ is necessarily trivial if $C_{K_\RR/C}$ is of odd order or more generally if $C_{K_\RR/C}$ is contained in the kernel of $\beta_{K_\RR}$.

With these tools at hand we proceed case by case:\\
$A_n$:
If $n\not\equiv 1\pmod{4}$, then $\beta_{K_\RR}$ is trivial by table 1 of Chap.\ VIII in \cite{book_bourbaki1968}. For $n\equiv1\pmod{4}$ and $K_\RR$ simply connected the same table shows that $\beta_{K_\RR}$ is non-trivial, and since the center $C_{K_\RR}$ is cyclic of order $n+1\equiv2\pmod{4}$ in this case, the map $\beta_{K_\RR/C}$ is trivial if and only if the order of $C$ is even.\\
$B_n$:
If $n\equiv 0,1\pmod{4}$, then $\beta_{K_\RR}$ is trivial by table 1 of loc.\ cit. For $n\equiv2,3\pmod{4}$ and $K_\RR$ simply connected $\beta_{K_\RR}$ is non-trivial, and since the center $C_{K_\RR}$ is of order $2$, the claim follows.\\
$C_n$:
The same reasoning as in type $B_n$ applies.\\
$D_n$:
If $n\equiv 0,1,3\pmod{4}$, then $\beta_{K_\RR}$ is trivial by table 1 of loc.\ cit. For $n\equiv2\pmod{4}$ and $K_\RR$ simply connected $\beta_{K_\RR}$ is non-trivial, the Galois action is trivial, and $(C_{K_\RR}^*)^{\Gal(\CC/\RR)}$ is the Klein four group $V_4$, generated by the two last fundamental weights $\omega_{n-1}$ and $\omega_n$, which both have non-trivial image under $\beta_{K_\FF}$ (cf.\ table 1 of loc.\ cit.). The dual $C_{\SO(2n)}^*$ of the center of $\SO(2n)$ is generated by the image of the first fundamental weight $\omega_1$. As an element of the dual of the center, $\omega_1$ is the sum of $\omega_{n-1}$ and $\omega_n$, hence in the kernel of $\beta_{K_\RR}$.\\
$E_6$: Since the center of the simply connected compact form is of order $3$, the claim follows.\\
$E_7$: Table 1 of loc.\ cit. shows that the simply connected compact form admits quaternionic representations, its center is of order $2$.\\
$E_8,F_4,G_2$: The simply connected compact form is adjoint.
\end{proof}

\begin{corollary}\label{cor:necessaryadmissiblegroups}
Let $K_\RR$ be a compact Lie group with the following properties:
\begin{itemize}
\item[(i)] $K_\RR$ is of class $\mathcal S$ (in the obvious sense).
\item[(ii)] ${\mathscr D}(K_\RR^0)$ is an isogenous image of a product of groups where Proposition \ref{prop:necessaryadmissiblegroups} applies.
\end{itemize}
Then for any quadratic extension $F'/F$ of number fields, ramified at infinity, every $F'/F$-semi-admissible model $K$ of $K_\RR$ with respect to a ramified archimedean place $v_\infty$ is $F'/F$-admissible.
\end{corollary}

\begin{proof}
This follows from Proposition \ref{prop:necessaryadmissiblegroups}, Lemma \ref{lem:extendenboreltits} and the multiplicativity of the Frobenius-Schur indicators $\FS_{F_{v_\infty}}(-)$ under outer products of representations.
\end{proof}

\begin{proposition}\label{prop:unconditionaladmissiblegroups}
Let $F$ be a number field and $K$ a reductive group over $F$, $F'/F$ a quadratic extension ramified at infinity, subject to the following conditions:
\begin{itemize}
\item[(i)] $K$ is of class $\mathcal S$.
\item[(ii)] For a fixed real place $v_\infty$ ramified in $F'/F$, $K_{F_{v_\infty}}^0$ is anisotropic.
\item[(iii)] $\mathscr D(K(F_{v_\infty})^0)$ is an isogenous image of a product of groups of the types
\begin{itemize}
\item[($A_n$)] 
$2\mid n$, any compact form;\\$2\nmid n$, any cover of odd degree of the adjoint compact form;
\item[]\hspace*{-2.5em}($B_n$, $C_n$, $D_n$, $E_7$)\\ 
the adjoint compact form;
\item[($E_6$)] all compact forms;
\item[]\hspace*{-2.5em}($E_8$, $F_4$, $G_2$)\\
the unique compact form;
\end{itemize}
\end{itemize}
Then $K$ is an $F'/F$-admissible model of $K(F_{v_\infty})$.
\end{proposition}

\begin{proof}
Lemma \ref{lem:extendenboreltits} reduces us to the case of connected $K$. Since the case of compact tori is clear, we are reduced to $K$ simple and connected. The remaining argument is a suitable adaption of the proof of Proposition \ref{prop:necessaryadmissiblegroups}.

We fix an algebraic closure $\overline{F}\subseteq F_{v_\infty}'$ of $F$. The map $\beta_{K_\RR}$ in the proof of Proposition \ref{prop:necessaryadmissiblegroups} may be defined relative to the extension $\overline{F}/F$ as follows. We still write $C_K^*$ for the quotient of the weight lattice divided by the root lattice of $K_{\overline{F}}$, with respect to some fixed Borel $B$ and maximal torus $T\subseteq B$, both defined over $\overline{F}$. Then we define
$$
\beta_{K}:\quad (C_{K}^*)^{\Gal(\overline{F}/F)}\;\to\;H^2(\Gal(\overline{F}/F);{\overline{F}}^\times),
$$
$$
\lambda\;\mapsto\;[\End_{K}(\res_{F'/F}X)],
$$
where $X$ is absolutely irreducible of highest weight $\lambda$ and defined over $F'$. This map enjoy the same properties as the map $\beta_{K_\RR}$ (cf.\ \cite{tits1971}).

Since the elements associated to non-split quaternion algebras in the Brauer group are of order $2$, $\beta_K$ factors over the maximal quotient of $C_K^*$ which is $2$-torsion.

We remark that we have a commutative square
\begin{equation}
\begin{CD}
(C_{K(F_{v_\infty})}^*)^{\Gal(F_{v_\infty}'/F_{v_\infty})}@>\beta_{K(F_{v_\infty})}>>H^2(\Gal(F_{v_\infty}'/F_{v_\infty});{F_{v_\infty}'}^\times)\\
@A{\gamma}AA @AAA\\
(C_{K}^*)^{\Gal(\overline{F}/F)}@>\beta_{K}>>H^2(\Gal(\overline{F}/F);{\overline{F}}^\times)
\end{CD}
\label{eq:betabasechange}
\end{equation}
where $\gamma$ is an monomorphism for trivial reasons. In particular if $\beta_{K(F_{v_\infty})}$ is non-trivial, so is $\beta_K$, but the converse needs not be true.

Again we inspect each case individually.\\
$A_n$:
Assume $K(F_{v_\infty})$ simply connected. Then $C_{K(F_{v_\infty})}=C_{K}$ is cyclic of order $n+1$. Hence if $n$ is even, $\beta_K$ must be trivial.\\
$B_n,C_n,D_n$: clear.\\
$E_6,E_7,E_8,F_4,G_2$: as before.
\end{proof}

\begin{remark} It seems that in type $A_n$ with $n$ odd we cannot do better in general: If $K$ splits over $F'$, then $(C_K^*)^{\Gal(\overline{F}/F)}$ is non-trivial of order $2$ whenever $K(F_{v_\infty})$ is simply connected and even if $\beta_{K(F_{v_\infty})}$ is trivial, the map $\beta_K$ could still happen to be non-trivial. Similar remarks apply to types $B_n, C_n,D_n$ and $E_7$.
\end{remark}

For later use we state the following generalization.

\begin{proposition}\label{prop:unitarygroups}
Let $F$ be a number field and $K$ a reductive group over $F$, $F'/F$ a quadratic extension ramified at at least one archimedean place, subject to the following conditions (i) and (ii):
\begin{itemize}
\item[(i)] For a fixed real place $v_\infty$ ramified in $F'/F$, $K_{F_{v_\infty}}^0$ is anisotropic.
\item[(ii)] $K$ is an $F$-form of a quotient of a product
$$
K_U\times K_D
$$
where $K_U$ is a product of unitary groups $\U(n)$, and $K_D$ is a group to which Proposition \ref{prop:unconditionaladmissiblegroups} applies.
\end{itemize}
Then $K$ is an $F'/F$-admissible model of $K(F_{v_\infty})$. The same conclusion is true if condition (ii) is replaced by
\begin{itemize}
\item[(ii')] $K$ is a $F'/F$-semi-admissible model of a quotient of a product
$$
K_U\times K_D
$$
where $K_U$ is a product of unitary groups $\U(n)$, and $K_D$ is a group to which Proposition \ref{prop:necessaryadmissiblegroups} applies.
\end{itemize}
\end{proposition}

\begin{proof}
We claim that under (i) and (ii) $\beta_K$ is trivial. To see this, we may assume without loss of generality that $K$ is an $F$-form of a product $K_U\times K_D$ as in the statement, which reduces us to the case $K=K_U$, since $\beta_K=\beta_{K_U}\times\beta_{K_D}$, and we know that $\beta_{K_D}=1$ by Proposition \ref{prop:unconditionaladmissiblegroups}.

We claim that $(\Lambda/\Lambda^0)^{\Gal(\overline{F}/F)}=0$. By \eqref{eq:betabasechange} it suffices to show
$$
(\Lambda/\Lambda_0)^{\Gal(F_{v_\infty}'/F_{v_\infty})}\;=\;0,
$$
which readily reduces to the case $K(F_{v_\infty})=\U(n)$. Write $e_1,\dots,e_n$ for the standard basis of $\Lambda$, and such that dominance is given by the usual condition
$$
a_1\geq a_2\geq\cdots\geq a_n
$$
for weights $\lambda=a_1e_1+\cdots a_ne_n$. Then complex conjugation acts via
$$
e_i\;\mapsto\;-e_{n+1-i},
$$
which shows that
$$
\Lambda^{\Gal(F_{v_\infty}'/F_{v_\infty})}\;=\;\{\lambda\mid\forall k:a_k=-a_{n+1-k}\}
\;\subseteq\;\Lambda_0,
$$
and the claim follows.

The proof in the case of conditions (i) and (ii') proceeds similarly using Proposition \ref{prop:necessaryadmissiblegroups}.
\end{proof}

\subsection{Explicit construction of models}\label{sec:explicitmodels}

In the cases where Proposition \ref{prop:necessaryadmissiblegroups} but not Proposition \ref{prop:unconditionaladmissiblegroups} (or Proposition \ref{prop:unitarygroups}) applies, we may try to find explicit suitable admissible models as follows.

The local conditions on $K$ in statement (ii) of Corollary \ref{cor:kglobaldescent} are always satisfied if $K^0$ has good reduction at all finite places with one possible exception $v_0$. Indeed, since the reduction of $K^0$ at finite places $v\neq v_0$ is smooth connected and therefore quasi-split by Lang's Theorem \cite{lang1958}, Hensel's Lemma implies that $K_{F_v}^0$ is quasi-split.

\subsubsection*{Orthogonal groups}

Consider the quadratic form
$$
f:\quad\ZZ^n\times\ZZ^n\to \ZZ^n,\quad ((x_i)_i,(y_j)_j)\;\mapsto\;x_1y_1+\cdots+x_ny_n,
$$
and define $\Oo(n)$ as the group scheme over $\ZZ$ whose $A$-valued points are given by
$$
\Oo(n)(A)\;:=\;\{g\in\Aut_A(A^n)\mid \forall x,y\in A^n: f(gx,gy)=f(x,y)\}.
$$
Then $\Oo(n)\otimes_\ZZ\QQ$ is a model of the compact real Lie group $\Oo(n)$, i.e.\ in particular its identity component is anisotropic over $\RR$.

For each prime $p\neq2$ the form $f$ is non-degenerate and the group $\Oo(n)^0_\ZZ\otimes\FF_p$ is smooth connected over $\FF_p$. Therefore $\Oo(n)^0\otimes_\ZZ\QQ$ is quasi-split at all $p\neq 2$ by Lang's Theorem \cite{lang1958}.

For any totally real number field $F/\QQ$, $\Oo(n)_F=\Oo(n)\otimes_\ZZ F$ is a model of $[F:\QQ]$ many copies of the compact group $\Oo(n)$. It is quasi-split at all finite places $v\nmid 2$. Therefore $K=\res_{F/\QQ}(\Oo(n)\otimes_\ZZ F)$ is quasi-split at all primes $p\neq 2$.

Then $K=\res_{F/\QQ}\Oo(n)_F$ is an $F'/\QQ$-admissible model of $\Oo(n)(F\otimes_\QQ\RR)$ for every imaginary quadratic extension $F'/\QQ$. We remark that $K^0$ is quasi-split over $\QQ(\sqrt{-1})/\QQ$.

This model of $\Oo(n)^[F:\QQ]$ comes with a canonical embedding into $\GL_n$, and we see that $G=\res_{F/\QQ}\GL(n)$ admits a subgroup $K=\res_{F/\QQ}(\Oo(n)\otimes_\ZZ F)$ satisfying the hypotheses of Theorem \ref{thm:realstandardmodules} below.

Products of this model produce suitable models $K\subseteq G$ for the maximal compact subgroup of the $\RR$-valued points of the groups $G=\res_{F/\QQ}\Oo(p,q)$, the latter being defined similarly as a subgroup of $\res_{F/\QQ}\GL_n$.

This construction also provides admissible models for $\SO(n)$.

\subsubsection*{Unitary groups}

We consider a CM field $F/F^+$, i.e.\ $F$ is a totally imaginary quadratic extension of a totally real number field $F$. Write $\tau$ for the non-trivial automorphism of $F/F^+$, and consider $F^n$ as an $2n$-dimensional $F^+$-vector space. Then we have the quadratic form
$$
f:\quad\OO_{F}^n\times\OO_{F}^n\to \OO_{F}^n,\quad ((x_i)_i,(y_j)_j)\;\mapsto\;x_1y_1^\tau+\cdots+x_ny_n^\tau,
$$
on the $\OO_{F^+}$-module $\OO_F^n$. Over any finite place $\mathfrak{p}\nmid 2$ of $F^+$ where $F/F^+$ is unramified the form $f$ is non-degenerate modulo $\mathfrak{p}$. Yet it is unclear if we obtain a model quasi-split at finite places where $F/F^+$ ramifies.

However by Proposition \ref{prop:unitarygroups} we know that for each real place $v$ of $F^+$ the group
$$
\U_f(A)\;:=\;\{g\in\Aut_A(F\otimes_{F^+}A^n)\mid \forall x,y\in F\otimes_{F^+}A^n: f(gx,g^\tau y)=f(x,y)\}
$$
is an $F/F^+$-admissible model of $\U(n)$ with respect to $v$, and likewise $\res_{F^+/\QQ}\U_f$ is a semi-admissible model of $[F^+:\QQ]$ many copies of $\U(n)$.

Applying restriction of scalars and Proposition \ref{prop:unitarygroups}, we obtain an $F'/\QQ$-admissible model $K$ over $\QQ$ for every imaginary quadratic extension $F'/\QQ$. This model embeds into $G=\res_{F/\QQ}\GL(n)$. Likewise we obtain models with embeddings into $\res_{F^+/\QQ}\U(p,q)$.

If $F=F^+(\sqrt{-1})$ this construction also provides semi-admissible models for $\res_{F^+/\QQ}\SU(n)$. The existence of semi-admissible models for $\SU(n)$ in general seems to be a delicate question.

\section{Cohomologically Induced Standard Modules}\label{sec:standardmodules}

In this section we specialize the general results obtained thus far to the case of standard modules and prove a number of optimal rationality results. In section \ref{sec:applications} we will apply these results to automorphic representations.

\subsection{Defininition of standard modules}

We depart from a reductive pair $(\lieg,K)$ over a subfield $F\subseteq\RR$ with Cartan involution $\theta$ (also defined over $F$). We assume that $(\lieg,K)$ is a model of a classical reductive pair. In particular $K^0$ is $\RR$-anisotropic.

We fix a quadratic extension $F'/F$ as in \eqref{eq:squareoffields}, subject to the same conditions, i.e.\ $F'\RR=\CC$. We write $(\lieg',K')$ for the base change of $(\lieg,K)$ to $F'$. We write $\tau$ for the non-trivial automorphism of the extension $F'/F$ and denote it also by a bar.

We adopt the following convention to ease notation. If we consider functors such as $H^q(\lieu;-)$, where $\lieu$ is defined over some base field $F$, then we denote canonical functor deduced from $H^q(\lieu;-)$ on the category of modules over any extension $F'/F$ still the same.


Let $\lieq$ be a $\theta$-stable germane parabolic inside of $\lieg'$ (defined as in section \ref{sec:reductivepairs}) with Levi decomposition $\lieq=\liel+\lieu$. Then $\liel$ is defined over $F$ and the complex conjugate of $\lieu$ is the opposite of $\lieu$, i.e.\
$$
\overline{\lieu}={\lieu}^-.
$$
We recall the Cartan decomposition
$$
\lieg\;=\;\liep\oplus\liek,
$$
which is defined over $F$. Then $\lieq\cap\liek'$ is a parabolic subalgebra of $\liek$. We let
$$
L\cap K\;=\;N_{K}(\lieq)
$$
denote the normalizer of $\lieq$ in $K$, which happens to be defined over $F$, as is the subgroup $L\subseteq G$, the normalizer of $\lieq$ in $G$.

Let $Y$ be an $(\liel,L\cap K)$-module over $F'$. As before we consider
$$
Y_{\lieq}\;:=\;Y\otimes_{F'}\wedge^{\dim\lieu}\lieu
$$
as a $(\lieq,L'\cap K')$-module with trivial action of the radical. Then we have the cohomologically induced module
$$
\mathcal R^q(Y)\;=\;R^q\Gamma(\Hom_{U(\lieq)}(U(\lieg'),Y_{\lieq})_{L\cap K}),
$$
where $\Gamma(\cdot)$ denotes the rational Zuckerman functor from section \ref{sec:equizuckerman} for the inclusion
$$
(\lieg',L'\cap K')\;\to\;(\lieg',K').
$$
We remark that we may compute $\mathcal R^q(Y)$ as the $q$-th right derived functor of the composition
$$
\Gamma\circ \Hom_{U(\lieq)}(U(\lieg'),-)_{L\cap K},
$$
which enables us to apply the Homological Base Change Theorem by choosing an admissible standard resolution of $Y_{\lieq}$ whenever we consider possibly infinite field extensions.

In these cases the functor $\mathcal R^q(-)$ preserves infinitesimal characters in the sense that if $L\cap K$ meets every (geometrically) connected component of $K$, and if $Y$, as an $(\liel',L'\cap K')$-module, has infinitesimal character $\lambda$, then $\mathcal R^q(Y)$ has infinitesimal character $\lambda+\rho(\lieu')$. This easily follows via base change from the classical case (cf.\ \cite[Corollary 5.25]{book_knappvogan1995} for example). Classical non-vanishing and irreducibility criteria generalize to our setting without modification and identical proofs.

We remark that the rational construction of the Hecke algebra in \cite{januszewski2015pre2} together with the associated Koszul resolution shows that $\mathcal R^q(-)$ commutes with base change without any hypothesis, which implies that above preservation statements eventually hold in general.

\subsection{The bottom layer}

Consider the commutative square of pairs
$$
\begin{CD}
(\lieg,L\cap K)@>>> (\lieg, K)\\
@AAA @AAA\\
(\liek,L\cap K)@>>> (\liek, K)
\end{CD}
$$
By Proposition \ref{prop:find} the corresponding Zuckerman functors commute with the tautological forgetful functors and we obtain a natural map
$$
\beta:\quad
\mathcal R_K^q(Y\otimes_{F'}\wedge^{\dim\lieu\cap\liep'}\lieu\cap\liep')\;:=\;R^q\Gamma_K(U(\liek')\otimes_{U(\overline{\lieq}\cap\liek')}(Y_{\lieq\cap\liek'}\otimes_{F'}\wedge^{\dim\lieu\cap\liep'}\lieu\cap\liep'))
$$
$$
\;\to\;
R^q\Gamma(U(\lieg')\otimes_{U(\overline{\lieq})}Y_{\lieq})\;=\;\mathcal R^q(Y),
$$
where the first Zuckerman functor $\Gamma_K$ is associated to the inclusion into the compact pair, and the second one to the inclusion into the (possibly) non-compact pair $(\lieg,K)$ as before. Traditionally this map is called the \lq{}bottom layer map\rq{}, even though conceptually it should be defined as a projection rather than an inclusion in the case of Zuckerman functors.

\subsection{The good and the fair range}

We define the middle degree
$$
S_{\lieq}\;:=\;\dim_{F'}\lieu\cap\liep'.
$$
If $\lambda$ is a character of $(\liel,L\cap K)$ defined over $F''/F'$, we set
$$
A_\lieq(\lambda)\;:=\;\mathcal R^{S_\lieq}(\lambda),
$$
which is always of finite length with infinitesimal character $\lambda+\rho(\lieu)$, and is defined over $F''$ by the discussion in section \ref{sec:equizuckerman}. By Theorem \ref{thm:equizuckerman} the $F''$-rational module $A_\lieq(\lambda)$ provides a model for the complex standard module of the same type that we denote by $A_\lieq(\lambda)_\CC$.

We say that $\lambda$ is in the {\em weakly good range} if for some $\theta$-stable Cartan $\lieh\subseteq\liel_\CC$ (defined over $\RR$) and a Borel $\lieb=\lieh\oplus\lien\subseteq\lieg_\CC$ satisfying $\lieu_\CC\subseteq\lien$, we have
\begin{equation}
\forall\alpha\in\Delta(\lien,\lieh):\quad {\rm Re}\langle\lambda+\rho(\lieu),\alpha\rangle\;\geq\;0,
\label{eq:weaklygood}
\end{equation}
for a choice of invariant bilinear form $\langle\cdot,\cdot\rangle$ on the complex space spanned by the roots which is positive definite on the real span. If strict inequality holds for all $\alpha$, then $\lambda$ is said to be in the {\em good range}.

Likewise $\lambda$ is in the {\em weakly fair range} if
\begin{equation}
\forall\alpha\in\Delta(\lieu,\lieh):\quad {\rm Re}\langle\lambda+\rho(\lieu),\alpha\rangle\;\geq\;0,
\label{eq:weaklyfair}
\end{equation}
and in the {\em fair range} if strict inequality holds for all $\alpha$.

Under the assumption
$$
K\;=\;K^0(L\cap K),
$$
we know that for $\lambda$ in the weakly good range the module $A_\lieq(\lambda)$ is absolutely irreducible or zero, and it is non-zero for $\lambda$ in the good range. If $\lambda_\CC$ is unitary and in the weakly fair range, then $A_\lieq(\lambda)_\CC$ is unitarizable. We refer to \cite{book_knappvogan1995} for further details.

\subsection{Rational models}

From now on we assume that $(\lieg,K)$ is attached to a connected reductive group $G$ over $F$. We call a character $\lambda$ of $\liel$ {\em integral} if it lifts to a rational character of the subgroup $L$ of $G$ corresponding to $\liel$. In the sequel dominance and positivity are understood with respect to $\lieb$. For integral dominant $\lambda$ we let $F(\lambda)$ denote a field of definition of the absolutely irreducible rational representation $M$ of $G$ of highest weight $\lambda$.

Recall that the center $Z(\lieg)$ of $U(\lieg)$ is defined over $F$ (cf.\ Proposition \ref{prop:centerrationality}). Therefore any character $\eta=\mu+\rho(\lien)$ of $Z(\lieg)$ is defined over its field of rationality $F(\eta)$, which agrees with the field of rationality of $\mu$ under the Galois action ${}_\Delta-$, cf.\ Proposition \ref{prop:hcgalois}.

 We write $F(\lambda,\eta)$ for the composite of the fields $F(\lambda)$ and $F(\eta)$, and likewise $F_0(\lambda)=F_0F(\lambda)$, etc., and fix an algebraic closure $\overline{F}/F$ containing all these fields whenever they are algebraic over $F$.

\begin{proposition}\label{prop:standardmodules}
With the above notation, we have
\begin{itemize}
\item[(a)] The standard module $A_\lieq(0)$ always admits a model over an $F_0\in\{F,F'\}$.
\item[(b)] For integral dominant $\lambda$, $A_\lieq(\lambda)$ is defined over $F_0(\lambda,\rho(\liel\cap\lieu))$.
\item[(c)] For integral $\lambda$ in the weakly good range, let $w\in W(\lieg_\CC,\lieh)$ rendering $w(\lambda)$ dominant. Then $A_\lieq(\lambda)$ is defined over $F_0(w(\lambda),\lambda+\rho(\lieu))$.
\end{itemize}
\end{proposition}

Theorem \ref{thm:realstandardmodules} below contain non-exhaustive lists of cases where $F_0=F$ and where $F_0$ is known to agree with the field of rationality of $A_\lieq(0)$ respectively.

\begin{proof}
Statement (a) is clear from our previous discussion.

Let $\eta$ denote a character of $Z(\lieg)$, assumed to be stable under $K$ and defined over a number field $F(\eta)$, and $M$ an absolutely irreducible rational $G$-module defined over a number field $F(\lambda)$. We consider the translation functor
$$
\mathcal T_M^\eta:\quad X\mapsto \pi_{\eta}(X\otimes_{F(\lambda,\eta)} M),
$$
where $X$ is a $Z(\lieg_{F(\lambda,\eta)})$-finite $(\lieg_{F(\lambda,\eta)},K_{F(\lambda,\eta)})$-module, stable under $K$, and $\pi_{\eta}$ is the projection to the $\eta$-primary component. The resulting functor is well defined, exact, and commutes with base change for obvious reasons.

Theorem 7.237 of \cite{book_knappvogan1995} carries over to this situation and shows that for $\eta=\lambda+\rho(\lieu)$, the resulting module
$$
\mathcal T_M^{\lambda+\rho(\lieu)}(A_\lieq(0)\otimes_{F_0} F_0(\lambda,\lambda+\rho(\lieu)))
$$
is an $F_0(\lambda,\lambda+\rho(\lieu))$-model of the standard module $A_{\lieq}(\lambda)$. Now $\lambda+\rho(\lien)$ is the infinitesimal character of $M$, hence defined over $F(\lambda)$. The identity
$$
\rho(\lien)-\rho(\lieu)\;=\;\rho(\liel\cap\lien)
$$
shows that $F_0(\lambda,\lambda+\rho(\lieu))=F_0(\lambda,\rho(\liel\cap\lien))$, and (b) follows.

As to (c), we may assume that $M$ admits $\lambda$ as an extremal weight, i.e.\ $M$ is of highest weight $w(\lambda)$ (for suitable $w\in W(\lieg_\CC,\lieh)$. Then (c) follows again by Theorem 7.237 of loc.\ cit.
\end{proof}

Consider any reductive group $K\subseteq K_\infty\subseteq G$ defined over $F$ with the property that $K_\infty$ is contained in $Z\cdot K$, where $Z\subseteq G$  denotes the center.
\begin{corollary}\label{cor:standardmodulecohomology}
For each integral dominant weight $\lambda$, and each absolutely irreducible rational $G$-representation $M$ of highest weight $\lambda$, defined over $F_0(\lambda)$, we have for any $q\in\ZZ$ a natural isomorphism
\begin{equation}
H^q(\lieg,K_\infty;A_\lieq(\lambda)\otimes M^\vee)\;\cong\;\Hom_{L\cap K}(\wedge^{q-S_{\lieq}}(\liel\cap\liep)/(\liez\cap\liep\cap\liek_\infty),F_0(\lambda)),
\label{eq:standardmodulecohomology}
\end{equation}
of $F_0(\lambda)$-vector spaces. Conversely every irreducible $(\lieg,K)$-module with non-vanishing $(\lieg,K)$-cohomology with coefficients in $M^\vee$ is isomorphic to $A_\lieq(\lambda)$ for some $\theta$-stable parabolic $\lieq$.
\end{corollary}

We remark that the $\theta$-stable parabolic $\lieq$ in the converse statement my be defined only over a finite extension of $F'$.

\begin{proof}
The identity \eqref{eq:standardmodulecohomology} is an immediate consequence of \cite[Theorem 3.3]{voganzuckerman1984} and the Homological Base Change Theorem. The rest of the claim follows by loc.\ cit.\ and uniqueness of models (Proposition \ref{prop:uniquemodels}).
\end{proof}

\begin{theorem}\label{thm:realstandardmodules}
Let $(\lieg,K)$ be a reductive pair attached to a connected reductive group $G$ over a number field $F$ and $F'/F$ be a quadratic extension, ramified at infinity, $\lieq\subseteq\lieg'$ a $\theta$-stable parabolic subalgebra defined over $F'$, subject to the following conditions:
\begin{itemize}
\item[(I)] For a fixed ramified archimedean place $v_\infty$, $K(F_{v_\infty})\subseteq G(F_{v_\infty})$ is a maximal compact subgroup.
\item[(II)] $K$ is an $F'/F$-semi-admissible model of $K(F_{v_\infty})$.
\item[(III)] $G(F_{v_\infty})$ is a quotient group of a direct product
$$
\prod_{i=1}^r H_{v_\infty,r},
$$
$$
\prod_{i=1}^r M_{v_\infty,r},
$$
accordingly, and $\lieq_{F_{v_\infty}'}$ decomposes compatibly into a direct sum
$$
\lieq_{F_{v_\infty}'}\;=\;\oplus_{i=1}^r \lieq_{v_\infty,i},
$$
and $\lambda_{F_{v_\infty}'}=\otimes_{i=1}^r\lambda_{v_\infty,i}$.
\item[(IV)] For each pair $(H_{v_\infty,i},\lieq_{v_\infty,i})$ in (D) one of the following cases applies:
\begin{itemize}
\item[(i)] 
\begin{equation}
H_{v_\infty,i}\in\{\SL_{2n+1}(\RR),\GL_n(\RR),\GL_n(\CC)\};
\label{eq:glist1}
\end{equation}
\item[(ii)] $H_{v_\infty,i}\in\{\SL_{2n}(\RR),\SL_{2n+1}(\CC),{\rm U}(p,q),\SO^*(2n),\SO(p,q),\Oo(p,q)\}$;
\item[(iii)] $H^{\rm der}$ is one of the following real Lie groups
\begin{itemize}
\item[($A_n$)] $\SO(2,1),\SU(2),\SO(3)$;
\item[($B_n$)] ${\rm Spin}(2p,2q+1)$, $2\mid p$;
\item[($C_n$)] ${\rm Sp}(p,q)$;
\item[($D_{2n}$)] ${\rm Spin}(2p,2q)$, $2\mid p$, $2\mid q$; $\SO(2p,2q)$, $2\mid p+q$; $\overline{\SO}(2p,2q)$, $2\mid p$, $2\mid q$; $\overline{\SO}^*(4n)$ when disconnected; ${\rm PSO}(2p, 2q)$, $2\mid p+q$; ${\rm PSO}^*(4n)$; all groups locally isomorphic to $\SO(2p+1,2q+1)$, $2\nmid p+q$;
\end{itemize}
\item[(iv)] $H_{v_\infty,i}^{\rm der}$ is any real form of the adjoint group of type $B_n,C_n,E_7$;
\item[(v)] $H_{v_\infty,i}^{\rm der}$ is any real form of type $G_2,F_4,E_8$;
\item[(vi)] $H_{v_\infty,i}^{\rm der}$ is a complex group of type $A_1,B_n,C_n,D_{2n},G_2,F_4,E_7,E_8$;
\item[(vii)] $H_{v_\infty,i}=\tilde{H}_{v_\infty,i}\times\tilde{H}_{v_\infty,i}$ for a real reductive group $\tilde{H}_{v_\infty,i}$, $K(F_{v_\infty})$ intersected with each component gives the same maximal compact subgroup,
$$
\lieq_{v_\infty,i}\;=\;
\begin{cases}
\tilde{\lieq}_{v_\infty,i}\times\tilde{\lieq}_{v_\infty,i},\;\text{or}\\
\tilde{\lieq}_{v_\infty,i}\times\overline{\tilde{\lieq}}_{v_\infty,i},
\end{cases}
$$
and
$$
\lambda_{v_\infty,i}\;=\;
\begin{cases}
\tilde{\lambda}_{v_\infty,i}\otimes\tilde{\lambda}_{v_\infty,i},\;\text{or}\\
\tilde{\lambda}_{v_\infty,i}\otimes\overline{\tilde{\lambda}}_{v_\infty,i},
\end{cases}
$$
accordingly.
\item[(viii)] ${\mathscr D}(M_{v_\infty,i}^0)$ is of adjoint type.
\end{itemize}
\end{itemize}
Then
\begin{itemize}
\item[(a)]
$A_\lieq(0)$ is defined over its field of rationality $F_0\in\{F,F'\}$.
\item[(b)] If $G$ is quasi-split and $\lieq$ a Borel subalgebra, then for any integral dominant weight $\lambda$ the standard module $A_\lieq(\lambda)$ is defined over its field of rationality $F_0(\lambda)$ as well.
\item[(c)]
Regardless of whether $G$ is quasi-split or not, if only the cases (i), (iii), (iv), (v) and (vi) occur, and if furthermore $\lieq_{v_\infty,i}$ is a Borel in case (i), then $F_0=F$.
\end{itemize}
\end{theorem}

For the definition of the groups $\overline{\SO}(2p,2q)$ and $\overline{\SO}^*(4n)$ we refer the reader to \cite{adams2014}. We remark that in the lists (ii) and (iii) the case $q=0$ is allowed (subject to the imposed parity conditions and the exception in case of $\SU(p,q)$).

\begin{proof}
We first remark that since $\lambda=0$ satisfies \eqref{eq:weaklygood} and \eqref{eq:weaklyfair}, the module $A_\lieq(0)_{F_{v_\infty}'}$ is unitarizable, irreducible or zero and has trivial central character. Therefore $A_\lieq(0)$ is absolutely irreducible or zero.

If there is no isomorphism
\begin{equation}
A_\lieq(0)\;\cong\;\overline{A_\lieq(0)},
\label{eq:standardmoduleconjugateiso}
\end{equation}
i.e.\ if $A_\lieq(0)$ is not self-dual, it does not admit a model over $F$ by Proposition \ref{prop:fsreal}, hence $F_0=F'$ in Proposition \ref{prop:standardmodules} and (a) follows in this case. Recall that if $G$ is quasi-split, the field of rationality of absolutely irreducible rational $G$-modules agrees with the field of definition, and (b) follows.

Hence we may assume there is an isomorphism \eqref{eq:standardmoduleconjugateiso}. Then the same is true over $F_{v_\infty}'$.

By construction the bottom layer
$$
B\;\subseteq\;A_{\lieq}(0)_{F'}
$$
is an irreducible $K'$-module, which is generated by an absolutely irreducible ${K'}^0$-module of highest weight $2\rho(\lieu\cap\liep)$. In particular $B$ is a self-dual $K'$-module over $F'$.

By Proposition \ref{prop:minimalktypedescent} and hypothesis (II) we have
$$
\FS_F(A_\lieq(0)_{F'})
\;=\;
\FS_{F_{v_\infty}}(A_\lieq(0)_{F_{v_\infty}'})
\;=\;
\FS_{F_{v_\infty}}(B_{F_{v_\infty}'}).
$$
Therefore, to prove statement (a) it suffices to show that
$$
\FS_{F_{v_\infty}}(A_\lieq(0)_{F_{v_\infty}'})
\;=\;-1
$$
does not occur.

By the multiplicativity of Frobenius-Schur indicators for products of groups and outer products of modules we may assume without loss of generality that $r = 1$, and treat each case in (IV) separately.

In cases (i) and (ii), if $H_{v_\infty,i}\neq\SL_{n}(\RR),\SL_{2n+1}(\CC),\SU(p,q)$, the maximal compact subgroup $M_{v_\infty,i}$ of $H_{v_\infty,i}$ is isomorphic to a product of compact unitary groups, and (full) orthogonal groups. The maximal compact subgroup of $\SL_{n}(\RR)$ is $\SO(n)$, that of $\SL_{2n+1}(\CC)$ is $\SU(2n+1)$. Therefore, the claim follows by Proposition \ref{prop:unitarygroups}.

We remark that for even $n$ in case (i), Theorem A of \cite{prasadramakrishnan2012} together with Theorem \ref{thm:fsunitary} also shows that the indicator is $1$, and the case of odd $n$ may be settled by the same way by consideration of the representation of the local Weil group attached to $A_\lieq(0)_{F_{v_\infty}}$.

The cases (iii), (iv), (v) and (vi), are all included in the list on page 2137 of \cite{adams2014}, hence by Theorem 5.8 of loc.\ cit.\ the Frobenius-Schur indicator $\FS_{F_{v_\infty}}(A_\lieq(0)_{F_{v_\infty}'})$ agrees with the central character evaluated at a specific element. Since the central character is trivial in our case, the claim follows again.

Case (vii) follows trivially in the first subcase, since the indicator is a square, hence equals $1$. The same argument applies to the second subcase: The Frobenius-Schur indicator of a module and its complex conjugate agree by definition, and by the Homological Base Change Theorem the two resulting standard modules on the two direct factors under consideration are complex conjugates of each other.

Case (viii) is a consequence of Remark \ref{rmk:adjointderivedgroup} and Lemma \ref{lem:extendenboreltits}, or Proposition \ref{prop:unconditionaladmissiblegroups}.

Statement (b) is a consequence of Proposition \ref{prop:standardmodules}.

As to statement (c), we claim that the hypothesis imply that $A_\lieq(0)_{F_{v_\infty}'}$ is self-dual.

In case (i), $A_\lieq(0)_{F_{v_\infty}'}$ is be self-dual if $\lieq$ is a Borel, because in this case $\overline{\lieq}$ is $K$-conjugate to $\lieq$. In all other cases (cases (iii), (iv), (v) and (vi)), Corollary 4.7 and Corollary 4.8 of \cite{adams2014} applies and shows that $A_\lieq(0)_{F_{v_\infty}'}$ is again self-dual. This concludes the proof of the Theorem.
\end{proof}

\section{Applications to Automorphic Representations}\label{sec:applications}

In this section we show the existence of global rational structures on spaces of cusp forms for several types of groups. In the case of $\GL_n$ we determine the optimal fields of definition and show that they agree with the fields of rationality in this case. For general reductive groups the situation is more delicate and we determine bounds for the degree of the field of definition over the field of rationality in the general case. We also discuss groups of Hermitian type.

We limit our treatment here to cuspidal representations, even though our methods admit generalizations to residual representations and Eisenstein series.

\subsection{Automorphic representations of reductive groups}

From now on we assume $G$ to be a connected linear reductive group over $\QQ$ and fix a subgroup $K\subseteq G$ defined over a finite extension field $\QQ_K/\QQ$ admitting a real place $v_\infty$ that we suppose to be fixed. Then $v_\infty$ corresponds to a fixed embedding $\QQ_K\to \RR$, and we suppose that $K(\RR)\subseteq G(\RR)$ is maximal compact. Recall that we assume implicitly that the connected component $K^0$ of the identity of $K$ is defined over $\QQ_K$ as well. We assume fixed a quadratic extension $\QQ_K'/\QQ_K$ in which $v_\infty$ ramifies. Write $Z$ for the center of $G$ and we let $K_\infty$ denote a closed subgroup of $G$ between $K$ and $Z\cdot K$ defined over $\QQ_K$.

We write $\Adeles_F$ for the ring of ad\`eles of a number field $F$ (a finite extension of $\QQ$) and let $\Adeles_F^{(\infty)}=F\otimes_\QQ\Adeles_\QQ$ denote the finite part of $\Adeles_F$. For notational simplicity we set $\Adeles:=\Adeles_\QQ$.

An irreducible automorphic representation of $G$ is an irreducible constituent $\Pi$ of the Hilbert space
\begin{equation}
L^2(G(\QQ)\backslash G(\Adeles),\omega)
\label{eq:l2formsomega}
\end{equation}
of forms transforming under $Z(\Adeles)$ according to an id\`ele class character $\omega:Z(\QQ)\backslash Z(\Adeles)\to\CC^\times$ and which are square integrable modulo center. We denote the subspace of cuspidal forms by
\begin{equation}
L_0^2(G(\QQ)\backslash G(\Adeles),\omega).
\label{eq:l2cuspformsomega}
\end{equation}
This space turns out to be a direct summand in \eqref{eq:l2formsomega}.

For such an irreducible $\Pi$ we consider its factorization $\Pi=\Pi_\infty\otimes\Pi_{(\infty)}$ into a representation of $G(\RR)$ and of $G(\Adeles^{(\infty)})$. Since we are only concerned with (essentially) unitary representations, we may pass from $\Pi_\infty$ to the space of smooth vectors $\Pi_\infty^{\rm CW}$ which is an admissible Fr\'echet representation of moderate growth of $G(\RR)$, i.e.\ $\Pi_\infty^{\rm CW}$ is the Casselman-Wallach completion of the underlying $(\lieg,K)$-module $\Pi_\infty^{(K)}$ of $K$-finite vectors in $\Pi_{(\infty)}$ and carries an action of $\lieg$ as well as of $G(\RR)$. Then $\Pi_\infty^{\rm CW}\otimes\Pi_{(\infty)}\subseteq\Pi$ is a dense subspace of $\Pi$, and $\Pi_\infty^{\rm CW}\otimes\Pi_{(\infty)}$ has the virtue of being canonically defined in terms of the subspace of $K$-finite vectors.

Our goal is to show the existence of rational structures on suitable subspaces of \eqref{eq:l2cuspformsomega}, since the full space \eqref{eq:l2cuspformsomega} turns out to be too large, even if $\omega$ is of type \lq{}$A_0$\rq{} in the generalized sense of \cite{buzzardgee2011}.

We consider for each compact open $K_{(\infty)}\subseteq G(\Adeles^{(\infty)})$ the locally symmetric space
$$
\mathscr
 X_G(K_{(\infty)})\;:=\;G(\QQ)\backslash G(\Adeles)/K_{(\infty)}\cdot K_\infty(\RR).
$$
Given any rational representation $M$ of $G$, defined over a field $\QQ_M\subseteq\CC$, we have for any intermediate field $\QQ_M\subseteq E\subseteq\CC$ an associated sheaf $\underline{M}(E)$ of $E$-vector spaces on $\mathscr X_G(K_{(\infty)})$. Whenever $K_{(\infty)}$ is sufficiently small and gives rise to a torsion free arithmetic group (modulo the central subgroup $K_\infty(\RR)\cap Z(\RR)$), $\mathscr X_G(K_{(\infty)})$ is a manifold. We henceforth always assume that this condition is satisfied.

Several cohomology theories are of interest to us. We have the standard sheaf cohomology
\begin{equation}
H^\bullet(\mathscr X_G(K_{(\infty)});\underline{M^\vee}(E)),
\label{eq:singularcohomology}
\end{equation}
which agrees with the singular cohomology for the local system associated to the representation $M(E)$. Similarly we have cohomology with compact supports
$$
H_{\rm c}^\bullet(\mathscr X_G(K_{(\infty)});\underline{M^\vee}(E)),
$$
and its canonical image in \eqref{eq:singularcohomology} is known as inner cohomology
\begin{equation}
H_{\rm !}^\bullet(\mathscr X_G(K_{(\infty)});\underline{M^\vee}(E)).
\label{eq:innercohomology}
\end{equation}
The existence of the Borel-Serre compactification implies that we have a distinguished triangle
$$
H_{\rm !}^\bullet(\mathscr X_G(K_{(\infty)});\underline{M^\vee}(E))\to
H^\bullet(\mathscr X_G(K_{(\infty)});\underline{M^\vee}(E))\to
H^\bullet(\partial\mathscr X_G(K_{(\infty)});\underline{M^\vee}(E)).
$$
This characterizes inner cohomology as the kernel of the restriction map to the boundary $\partial\mathscr X_G(K_{(\infty)})$ of the Borel-Serre compactification.

For $E=\CC$ the space \eqref{eq:singularcohomology} admits a description via automorphic representations of $G$. By Theorem 18 in \cite{franke1998} and Theorem 2.3 in \cite{frankeschwermer1998}, significantly generalizing pioneering work of Harder \cite{harder1975,harder1987}, we have a decomposition
\begin{equation}
\varinjlim_{K_{(\infty)}}H^\bullet(\mathscr X_G(K_{(\infty)});\underline{M^\vee}(\CC))\;\cong\;
\widehat{\bigoplus_{\{P\}}}
\widehat{\bigoplus_{\varphi\in\Phi_{M,\{P\}}}}
H^\bullet(\lieg,K_\infty;\mathcal A_{M,\{P\},\varphi}\otimes M^\vee)(\chi_{M}),
\label{eq:cohomologyspectraldata}
\end{equation}
where the outer sum ranges over associate classes $\{P\}$ of $\QQ$-rational parabolic subgroups $P\subseteq G$ and in the inner sum $\varphi=(\varphi_Q)_{Q\in\{P\}}$ ranges over the set $\Phi_{M,\{P\}}$ of classes $\varphi_Q$ of associate irreducible cuspidal automorphic representations of Levi components of parabolics $Q\in\{P\}$, subject to certain elementary conditions, most importantly that the infinitesimal character of $\varphi_Q$ occurs in $M^\vee$. Likewise $(\chi_{M})$ denotes a suitable twist of the Hecke action depending on $M$, in order to render the isomorphism \eqref{eq:cohomologyspectraldata} $G(\Adeles^{(\infty)})$-equivariant.

Finally $\mathcal A_{M,\{P\},\varphi}$ is the subspace of the space of smooth automorphic forms on $G(\Adeles)$ with constant terms supported at the various $Q\in\{P\}$ and contained in the sum of the elements of the finite set $\varphi_Q$. Equivalently the space $\mathcal A_{M,\{P\},\varphi}$ is spanned by all residues and derivatives of cuspidal Eisenstein series associated to the cuspidal datum $(\{P\},\varphi)$.

The isomorphism \eqref{eq:cohomologyspectraldata} is a topological incarnation of Langlands' decomposition of the space of automorphic forms \eqref{eq:l2formsomega} into the Hilbert direct sum of spaces $\mathcal A_{M,\{P\},\varphi}$ (cf.\ \cite{langlands1976,moeglinwaldspurger1994}).

By \eqref{eq:cohomologyspectraldata} an irreducible cuspidal automorphic representation $\Pi$ of $G$ contributes to \eqref{eq:singularcohomology} in degree $q$ for some sufficiently small $K_{(\infty)}$ if and only if
\begin{equation}
H^q(\lieg_\CC,K_\infty(\CC);\Pi_\infty^{\rm CW}\otimes M^\vee(\CC))\neq 0.
\label{eq:gkcohomology}
\end{equation}
We remark that the space \eqref{eq:gkcohomology} does not change if we replace $\Pi_\infty^{\rm CW}$ by the subspace of $K$-finite vectors $\Pi_\infty^{(K)}$. If condition \eqref{eq:gkcohomology} is satisfied, we call $\Pi$ {\em cohomological with respect to $M$}, or simply {\em cohomological}.

As outlined in \eqref{eq:cohomologyspectraldata}, in general the cuspidal, residual and continuous spectrum of $G(\Adeles)$ contribute to \eqref{eq:singularcohomology}, whereas only the discrete spectrum (i.e.\ residual and cuspidal) contributes to \eqref{eq:innercohomology}.

The contribution coming from cuspidal automorphic representations (the summand $\{P\}=\{G\}$ on the right hand side of \eqref{eq:cohomologyspectraldata}) gives rise to the cuspidal cohomology
\begin{equation}
H_{\rm cusp}^\bullet(\mathscr X_G(K_{(\infty)});\underline{M^\vee}(\CC)).
\label{eq:cuspidalcohomology}
\end{equation}
Cuspidal cohomology is a subspace of inner cohomology \eqref{eq:innercohomology}.

In general it is unclear if the individual summands of \eqref{eq:cohomologyspectraldata} are rational subspaces of \eqref{eq:singularcohomology}, i.e.\ preserved under the natural action of $\Aut(\CC/\QQ)$. Currently this is only known for $G=\res_{F/\QQ}\GL_n$ (cf.\ Theorem 20 in \cite{franke1998}), which is a significant generalization of Clozel's rationality result \cite{clozel1990} for cuspidal cohomology in this case. For a rationality statement valid for cuspidal cohomology for general reductive $G$ we refer to Proposition \ref{prop:cuspidalrationality} below.

\subsection{Rationality considerations for Hecke modules}

We remark that by definition inner cohomology is a $\QQ_M$-rational subspace of the full cohomology \eqref{eq:singularcohomology}. As outlined above, the same is expected to be the case for the cuspidal cohomology \eqref{eq:cuspidalcohomology}. In general this is unknown and intimately related to rationality properties of cohomological cusp forms. In general we have

\begin{proposition}[Li-Schwermer]\label{prop:cuspidalrationality}
Assume that every absolutely irreducible constituent of $M$ is of {\em regular} highest weight. Then \eqref{eq:cuspidalcohomology} is a $\QQ_M$-rational subspace of \eqref{eq:singularcohomology} and in particular carries a natural $\QQ_M$-structure. Furthermore, any cuspidal representation contributing to \eqref{eq:singularcohomology} is tempered at infinity.
\end{proposition}

\begin{remark}
The hypothesis of Proposition \ref{prop:cuspidalrationality} implies that Franke's Eisenstein spectral sequence in Theorem 19 of \cite{franke1998} degenerates.
\end{remark}

\begin{proof}
The last statement is contained in Proposition 5.2 in \cite{lischwermer2004}. For the first part, since any highest weight of $M$ is regular, an elementary consequence of the Vogan-Zuckerman classification of the representations $\Pi_\infty^{\rm CW}$ of $G(\RR)$ satisfying \eqref{eq:gkcohomology} (cf.\ \cite{voganzuckerman1984}) is that $\Pi_\infty^{\rm CW}$ must be tempered (i.e.\ cohomologically induced from a $\theta$-stable Borel, cf.\ Proposition 4.2 in \cite{lischwermer2004}). Therefore, any representation $\Pi$ contributing to inner cohomology must be cuspidal by \cite{wallach1984}, since it contributes to the discrete spectrum and is tempered. Therefore cuspidal cohomology agrees with inner cohomology in this case and the claim follows.
\end{proof}

We remark that if we consider the space
\begin{equation}
L_0^2(G(\QQ)\backslash G(\Adeles);M(\CC))
\label{eq:gcuspforms}
\end{equation}
generated by cusp forms of $G$ which generate irreducible representations contributing to cohomology with coefficients in $M^\vee(\CC)$, then we have an isomorphism
$$
H^\bullet(\lieg_\CC,K_\infty(\CC);L_0^2(G(\QQ)\backslash G(\Adeles)^{(K)};M(\CC))\otimes M^\vee(\CC))
$$
\begin{equation}
\to
\varinjlim_{K_{(\infty)}}H_{\rm cusp}^\bullet(\mathscr X_G(K_{(\infty)});\underline{M^\vee}(\CC))
\label{eq:cuspidalcohomologyiso}
\end{equation}
of $G(\Adeles^{(\infty)})$-modules. Under the assumptions of Proposition \ref{prop:cuspidalrationality}, the right hand side is defined over $\QQ_M$, and we obtain a $\QQ_M$-structure on the $(\lieg,K)$-cohomology of \eqref{eq:gcuspforms}. One of our aims is to extend this rational structure to the entire space of $K$-finite cusp forms in \eqref{eq:gcuspforms}.

Due to the possible variation of the infinity component of $\Pi$ contributing to \eqref{eq:gcuspforms}, we are led to consider finer multiplicities and to prove the existence of rational structures on irreducible representations. The latter will be achieved in Theorems \ref{thm:finitedefinition} and \ref{thm:gglobalrational} below.

For each irreducible admissible representation $\Pi$ of $G(\Adeles)$ we consider its finite multiplicity $m_{L^2_0}(\Pi)$ in \eqref{eq:gcuspforms}. In each degree $q$ we have a cohomological multiplicity
$$
m_{\rm coh}^{M,q}(\Pi_\infty)\;:=\;\dim
H^q(\lieg_\CC,K_\infty(\CC);\Pi_\infty^{\rm CW}\otimes M^\vee(\CC)).
$$
Then, by the above isomorphism, the multiplicity $m_{\rm cusp}(\Pi_{(\infty)})$ of the finite part contributing to the inductive limit of the cuspidal cohomology spaces \eqref{eq:cuspidalcohomology} of finite level in degree $q$ is explicitly given by
\begin{equation}
m_{\rm cusp}^{M,q}(\Pi_{(\infty)})\;=\;
\sum_{\Pi_\infty^{\rm CW}}
m_{\rm coh}^{M,q}(\Pi_\infty^{\rm CW})\cdot
m_{L^2_0}(\Pi_\infty^{\rm CW}\otimes\Pi_{(\infty)}),
\label{eq:cuspidalmultiplicityformula}
\end{equation}
where $\Pi_{\infty}^{\rm CW}$ runs over isomorphism classes of irreducible Casselman-Wallach representations of $G(\RR)$. In this optic we may attach to $\Pi_{(\infty)}$ a set of isomorphism classes
\begin{equation}
\Phi_{\infty}^{\rm cusp}(\Pi_{(\infty)})\;:=\;
\{\Pi_\infty^{\rm CW}\mid
\exists q:
m_{\rm coh}^{M,q}(\Pi_\infty^{\rm CW})\cdot
m_{L^2_0}(\Pi_\infty^{\rm CW}\otimes\Pi_{(\infty)})\neq 0\}/\sim.
\label{eq:cohomologicallpacket}
\end{equation}

For each $\Pi_\infty^{\rm CW}$ occuring in \eqref{eq:cohomologicallpacket}, there is a minimal $q_{\rm min}$ and a maximal $q_{\rm max}$ degree in which $\Pi_\infty^{\rm CW}$ contributes. By Vogan-Zuckerman \cite{voganzuckerman1984} (cf.\ Corollary \ref{cor:standardmodulecohomology}) we have
$$
m_{\rm coh}^{M,q_{\rm min}}(\Pi_\infty^{\rm CW})\;=\;
m_{\rm coh}^{M,q_{\rm max}}(\Pi_\infty^{\rm CW})\;=\;1,
$$
and for each $q_{\rm min}\leq q\leq q_{\rm max}$,
$$
m_{\rm coh}^{M,q}(\Pi_\infty^{\rm CW})\;=\;
{{q_{\rm max}-q_{\rm min}}\choose{q-q_{\rm min}}}\;>\;0,
$$
since the $(\lieg,K_\infty)$-cohomology is an exterior algebra. Furthermore the length of the range of degrees where $\Pi_\infty^{\rm CW}$ contributes is minimal if $\Pi_\infty^{\rm CW}$ is tempered.

We remark that formula \eqref{eq:cuspidalmultiplicityformula} does in general not extend to inner cohomology, since the discrete spectrum fails to inject into cohomology in general.

\subsection{Rational structures on irreducible representations}

The previous discussion shows that, in order to extend the natural rational structure on cuspidal cohomology to a global rational structure, we are led to consider rational structures on irreducible $G(\Adeles^{(\infty)})$-modules. The only known case where we have such a rational structure (and even an optimal one) is $G=\res_{F/\QQ}\GL(n)$ for a number field $F$ due to Clozel \cite{clozel1990}.

Shin and Templier showed in Proposition 2.15 in \cite{shintemplier2014} that for general $G$ and general cohomological cuspidal $\Pi$ the field of rationality $\QQ(\Pi)$ of $\Pi^{(\infty)}$ is a number field. However They did not construct a rational structure on the latter space defined over a number field. The authors of loc.\ cit.\ exhibit only a $\overline{\QQ}$-rational structure. Our first task is to prove the existence of a rational structure defined over a finite extension of the field of rationality $\QQ(\Pi)$.

For any field extension $E/\QQ$ we consider the Hecke algebra $\mathcal H_E(G(\Adeles^{(\infty)});K_{(\infty)})$ of $E$-valued bi-invariant compactly supported functions. For any simple Hecke module $S$ over $E$ and any finite-dimensional Hecke module $H$ we write $H[S]$ for the $S$-isotypic component of $H$, which is the linear span of all homomorphic images of $S$ inside $H$. We let $m(S;H)$ denote the {\em length} of $H[S]$, which is the multiplicity of $S$ in $H$ as a {\em submodule}. We remark that this notion of multiplicity is not additive in short exact sequences since we do not suppose $H$ to be semisimple.

\begin{remark}\label{rmk:qminqpi}
For $\GL(n)$ multiplicity one implies $\QQ_M\subseteq\QQ(\Pi)$ whenever $M$ is absolutely irreducible. For general $G$ the field $\QQ_M$ is not a subfield of $\QQ(\Pi)$.
\end{remark}

We write $\QQ_M(\Pi)$ for the composite $\QQ_M\cdot\QQ(\Pi)$, and likewise for a field of definition $\QQ(\Pi)^{\rm rat}$ of $\Pi_{(\infty)}$.

\begin{theorem}\label{thm:finitedefinition}
Let $\Pi$ be any irreducible cohomological cuspidal automorphic representation of $G(\Adeles)$. Then
\begin{itemize}
\item[(a)] The finite part $\Pi_{(\infty)}$ is defined over a finite extension $\QQ(\Pi)^{\rm rat}/\QQ(\Pi)$.
\item[(b)] Any $\QQ(\Pi)^{\rm rat}$-rational structure on $\Pi_{(\infty)}$ is unique up to complex homotheties.
\item[(c)] For any rational $G$-representation $M$ defined over $\QQ_M$, if we assume that $[\QQ_M(\Pi)^{\rm rat}:\QQ_M(\Pi)]$ is minimal, then
\begin{equation}
[\QQ_M(\Pi)^{\rm rat}:\QQ_M(\Pi)]\;\mid\;
m\left(\Pi_{(\infty)}^{K_{(\infty)}};H_{\bullet}^q(\mathscr X_G(K_{(\infty)});\underline{M^\vee}(\CC))\right)
\label{eq:degreemultiplicity}
\end{equation}
for every degree $q$ and any sufficiently small compact open $K_{(\infty)}$, and $\bullet\in\{-,!\}$ and also $\bullet={\rm cusp}$ if cuspidal cohomology \eqref{eq:cuspidalcohomology} admits a $\QQ_M$-structure.
\item[(d)] Under the hypotheses of (c), if $\Pi_{(\infty),\QQ(\Pi)^{\rm rat}}$ is a model for $\Pi_{(\infty)}$ over $\QQ(\Pi)^{\rm rat}$, then for any subfield $\QQ_M(\Pi)\subseteq E\subseteq\QQ_M(\Pi)^{\rm rat}$ or $\QQ_M\subseteq E\subseteq\QQ_M(\Pi)$
$$
\res_{\QQ_M(\Pi)^{\rm rat}/E}\Pi_{(\infty),\QQ_M(\Pi)^{\rm rat}}
$$
is an irreducible $G(\Adeles_{(\infty)})$-module over $E$ which embeds into
\begin{equation}
H_{\bullet}^q(\mathscr X_G(K_{(\infty)});\underline{M^\vee}(E)).
\label{eq:hbullet}
\end{equation}
\end{itemize}
\end{theorem}

We remark that cohomological cuspidal representations $\Pi$ as in Theorem \ref{thm:finitedefinition} are known to be $C$-algebraic in the terminology of \cite{buzzardgee2011}, cf.\ Lemma 2.14 in \cite{shintemplier2014}.

\begin{remark}
The relation \eqref{eq:degreemultiplicity} is a non-empty statement if and only if condition \eqref{eq:finitekinvariants} below is satisfied and if furthermore $\Pi$ contributes to inner (resp.\ cuspidal) cohomology in degree $q$ with coefficients in $M$.
\end{remark}

\begin{remark}
In (c) and (d) we do not assume $M$ to be irreducible. By passing from an absolutely irreducible $M$ for which $\Pi$ contributes to an irreducible submodule of $\res_{\QQ_M/\QQ}M$, we obtain a local system defined over $\QQ$ and the contribution of $M$ in (c) and (d) disappears.
\end{remark}

\begin{remark}
It is not clear that
\begin{equation}
\QQ(\Pi)^{\rm rat}\;=\;\QQ(\Pi)
\label{eq:defratequality}
\end{equation}
in general. Such a statement cannot be deduced from Proposition \ref{prop:cuspidalrationality} for regular cohomological weights by Galois descent, since multiplicity one fails in general, and we currently do not have enough control over preservation of automorphic representations under Galois twists.
\end{remark}

\begin{remark}
For representations which are unramified everywhere \eqref{eq:defratequality} can be achieved (cf.\ Lemma 2.2.3 and Corollary 2.2.4 in \cite{buzzardgee2011}). In particular any possible failure of the equality \eqref{eq:defratequality} would be attributable to ramification of $\Pi_{(\infty)}$.
\end{remark}

\begin{remark}
For $G=\res_{F/\QQ}\GL(n)$, equality \eqref{eq:defratequality} is known and was proven by Clozel who relied on a rational classification of admissible irreducibles at the non-archimedean places, which Clozel achieved in \cite{clozel1990}. We remark that in this case multiplicity one is known for $G$, and therefore Theorem \ref{thm:finitedefinition} provides a direct proof for Clozel's result that the field of definition agrees with the field of rationality in this case without recourse to a rational classification of local representations at finite primes. Furthermore we have $\QQ_M\subseteq\QQ(\Pi)$ in this case.
\end{remark}

\begin{proof}[Proof of Theorem \ref{thm:finitedefinition}]
The proof is a refinement of the argument in the proofs of \cite[Proposition 3.13]{clozel1990} and \cite[Proposition 2.15, (ii)]{shintemplier2014}.

Given $\Pi$ as in the theorem, we choose a sufficiently small compact open $K_{(\infty)}\subseteq G(\Adeles^{(\infty)})$ satisfying
\begin{equation}
\Pi^{K_{(\infty)}}\neq 0.
\label{eq:finitekinvariants}
\end{equation}
Then $\Pi_{(\infty)}^{K_{(\infty)}}$ is an irreducible module for the complex Hecke algebra $\mathcal H_\CC(G(\Adeles^{(\infty)}),K_{(\infty)})$ of compactly supported bi-$K_{(\infty)}$-invariant $\CC$-valued functions on $G(\Adeles^{(\infty)})$. By our hypothesis it occurs as a direct summand in \eqref{eq:hbullet} for some rational $G$-representation $M$ and $E=\CC$. We do not assume that $M$ is absolutely irreducible. In case $\bullet=\lq{}!\rq{}$ we remark that the $\QQ_M$-structure on \eqref{eq:singularcohomology} descends to inner cohomology which therefore inherits a $\QQ_M$-structure as Hecke module for the $\QQ_M$-rational Hecke algebra of $\QQ_M$-rational bi-invariant functions.

We consider the $\Pi_{(\infty)}^{K_{(\infty)}}$-isotypic component
\begin{equation}
H_{\bullet}^q(\mathscr X_G(K_{(\infty)});\underline{M^\vee}(\CC))[
\Pi_{(\infty)}^{K_{(\infty)}}],
\label{eq:bulletisotypiccomponent}
\end{equation}
in a degree $q$ where $\Pi$ contributes non-trivially. Then the space \eqref{eq:bulletisotypiccomponent} is stable under the Galois action of $\Aut(\CC/\QQ_M(\Pi))$ and stable under the Hecke action which commutes with the Galois action because \eqref{eq:hbullet} is a $\QQ_M$-rational Hecke module by our hypothesis. Therefore \eqref{eq:bulletisotypiccomponent} is a $\QQ_M(\Pi)$-rational subspace of \eqref{eq:hbullet} with $\QQ_M(\Pi)$-rational Hecke action.

This claim is true for any compact open $K_{(\infty)}$ satisfying \eqref{eq:finitekinvariants}. Since these rational stuctures are compatible with the restriction maps, passing to the direct limit provides us with a $\QQ_M(\Pi)$-rational structure on the $G(\Adeles^{(\infty)})$-module
\begin{equation}
\varinjlim_{K_{(\infty)}}
H_{\bullet}^q(\mathscr X_G(K_{(\infty)});\underline{M^\vee}(\CC))[
\Pi_{(\infty)}^{K_{(\infty)}}]\;\cong\;
\left(\varinjlim_{K_{(\infty)}}H_{\bullet}^q(\mathscr X_G(K_{(\infty)});\underline{M^\vee}(\CC))\right)[
\Pi_{(\infty)}].
\label{eq:limbulletisotypiccomponent}
\end{equation}
We fix a $\QQ_M(\Pi)$-rational model $H_{\QQ_M(\Pi)}$ of \eqref{eq:limbulletisotypiccomponent}.

The irreducible finite-dimensional Hecke module $\Pi_{(\infty)}^{K_{(\infty)}}$ admits a model $\Pi_{(\infty),\QQ(\Pi)^{\rm rat}}^{K_{(\infty)}}$ over a finite extension $\QQ_M(\Pi)^{\rm rat}$ of $\QQ_M(\Pi)$, a priori possibly depending on $K_{(\infty)}$, and a $\QQ(\Pi)^{\rm rat}$-rational Hecke equivariant embedding
$$
i:\quad \Pi_{(\infty),\QQ(\Pi)^{\rm rat}}^{K_{(\infty)}}\;\to\;
H_{\bullet}^q(\mathscr X_G(K_{(\infty)});\underline{M^\vee}(\QQ(\Pi)^{\rm rat}))[\Pi_{(\infty)}^{K_{(\infty)}}],
$$
for a fixed $K_{(\infty)}$ satisfying \eqref{eq:finitekinvariants}. The map
$$
M\;\mapsto\;M^{K_{(\infty)}}
$$
induces a Galois equivariant lattice isomorphism between the lattice of $G(\Adeles^{(\infty)})$-submodules of \eqref{eq:limbulletisotypiccomponent} and the lattice of Hecke submodules of \eqref{eq:bulletisotypiccomponent}. Therefore there is a unique submodule $M$ of \eqref{eq:limbulletisotypiccomponent} satisfying
$$
M^{K_{(\infty)}}\;=\;i\left(\Pi_{(\infty)}^{K_{(\infty)}}\right).
$$
By Galois descent for vector spaces, this shows that $\Pi_{(\infty)}$ admits a model over $\QQ_M(\Pi)^{\rm rat}$ and (a) follows.

This argument reduces the statements (b) to (d) to equivalent statements about finite-dimensional Hecke modules of level $K_{(\infty)}$.

Uniqueness of the rational structure thus obtained up to complex homotheties follows from Schur's Lemma for irreducible complex $G(\Adeles^{(\infty)})$-modules (or equivalently finite-dimensional Hecke modules) as in the proof of Proposition \ref{prop:uniquemodels}. This proves (b).

As to (c), let $\Pi_{\QQ_M(\Pi)}$ denote a simple $\QQ_M(\Pi)$-rational $\mathcal H_{\QQ_M(\Pi)}(G(\Adeles^{(\infty)}),K_{(\infty)})$-module occuring in $H_{\QQ_M(\Pi)}$. We consider the division algebra $A_{\QQ_M(\Pi)}=\End(\Pi_{\QQ_M(\Pi)})$ over $\QQ_M(\Pi)$. Let
$$
C\;\subseteq\;A_{\QQ_M(\Pi)}
$$
denote the center, and fix an embedding $C\to\CC$ extending the given embedding $\QQ_M(\Pi)\to\CC$. By restriction of the codomain this also fixes an embedding $C\to\overline{\QQ}\subseteq\CC$. On the one hand, by Burnside's Theorem,
$$
A_{\QQ_M(\Pi)}\otimes_{C}\overline{\QQ}\;\cong\;M_n(\overline{\QQ}),
$$
for some $n\geq 1$. On the other hand,
$$
A_{\QQ_M(\Pi)}\otimes_{\QQ_M(\Pi)}\overline{\QQ}\;=
(A_{\QQ_M(\Pi)}\otimes_{\QQ_M(\Pi)}C)
\otimes_{C}\overline{\QQ},
$$
and, as $C$-algebras,
$$
A_{\QQ_M(\Pi)}\otimes_{\QQ_M(\Pi)}C\;\cong\;A_{\QQ_M(\Pi)}^{[C:\QQ_M(\Pi)]}.
$$
Therefore, by the analog of Proposition \ref{prop:hombasechange} for Hecke modules,
$$
\End(\Pi_{\QQ_M(\Pi)}\otimes_{\QQ_M(\Pi)}\overline{\QQ})\;=\;A_{\QQ_M(\Pi)}\otimes_{\QQ_M(\Pi)}\overline{\QQ}\;\cong\;(M_n(\overline{\QQ}))^{[C:\QQ_M(\Pi)]},
$$
where
\begin{equation}
n\cdot[C:\QQ_M(\Pi)]\;=\;m(\Pi_{(\infty)}^{K_{(\infty)}};\Pi_{\QQ_M(\Pi)}).
\label{eq:firstmultiplicityrelation}
\end{equation}
A finite field extension $F/\QQ_M(\Pi)$ is a minimal spliting field for $A_{\QQ_M(\Pi)}$ if and only if there is an embedding $F\to A_{\QQ_M(\Pi)}$ of $\QQ_M(\Pi)$-algebras identifying $F$ with a maximal (commutative) subfield of $A_{\QQ_M(\Pi)}$. Then $F$ contains the center $C$ of $A_{\QQ_M(\Pi)}$ and
$$
\dim_C F\;=\;n,
$$
hence \eqref{eq:firstmultiplicityrelation} gives
\begin{equation}
[F:\QQ_M(\Pi)]\;=\;m(\Pi_{(\infty)}^{K_{(\infty)}};\Pi_{\QQ_M(\Pi)}).
\label{eq:degreemultiplicityrelation}
\end{equation}
Therefore,
$$
[F:\QQ_M(\Pi)]\cdot
m(\Pi_{\QQ_M(\Pi)};H_{\bullet}^q(\mathscr X_G(K_{(\infty)});\underline{M^\vee}(\QQ_M(\Pi))))\;=\;
$$
\begin{equation}
m(\Pi_{(\infty)}^{K_{(\infty)}};H_{\bullet}^q(\mathscr X_G(K_{(\infty)});\underline{M^\vee}(\CC))).
\label{eq:secondmultiplicityrelation}
\end{equation}
This proves (c).

This argument also shows that for any $F$-rational model $\Pi_{(\infty),F}^{K_{(\infty)}}$ of $\Pi_{(\infty)}^{K_{(\infty)}}$ and any subfield $\QQ_M(\Pi)\subseteq E\subseteq F$,
$$
\res_{F/E}\Pi_{(\infty),F}^{K_{(\infty)}}
$$
is an irreducible Hecke module over $E$, for \eqref{eq:degreemultiplicityrelation} implies that
$$
\res_{E/\QQ_M(\Pi)}(
\res_{F/E}\Pi_{(\infty),F}^{K_{(\infty)}}
)\;=\;
\res_{F/\QQ_M(\Pi)}\Pi_{(\infty),F}^{K_{(\infty)}}\;=\;\Pi_{\QQ_M(\Pi)}
$$
is irreducible over $\QQ_M(\Pi)$.

If $\QQ_M\subseteq E\subseteq\QQ_M(\Pi)$, then
$$
\res_{F/E}\Pi_{(\infty),F}^{K_{(\infty)}}\;=\;
\res_{\QQ_M(\Pi)/E}\Pi_{\QQ_M(\Pi)}
$$
is defined over $E$ and decomposes over $\QQ_M(\Pi)$ into $[\QQ_M(\Pi):E]$ pairwise non-isomorphic Galois conjugates of $\Pi_{\QQ_M(\Pi)}$, and therefore is irreducible over $E$. This proves the irreducibility claims in (d).

By the universal property of restriction of scalars and the irreducibility we just proved, every embedding of $\Pi_{(\infty),F}^{K_{(\infty)}}$ into \eqref{eq:hbullet} (for $F$-rational coefficients) naturally induces an embedding of $\res_{F/E}\Pi_{(\infty),F}^{K_{(\infty)}}$ over the subfield $E$.
\end{proof}

\begin{corollary}\label{cor:cuspidalrationaldecomposition}
Let $M$ be a $G$-module defined over $\QQ_M$ and assume that cuspidal cohomology \eqref{eq:cuspidalcohomology} admits a $\QQ_M$-rational structure. Then we have a $\QQ_M$-rational decomposition
\begin{equation}
\varinjlim_{K_{(\infty)}}H_{\rm cusp}^q(\mathscr X_G(K_{(\infty)});\underline{M^\vee}(\QQ_M))\;\cong\;
\bigoplus_{\Pi_{(\infty)}}
m_{\Pi_{(\infty)}}^{M,q}
\res_{\QQ_M(\Pi)^{\rm rat}/\QQ_M}
\Pi_{(\infty),\QQ_M(\Pi)^{\rm rat}}
\label{eq:cuspidalrationaldecomposition}
\end{equation}
into irreducible $G(\Adeles^{(\infty)})$-modules. The sum ranges over irreducible $G(\Adeles^{(\infty)})$-modules contributing to inner cohomology in degree $q$ over $\CC$, $\QQ_M(\Pi)^{\rm rat}$ denotes a minimal field of definition for $\Pi_{(\infty)}$ and the multiplicities are given by
$$
m_{\Pi_{(\infty)}}^{M,q}\;=\;\frac{m_{\rm cusp}^{M,q}(\Pi_{(\infty)})}{[\QQ_M(\Pi)^{\rm rat}:\QQ_M]}.
$$
\end{corollary}

\begin{proof}
The irreducibility of each module $\res_{\QQ_M(\Pi)^{\rm rat}/\QQ_M}\Pi_{(\infty),\QQ_M(\Pi)^{\rm rat}}$ is clear by statement (d) of the Theorem, and also that we obtain for every monomorphism
$$
\Pi_{(\infty),\QQ_M(\Pi)^{\rm rat}}\;\to\;
\varinjlim_{K_{(\infty)}}H_{\rm cusp}^q(\mathscr X_G(K_{(\infty)});\underline{M^\vee}(\QQ_M(\Pi)^{\rm rat}))
$$
an embedding
$$
\res_{\QQ_M(\Pi)^{\rm rat}/\QQ_M}
\Pi_{(\infty),\QQ_M(\Pi)^{\rm rat}}\;\to\;
\varinjlim_{K_{(\infty)}}H_{\rm cusp}^q(\mathscr X_G(K_{(\infty)});\underline{M^\vee}(\QQ_M)).
$$
The existence of a decomposition now follows by induction. The multiplicity formula is a consequence of the proof of \eqref{eq:degreemultiplicity}.
\end{proof}

\begin{remark}
Without the hypothesis that cuspidal cohomology admits a $\QQ_M$-structure the same proof shows that the right hand side of \eqref{eq:cuspidalrationaldecomposition} embeds into the colimits of inner and singular cohomology.
\end{remark}

\begin{corollary}\label{cor:qbarcuspidalcohomology}
For any finite-dimensional rational representation $M$ of $G$, cuspidal cohomology \eqref{eq:cuspidalcohomology} carries a $\overline{\QQ}$-rational structure and is defined over a number field.
\end{corollary}

\begin{remark}
The field of definition of cuspidal cohomology of finite level may depend on the level $K_{(\infty)}$. Therefore, passing to the colimits over all $K_{(\infty)}$, we only obtain a $\overline{\QQ}$-rational structure for $G(\Adeles^{(\infty)})$-modules in Corollary \ref{cor:qbarcuspidalcohomology}.
\end{remark}

\begin{remark}
It is still unclear if cuspidal cohomology is a $\overline{\QQ}$-rational or even $\QQ_M$-rational subspace of inner cohomology \eqref{eq:innercohomology}.
\end{remark}

\begin{proof}
Over $\CC$ cuspidal cohomology decomposes into a direct sum of irreducibles, and by Theorem \ref{thm:finitedefinition} the same is true over $\overline{\QQ}$ (by the argument in the proof of Theorem \ref{thm:finitedefinition} we know that both decompositions are equivalent). Since cuspidal cohomology of finite level is of finite length, and each irreducible admits a model over a number field, the decomposition into absolutely irreducibles is already defined over a number field.
\end{proof}

As global analog of Theorem \ref{thm:finitedefinition} we obtain

\begin{theorem}\label{thm:gglobalrational}
Assume that every $K(\RR)$-conjugacy class of $\theta$-stable parabolic subalgebras of $\lieg_\CC$ admits a representative defined over $\QQ_K'$. If $\Pi$ is an irreducible cuspidal automorphic representation of $G(\Adeles)$ contributing to \eqref{eq:cuspidalcohomology} for an absolutely irreducible rational $G$-module $M$, then
\begin{itemize}
\item[(a)] The $(\lieg,K)$-module $\Pi_\infty^{(K)}$ admits a model over the composite field $\QQ_K'\QQ_M$.
\item[(b)] The $(\lieg,K)\times G(\Adeles^{(\infty)})$-module $\Pi^{(K)}$ admits a model over the number field $\QQ_K'\QQ_M(\Pi)^{\rm rat}$.
\end{itemize}
\end{theorem}

\begin{remark}
We may replace $\QQ_K'$ by $\QQ_K$ in the statement of the Theorem if the standard module $A_\lieq(0)$ with trivial infinitesimal character occuring in the same coherent family as $\Pi_\infty^{(K)}$ admits a model over $\QQ_K$, cf.\ Theorem \ref{thm:realstandardmodules} for a list of known cases. We also emphasize that $\QQ_M$ is not necessarily contained in $\QQ(\Pi)^{\rm rat}$.
\end{remark}

\begin{proof}
Since by Vogan-Zuckerman \cite{voganzuckerman1984} the Harish-Chandra module $\Pi_\infty^{(K)}$ is a cohomologically induced standard module $A_\lieq(\lambda)_\CC$ where $\lambda$ is the highest weight of $M$, the claim about the field of definition of $\Pi_\infty^{(K)}$ follows from Proposition \ref{prop:standardmodules}.

The existence of the global rational structure then is a direct consequence of the above and Theorem \ref{thm:finitedefinition}.
\end{proof}

\subsection{A global rational structure on cusp forms}\label{sec:globalrationality}

Given a rational $G$-module $M$ defined over $\QQ_M$, we define the Hilbert space sum 
\begin{equation}
L_0^2(G(\QQ)\backslash G(\Adeles);M)\;:=\;
\widehat{\bigoplus_{\begin{subarray}c\omega\\\omega_\infty\subseteq M\end{subarray}}}L_0^2(G(\QQ)\backslash G(\Adeles);\omega)
\label{eq:l2centralcharactersum}
\end{equation}
of the spaces of cusp forms \eqref{eq:l2cuspformsomega} where $\omega$ ranges over all quasi-characters of the center $Z(\Adeles)$ of $G$, whose infinity component occurs in $M$. We remark that these $\omega$ are $C$-algebraic in the sense of \cite{buzzardgee2011}.

We consider the subspace
$$
L_0^2(G(\QQ)\backslash G(\Adeles);M)_{\rm coh}\;\subseteq\;
L_0^2(G(\QQ)\backslash G(\Adeles);M)
$$
which is given by the closure of the span of the sum of the irreducible cuspidal automorphic representations $\Pi$ which are cohomological with respect to $M$ (i.e.\ \eqref{eq:gkcohomology} is satisfied). All these representations occur in the right hand side, therefore the left hand side is well defined. By Schur's Lemma each such $\Pi$ occurs in a uniquely determined summand of the right hand side of \eqref{eq:l2centralcharactersum}.

We have the following globalization of Corollary \ref{cor:qbarcuspidalcohomology}.

\begin{theorem}\label{thm:l2qbarstructure}
For any rational $G$-representation $M$ defined over a number field $\QQ_M$ the $G(\Adeles)$-representation
\begin{equation}
L_0^2(G(\QQ)\backslash G(\Adeles);M)_{\rm coh}
\label{eq:l2coh}
\end{equation}
admits a rational structure over $\overline{\QQ}$ in the following sense: There is a dense subspace
$$
L_0^2(G(\QQ)\backslash G(\Adeles);M)_{\rm coh}^{K}
$$
of \eqref{eq:l2coh} contained in the smooth $K(\CC)$-finite vectors which admits a $\overline{\QQ}$-rational model
\begin{equation}
L_0^2(G(\QQ)\backslash G(\Adeles);M)_{\rm coh,\overline{\QQ}}
\label{eq:l2qbarstructure}
\end{equation}
as $(\lieg,K)\times G(\Adeles^{(\infty)})$-module. It enjoys the following properties:
\begin{itemize}
\item[(a)] Topologically irreducible subquotients of \eqref{eq:l2qbarstructure} correspond bijectively to irreducible subquotients of \eqref{eq:l2coh}, the relation is given by taking closures.
\item[(b)] Taking $(\lieg,K_\infty)$-cohomology in degree $q$ over $\overline{\QQ}$ produces a $G(\Adeles^{(\infty)})$-module isomorphic to the inductive limit of the $\overline{\QQ}$-rational structure on cuspidal cohomology in Corollary \ref{cor:qbarcuspidalcohomology}.
\end{itemize}
\end{theorem}

\begin{proof}
By Theorem \ref{thm:finitedefinition} the underlying $(\lieg,K)$-module of each irreducible $\Pi$ contributing to \eqref{eq:l2coh} admits a model over a number field. Since \eqref{eq:l2coh} decomposes into a direct Hilbert sum of such $\Pi$, the existence of a $\overline{\QQ}$-rational structure as claimed follows if we define \eqref{eq:l2qbarstructure} as the algebraic direct sum of the $\overline{\QQ}$-rational models of the underling $(\lieg,K)$-modules of the irreducibles $\Pi$.

The claim about preservation of irreducible subquotients (direct summands in the case at hand) is clear and the rest of the argument is standard.

The last statement is a consequence of the isomorphism \eqref{eq:cuspidalcohomologyiso} and the Homological Base Change Theorem.
\end{proof}

We expect that the field of rationality of the rational structure \eqref{eq:l2qbarstructure} from Theorem \ref{thm:l2qbarstructure} is a subfield of $\QQ_{K'}\QQ_M$, see Proposition \ref{prop:twistingmultiplicities} below.

If $\Pi$ is a factorizable automorphic representation of $G(\Adeles)$, we define for $\sigma\in\Aut(\CC/\QQ_K)$ the smooth $G(\Adeles)$-representation
$$
\Pi^\sigma\;:=\;\Pi_\infty^\sigma
\hat\otimes\Pi_{(\infty)}^\sigma,
$$
where $\Pi_\infty^\sigma$ is the Casselman-Wallach completion of $\left(\Pi_\infty^{(K)}\right)^\sigma$ (we refer the reader to \cite{casselman1989,bernsteinkrotz2014} for the notion of Casselman-Wallach completion).

\begin{conjecture}\label{conj:automorphictwisting}
For any irreducible cuspidal automorphic representation $\Pi$ of $G(\Adeles)$ which is cohomological with respect to a rational $G$-representation $M$, the representation $\Pi^\sigma$ of $G(\Adeles)$ is cuspidal automorphic whenever $\sigma\in\Aut(\CC/\QQ_K)$.
\end{conjecture}

We remark that $\Pi^\sigma$ is necessarily irreducible and if it is cuspidal automorphic, it is automatically cohomological with respect to $M^\sigma$.

We also expect Conjecture \ref{conj:automorphictwisting} to be true for $C$-algebraic representations in the sense of Buzzard and Gee \cite{buzzardgee2011}. However for our purpose at hand (Proposition \ref{prop:twistingmultiplicities} below) the statement for cohomological representations suffices.

\begin{remark}
Theorems 4.2.3 and 4.3.1 of \cite{blasiusharrisramakrishnan1994} provide evidence for Conjecture \ref{conj:automorphictwisting} for a larger class of representations of groups of Hermitian type.
\end{remark}

The statement of Conjecture \ref{conj:automorphictwisting} is equivalent to saying that the map
$$
\Pi_\infty^{\rm CW}\;\mapsto\;\Pi_\infty^\sigma
$$
induces a bijection between $\Phi_{\infty}^{\rm cusp}(\Pi_{(\infty)})$ and $\Phi_{\infty}^{\rm cusp}(\Pi_{(\infty)}^\sigma)$ (cf.\ \eqref{eq:cohomologicallpacket}).

In particular it implies that $\Phi_{\infty}^{\rm cusp}(\Pi_{(\infty)}^\sigma)$ is non-empty, which in turn is equivalent to saying that $\Pi_{(\infty)}^\sigma$ is the finite part of a cuspidal automorphic representation. This statement is implied by Proposition \ref{prop:cuspidalrationality} for sufficiently regular weights, because if $M$ satisfies the hypothesis of Proposition \ref{prop:cuspidalrationality}, then $M^\sigma$ does so as well.

We conclude that if Conjecture \eqref{conj:automorphictwisting} is true, then for every $\Pi$ which occurs in \eqref{eq:l2coh} and every $\sigma\in\Aut(\CC/\QQ_K\QQ_M)$, the representation $\Pi^\sigma$ occurs in \eqref{eq:l2coh}. However it is not clear that the multiplicities agree. Therefore we are naturally led to the stronger

\begin{conjecture}\label{conj:twistingmultiplicities}
Let $\Pi$ be an irreducible cohomological cuspidal automorphic representation of $G(\Adeles)$. Then we have for every $\sigma\in\Aut(\CC/\QQ_K)$ an identity
\begin{equation}
m_{L^2_0}(\Pi_\infty^{\rm CW}\otimes\Pi_{(\infty)})\;=\;
m_{L^2_0}(\Pi_\infty^{\sigma}\otimes\Pi_{(\infty)}^\sigma).
\label{eq:twistingmultiplicities}
\end{equation}
of multiplicities.
\end{conjecture}

\begin{proposition}\label{prop:twistingmultiplicities}
  Assume Conjecture \ref{conj:twistingmultiplicities}. Then for every rational $G$-representation $M$ the field of rationality of \eqref{eq:l2qbarstructure} is $\QQ_K\QQ(M)$, where $\QQ(M)$ denotes the field of rationality of $M$.
\end{proposition}

\begin{proof}
The space \eqref{eq:l2qbarstructure} is a direct sum of irreducibles, and for each irreducible summand $X$ occuring the multiplicity $m>0$, the identity\eqref{eq:twistingmultiplicities} implies that for $\sigma\in\Aut(\CC/\QQ_K\QQ(M))$, $X^\sigma$ occurs in \eqref{eq:l2qbarstructure} with multiplicity $m$ as well. Therefore the $\sigma$-twist of \eqref{eq:l2qbarstructure} is isomorphic to \eqref{eq:l2qbarstructure}.
\end{proof}

Despite the fact that each irreducible constituent $X$ of \eqref{eq:l2qbarstructure} is defined over a number field, the lack of control over the field of definition as a possibly non-trivial extension of the field of rationality prevents us from deducing the existence of a rational structure defined over a number field.

In the case of $\GL(n)$ we will prove an optimal result for the field of definition in Theorem \ref{thm:globalglnrationality} below. In general we have

\begin{proposition}\label{prop:twistingcohomologicalmultiplicities}
Assume Conjecture \ref{conj:twistingmultiplicities} is true. Then for any irreducible cuspidal automorphic representation $\Pi$ of $G(\Adeles)$, every $\sigma\in\Aut(\CC/\QQ_K)$ and every degree $q$,
$$
m_{\rm cusp}^{M,q}(\Pi_{(\infty)})\;=\;
m_{\rm cusp}^{M^\sigma,q}(\Pi_{(\infty)}^\sigma).
$$
\end{proposition}

\begin{proof}
We begin with the observation that we have an identity
$$
m_{\rm coh}^{M,q}(\Pi_\infty^{\rm CW})\;=\;
m_{\rm coh}^{M^\sigma,q}(\Pi_\infty^{\sigma}),
$$
which is an immediate consequence of the Homological Base Change Theorem. Therefore formula \eqref{eq:cuspidalmultiplicityformula} implies the claim.
\end{proof}

The conclusion of Proposition \ref{prop:twistingcohomologicalmultiplicities} is also implied by the conjectural statement that cuspidal cohomology admits a model over $\QQ_K\QQ_M$ (or even $\QQ_M)$. Unfortunately, without improving Proposition \ref{prop:twistingmultiplicities} to a statement about a field of definition finite over $\QQ$, the converse seems currently out of reach, even admitting Conjecture \ref{conj:automorphictwisting}.

\subsection{Automorphic representations in the Hermitian case}\label{sec:hermitiancase}

In this section we put ourselves in the context of \cite{blasiusharrisramakrishnan1994}, i.e.\ $G$ is a connected reductive group over $\QQ$ of Hermitian type, together with a $G(\RR)$ conjugacy class $\mathscr X$ of morphisms
$$
h:\quad \res_{\CC/\RR}\GL_1\to G_\RR
$$
over $\RR$ satisfying Deligne's axioms \cite{deligne1979b} of a Shimura datum. We also suppose that condition (1.1.3) of \cite{blasiusharrisramakrishnan1994} is satisfied: The maximal $\QQ$-split torus in the center $Z$ of $G$ is also a maximal $\RR$-split torus in $Z_\RR$. Then
$$
\mathscr X\;\cong\;G(\RR)/K_\infty
$$
where $K_\infty$ is the centralizer of $h(\CC^\times)$ in $G(\RR)$, and for neat compact open $K_{(\infty)}$ the locally symmetric spaces $\mathscr X_G(G_{(\infty)})$ are the complex points of quasi-projective algebraic varieties defined over the reflex field $E$, the so-called \lq{}canonical models\rq{}. It is known that $E$ is a number field and is characterized by the property that $\Aut(\CC/E)$ is the stabilizer of the restriction of the complexification of $h$ to $\GL_1(\CC)$ (the latter naturally embeds into $\res_{\CC/\RR}\GL_1(\CC)$). For arbitrary $\sigma\in\Aut(\CC/\QQ)$ there is another canonically defined twisted Shimura datum, which gives rise to the $\sigma$ twisted canonical models. We refer to section 1 of \cite{blasiusharrisramakrishnan1994} for this and further details.

Deligne's axioms imply that we have a Hodge decomposition
\begin{equation}
\lieg_\CC\;=\;\liek_{\infty,\CC}\oplus\liep^+\oplus\liep^-
\label{eq:ghodgedecomposition}
\end{equation}
into isotypic components of the action of $h(\res_{\CC/\RR}\GL_1(\CC))$ for the characters $1$, $z\overline{z}^{-1}$, and $\overline{z}z^{-1}$ respectively. The spaces $\liep^\pm$ are abelian Lie subalgebras of $\lieg_\CC$ and canonically identified with the (anti-)holomorphic tangent space of $\mathscr X$.

Via the decomposition \eqref{eq:ghodgedecomposition} we define the complex parabolic subalgebra
$$
Q\;:=\;K_{\infty,\CC}\cdot\exp(\liep^-)
$$
in the complexification $G(\CC)$.

We assume additionally that $G(\RR)$ has the following property: Let $G^\CC$ denote the analytic group of matrices with Lie algebra the complexification $\lieg_\CC$ of the Lie algebra of $G$. Then $G(\RR)\subseteq G^{\CC}\cdot Z(G(\RR))$ where $Z(G(\RR))$ denotes the centralizer of $G(\RR)$ in the total general linear group of matrices under consideration.

Let $H\subseteq K_\infty^0$ denote a maximally compact Cartan subgroup. Then $H$ is also a Cartan subgroup of $G(\RR)^0$. Consider the root system $\Phi=\Delta(\lieg_\CC,\lieh_\CC)$. Write $\Phi_{\rm c}$ (resp.\ $\Phi_{\rm n}$) for the subsets of (non-)compact roots in $\Phi$. Choose a positive system $\Phi^+\subseteq\Phi$ in such a way that
$$
\Phi^+\cap\Phi_{\rm n}\;=\;\Delta(\liep^+,\lieh_\CC).
$$


Following section 3 of loc.\ cit.\ we consider the following class of representations of $G(\RR)$. Let $\lambda\in\lieh_\RR^*$ denote an analytically integral linear form, which is not orthogonal to any member of $\Phi_{\rm c}$. Let $\Psi\subseteq\Phi$ a positive system rendering $\lambda$ dominant, and write $\pi(\lambda,\Psi)$ for the nontrivial irreducible tempered representation of ${\mathscr D}(G(\RR)^0)$ with Harish-Chandra parameter $\lambda$ and positive system $\Psi$. Then for regular $\lambda$ the representation $\pi(\lambda,\Psi)$ is a member of the {\em discrete series} of ${\mathscr D}(G(\RR)^0)$, and otherwise it is a {\em nondegenerate limits of discrete series} representation.

By the work of Harish-Chandra we know that $\pi(\lambda,\Psi)\cong\pi(\lambda',\Psi')$ if and only if the pairs $(\lambda,\Psi)$ and $(\lambda',\Psi')$ are conjugate by a member of the compact Weyl group $W(\liek_\CC,\lieh_\CC)$. In particular we may assume without loss of generality that
$$
\Phi^+\cap\Phi_{\rm c}\;\subseteq\;\Psi.
$$
Then $\lambda-\rho(\Phi^+)$ is dominant with respect to $\Phi^+\cap\Phi_{\rm c}$ and integral.

The representation $\pi(\lambda,\Psi)$ extends to a representation of $G(\RR)$ as follows. Set
$$
Z_\infty\;\;:=\;\;Z(\RR)\cap K_\infty.
$$
Then
$$
G(\RR)\;=\;G_\pm(\RR)\cdot Z_\infty^0
$$
where the two groups on the right hand side have trivial intersection and ${\mathscr D}(G(\RR)^0)$ is of finite index in $G_\pm(\RR)$. Choose any irreducible direct summand
$$
\pi_\pm(\lambda,\Psi)\;\subseteq\;\ind_{{\mathscr D}(G(\RR)^0)}^{G_\pm(\RR)}\pi(\lambda,\Psi).
$$
Then, as a submodule of the induced representation, the isomorphism class of $\pi_\pm(\lambda,\Psi)$ is characterized by its central character $\omega_\pm(\lambda,\Psi)$.

For every character $\beta:Z_\infty\to\CC^\times$, which occurs in the restriction of a rational representation of $K_\infty$, and which agrees on $Z_\infty\cap G_\pm(\RR)$ with $\omega_\pm(\lambda,\Psi)$, we have a unique extension of $\pi_\pm(\lambda,\Psi)$ to a representation $\pi_\pm(\lambda,\Psi)\otimes\beta$ of $G(\RR)$.

Following \cite{blasiusharrisramakrishnan1994}, we define the sets
$$
\Sigma_{\bullet}\;:=\;
\{
\pi_\pm(\lambda,\Psi)\otimes\beta\;\mid\;
\lambda,\Psi,\;\text{as above and}\;\pi(\lambda,\Psi)\;\text{in the}\;
$$
$$
%
\text{(limits) of discrete series of ${\mathscr D}(G(\RR)^0)$ for $\bullet=\lq{\rm d}\rq$ ($\bullet=\lq{\rm ld}\rq$).}
\}.
$$


Assume that $K\subseteq G$ is a semi-admissible model of a maximal compact subgroup of $G$, defined over a number field $\QQ_K$ with a distinguished real place $v_\infty$.

\begin{theorem}\label{thm:hermitianrationality}
If $\Pi$ is an irreducible cuspidal automorphic representation of $G(\Adeles)$ with infinity component $\Pi_\infty\in\Sigma_{\rm d}\cup\Sigma_{\rm ld}$. Then the $(\lieg,K)\times G(\Adeles^{(\infty)})$-module $\Pi^{(K)}$ admits a model over a number field $F$.
\end{theorem}

\begin{proof}
We know by Theorem 3.2.2 of \cite{blasiusharrisramakrishnan1994} that the finite part $\Pi_{(\infty)}$ of $\Pi$ admits a model over a number field.

By the hypothesis we have $\Pi_\infty\cong\pi_\pm(\lambda,\Psi)\otimes\beta$. Therefore we find a $\theta$-stable Borel $\lieq\subseteq\lieg'$, depending on the positive system $\Psi$, and defined over a finite extension $\QQ_K'$ of $\QQ_K$ with the property that $\Pi_\infty^{(K)}=A_\lieq(\lambda_\beta)$ for a suitable character $\lambda_\beta$ of  of the Levi pair $(\liel,L\cap K)$ of $(\lieq,L\cap K)$, depending on the characters $\lambda$ and $\beta$. Proposition \ref{prop:standardmodules} applies and shows that $A_\lieq(\lambda_\beta)$ admits a model over a number field, and the claim follows.
\end{proof}

\begin{remark}
By Theorem 4.4.1 in \cite{blasiusharrisramakrishnan1994} the field of rationality $\QQ(\Pi)$ of the finite part of $\Pi$ is known to be totally real or a CM field. Applying our argument from the proof of equation \eqref{eq:degreemultiplicity} in Theorem \ref{thm:finitedefinition} to coherent cohomology one may deduce a similar bound for the degree of the field of definition of $\Pi_{(\infty)}$ over $\QQ(\Pi)$. Likewise Theorem \ref{thm:realstandardmodules} implies the existence of optimal fields of definition for $\Pi_\infty$ in the cases where it applies.
\end{remark}

\subsection{Automorphic representations of $\GL(n)$}\label{sec:glncase}

In this section we specialize to the following situation. We let $F$ denote a number field and $G:=\res_{F/\QQ}\GL(n)$. We have
$$
G(\RR)\;=\;\prod_{v\;\text{real}}\GL_n(\RR)\times \prod_{v\;\text{complex}}\GL_n(\CC),
$$
where in the first product $v$ runs over the real places of $F$ and in the second $v$ runs through the complex places, i.e.\ pairs of complex conjugate embeddings $v,\overline{v}:F\to\CC$.

The cohomological representations in the unitary duals of $\GL_n(\RR)$ and $\GL_n(\CC)$ have been identified by Speh \cite{speh1983} and Enright \cite{enright1979} respectively, predating the general approach of Vogan-Zuckerman \cite{voganzuckerman1984}, which at the time was based on certain unitarity assumptions.

Fix an admissible model $K\subseteq G$ of a maximal compact subgroup $G(\RR)$ in the sense of section \ref{sec:admissiblegroups}, defined over a number field $\QQ_K$. The field $\QQ_K$ comes with a distinguished real place $v_\infty$. We also assume that we find models of $\theta$-stable Borel subalgebras over a quadratic extension $\QQ_K'$ of $\QQ_K$. In such an extension $v_\infty$ ramifies. Write $\lieg$ for $\QQ_K$-Lie algebra of $G$.

\begin{remark}
For $F$ totally real or a CM field we know from the explicit constructions in section \ref{sec:explicitmodels} that we may choose $\QQ_K=\QQ$ and $\QQ_K'$ is an imaginary quadratic field. We may arrange $\QQ_K'=\QQ(\sqrt{-1})$ for totally real $F$.
\end{remark}

In this section we prove

\begin{theorem}\label{thm:glnrationality}
Keep the notation as before. Let $M$ be an absolutely irreducible finite-dimensional $G$-module, which is defined over a number field $\QQ_M$. Then the non-degenerate complex infinitesimally unitary irreducible $(\lieg,K)$-module $V_\CC$ satisfying
$$
H^\bullet(\lieg,K_\infty^0;V_\CC\otimes M)\;\neq\;0,
$$
has a unique model $V_{\QQ_M}$ over $\QQ_K\QQ_M$.
\end{theorem}

\begin{proof}
Up to a quadratic twist, this follows from Theorem \ref{thm:realstandardmodules}, (a).
\end{proof}

\begin{theorem}\label{thm:globalglnrationality}
Let $\Pi$ be an irreducible cuspidal regular algebraic representation of $G(\Adeles_\QQ)=\GL_n(\Adeles_F)$ with field of rationality $\QQ(\Pi)$ in the sense of Clozel \cite{clozel1990}. Then $\Pi$ is defined over $\QQ_K(\Pi)$ in the following sense:
\begin{itemize}
\item[(i)] There exists a unique $(\lieg,K)\times G(\Adeles_\QQ^{(\infty)})$-module $\Pi_{\QQ_K(\Pi)}$ over $\QQ_K(\Pi)$ which is a model of $\Pi^{(K)}$.
\item[(ii)] The module $\Pi_{\QQ_K(\Pi)}$ is irreducible and its image in $\Pi^{(K)}$ is unique up to homotheties.
\end{itemize}
\end{theorem}

\begin{proof}
By Clozel's Th\'eor\`eme 3.13 in \cite{clozel1990}, we know that the field of definition of the finite part $\Pi_{(\infty)}$ agrees with $\QQ(\Pi)$. Furthermore the field $\QQ(\Pi)$ contains the field of rationality $\QQ(M)$ of the absolutely irreducible $G$-module $M$ for which $\Pi$ contributes to cuspidal cohomology \eqref{eq:cuspidalcohomology}. Since $G$ is quasi-split, the field of definition $\QQ_M$ of $M$ agrees with $\QQ(M)$ (cf.\ \cite{boreltits1965}). By Theorem \ref{thm:glnrationality} the $(\lieg,K)$-module $\Pi^{(K)}$ admits a model over $\QQ_K\QQ_M\subseteq\QQ_K(\Pi)$, and (i) follows.

Statement (ii) is a consequence of Proposition \ref{prop:uniquemodels}.
\end{proof}

\begin{remark}
Considering an algebraic but not necessarily regular algebraic representation $\Pi$, part (iii) of Proposition \ref{prop:standardmodules} shows that the statement and proof of Theorem \ref{thm:globalglnrationality} generalizes to this case, whenever the finite part of $\Pi$ is rational over a number field. This is known for $1\leq n\leq 2$ and $F$ totally real, since the results from section \ref{sec:hermitiancase} apply, but unkown in general.
\end{remark}

Theorem \ref{thm:globalglnrationality} allows us to improve Theorem \ref{thm:l2qbarstructure}: For $F$ totally real or a CM field it implies the existence of a global $\QQ$-structure on regular algebraic automorphic cusp forms for appropriate choices of $K$. In general we obtain a $\QQ_K$-structure.

In the case of $\GL(n)$ the space \eqref{eq:l2coh} of cohomological cusp forms has an intrinsic description. For any character $\omega$ of the center $Z(\lieg)$ of the universal enveloping algebra of $G$, which we assume regular algebraic in the sense that it occurs as the infinitesimal character of an absolutely irreducible rational $G$-representation $M$, which we assume to be essentially conjugate self-dual over $\QQ$, we consider the space
$$
L_0^2(\GL_n(F)Z(\RR)^0\backslash{}\GL_n(\Adeles_F);\omega)
$$
of automorphic cusp forms as in the introduction, and similarly
\begin{equation}
L_0^2(\GL_n(F)Z(\RR)^0\backslash{}\GL_n(\Adeles_F);\res_{\QQ_K(\omega)/\QQ_K}\omega)
\;=\;
\label{eq:l2cuspformst}
\end{equation}
$$
\quad\bigoplus_{\tau:\QQ_K(\omega)\to\CC}
L_0^2(\GL_n(F)Z(\RR)^0\backslash{}\GL_n(\Adeles_F);\omega^\tau).
$$
Here $\tau$ ranges over the embeddings extending the fixed embedding $v_\infty:\QQ_K\to\CC$. The left hand side may be interpreted as a space of vector valued cusp forms. By construction this space should be defined over $\QQ_K$. This is indeed the case.

\begin{theorem}\label{thm:globalqrationality}
As a representation of $\GL_n(\Adeles_F)$, the space \eqref{eq:l2cuspformst} is defined over $\QQ_K$, i.e.\ there is a  basis of the space \eqref{eq:l2cuspformst} which generates a $\QQ_K$-subspace
\begin{equation}
L_0^2(\GL_n(F)Z(\RR)^0\backslash{}\GL_n(\Adeles_F);\res_{\QQ_K(\omega)/\QQ_K}\omega)_{\QQ_K}
\label{eq:l2cuspformsqbart}
\end{equation}
This rational structure has the following properties:
\begin{itemize}
\item[(a)] The complexification
$$
L_0^2(\GL_n(F)Z(\RR)^0\backslash{}\GL_n(\Adeles_F);\res_{\QQ_K(\omega)/\QQ_K}\omega)_{\QQ_K}\otimes\CC
$$
is naturally identified with the subspace of smooth $K$-finite vectors in \eqref{eq:l2cuspformst}.
\item[(b)] The space \eqref{eq:l2cuspformsqbart} is stable under the natural action of $(\lieg,K)\times\GL_n(\Adeles_F^{(\infty)})$.
\item[(c)] For any field extension $F'/F/\QQ_K$ the extension of scalars functor $-\otimes_F F'$ as in (a) is faithful and preserves extension classes.
\item[(d)] To each irreducible subquotient $\Pi$ of the space \eqref{eq:l2cuspformst} corresponds a unique irreducible subquotient $\Pi_{\overline{\QQ}}$ of the $\overline{\QQ}$-rational structure induced by \eqref{eq:l2cuspformsqbart}, and vice versa.
\item[(e)] The module $\Pi_{\overline{\QQ}}$ is defined over its field of rationality $\QQ(\Pi_{\overline{\QQ}})$, which agrees with Clozel's field of rationality $\QQ_K(\Pi)$ over $\QQ_K$ and is a number field.
\item[(f)] As an abstract $\QQ_K(\Pi)$-rational structure on $\Pi$, $\Pi_{\QQ_K(\Pi)}\subseteq\Pi$ from (d) is unique up to complex homotheties.
\item[(g)] For each $\tau\in\Aut(\CC/\QQ_K)$ and $\Pi$ as in (d),
$$
(\Pi_{\QQ_K(\Pi)})^\tau\;=\;\Pi_{\QQ_K(\Pi)}\otimes_{\QQ_K(\Pi),\tau}\QQ_K(\Pi^\tau)
$$
is the unique rational structure on $\Pi^\tau$.
\item[(h)] For every finite order character $\xi:F^\times\backslash\Adeles_F^\times\to\CC^\times$ and every $\tau\in\Aut(\CC/\QQ_K)$,
$$
(\Pi_{\QQ_K(\Pi)}\otimes\xi)^\tau=\Pi_{\QQ_K(\Pi)}^\tau\otimes\xi^\tau.
$$
\item[(i)] Taking $(\lieg,K)$-cohomology sends the $\QQ$-structure \eqref{eq:l2cuspformsqbart} to the natural $\QQ_K$-structure on cuspidal cohomology
\begin{equation}
H_{\rm cusp}^{\bullet}(\GL_n(F)\backslash{}\GL_n(\Adeles_F)/K_\infty;\res_{\QQ_K(\omega)/\QQ_K}M^\vee).
\label{eq:hcuspt}
\end{equation}
\end{itemize}
Moreover (a), (h) and (i), for a fixed degree within the cuspidal range, characterize \eqref{eq:l2cuspformsqbart} as a $\QQ_K$-subspace of \eqref{eq:l2cuspformst} uniquely. If $F$ is totally real or a CM field we may arrange $\QQ_K=\QQ$.
\end{theorem}

\begin{proof}
We first remark that $\res_{\QQ_K(\omega)/\QQ_K}M$ is defined over $\QQ_K$ and we have
$$
(\res_{\QQ_K(\omega)/\QQ_K}M)\otimes\CC\;=\;\bigoplus_{\tau:\QQ_K(\omega)\to\CC}M\otimes_{\QQ_K(\omega),\tau}\CC.
$$
Therefore every irreducible automorphic representation $\Pi$ occuring in \eqref{eq:l2cuspformst} is regular algebraic up to a twist by a finite order Hecke character. The twist is only necessary when $n$ is odd and $F$ admits a real embedding (cf.\ \cite{clozel1990}).

Let us assume first that $\Pi$ is regular algebraic. Then its finite part occurs in
$$
H_{\rm cusp}^{\bullet}(\GL_n(F)\backslash{}\GL_n(\Adeles_F)/K_\infty^0;M^{\vee}\otimes_{\QQ_M,\tau}\CC)
$$
for a uniquely determined embedding $\tau:\QQ_M\to\CC$, where for notational simplicity we assume $\QQ_K\subseteq\QQ_M$ in the sequel.

We observe that as $\Pi$ is regular algebraic so is $\Pi^\tau$ and the latter also occurs in \eqref{eq:l2cuspformst} for any $\tau\in\Aut(\CC/\QQ_K)$ by \cite{clozel1990}.

Furthermore, by Theorem \ref{thm:globalglnrationality} the subspace $\Pi^{(K)}$ of $K$-finite vectors is defined over the number field $\QQ_K(\Pi)$. Denote a $\QQ_K(\Pi)$-rational model for the latter by $\Pi_{\QQ_K(\Pi)}$. Then the $(\lieg,K)\times \GL_n(\Adeles_F^{(\infty)})$-module $\res_{\QQ_K(\Pi/\QQ_K)}\Pi_{\QQ_K(\Pi)}$ is defined over $\QQ_K$. Complexification gives
$$
(\res_{\QQ_K(\Pi/\QQ_K)}\Pi_{\QQ_K(\Pi)})\otimes\CC\;=\;
\bigoplus_{\tau:\QQ_K(\tau)\to\CC}\left(\Pi_{\QQ_K(\Pi)}\otimes_{\QQ_K(\Pi),\tau}\CC\right).
$$
If $\Pi$ is not cohomological, we find a quadratic Hecke character
$$
\xi:\quad F^\times\backslash\Adeles_F^\times\to\CC^\times
$$
with the property that
$$
\Pi\;\cong\;\Pi^0\otimes\xi
$$
and $\Pi^0$ cohomological, as $\Pi\otimes\xi^{-1}$ is again an automorphic representation of $\GL(n)$ by \cite{arthurclozel1989}. Since $\xi$ is defined over $\QQ$, we again conclude that $\QQ_K(\Pi)/\QQ_K$ is finite and $\Pi$ is defined over this field.

Summing up, this shows for the subspace of $K$-finite cusp forms, by multiplicity one for $\GL(n)$ \cite{piatetskishapiro1979,shalika1974},
$$
L_0^2(\GL_n(F)Z(\RR)^0\backslash{}\GL_n(\Adeles_F);\res_{\QQ_K(\omega)/\QQ_K}\omega)^{(K)}
\;=
\quad\bigoplus_{\Pi}
(\res_{\QQ_K(\Pi/\QQ_K)}\Pi_{\QQ_K(\Pi)})\otimes\CC,
$$
where on the right hand side $\Pi$ runs through a system of representatives for the Galois orbits
$$
\{\Pi^\tau\mid\tau\in\Aut(\CC/\QQ_K)\}
$$
occuring in \eqref{eq:l2cuspformst}. This proves the existence of a $\QQ_K$-rational model for the space of cusp forms.

By our previous results, the justification of the statements (a) to (i) is straightforward. Statements (a) and (b) follow form Theorem \ref{thm:globalglnrationality}. By (a) the left hand side of \eqref{eq:l2cuspformsqbart} is dense in \eqref{eq:l2cuspformst}. Statement (c) is a consequence of Corollary \ref{cor:extbasechange}, and (d) is a consequence of (b) and Theorem \ref{thm:globalglnrationality}, as is (e). Statement (f) follows from Hilbert's Satz 90, cf.\ Proposition \ref{prop:uniquemodels}. By construction (g) and (h) are true. Statement (i) follows from Matsushima's Formula and our Proposition \ref{prop:equicohomology}, a consequence of the Homological Base Change Theorem, and the fact that we may associate to an even finite order character $\xi$ a cohomology class of degree 0, such that cup product with this class realizes the twist with $\xi$. Fruthermore, if $\xi$ is odd, or if $\phi$ is a non-cohomological cusp form in \eqref{eq:l2cuspformst} generating an irreducible representation, the form
$$
\phi_\xi(g)\;:=\;\xi(\det(g))\cdot \phi(g)
$$
lies again in \eqref{eq:l2cuspformst} by Arthur-Clozel \cite{arthurclozel1989}, is $K$-finite if and only if $\phi$ is $K$-finite, and thus this correspondence may serve as the bridge to transport any normalization of the rational structure on the subspace of cohomoligcal forms to the space of non-cohomological forms. For an explicit exposition of the Eichler-Shimura map realizing Matsushima's formula in the case of totally real $F$ (yet valid in general), we refer to Section 6.2 in \cite{januszewski2014}. Finally the uniqueness statement is a consequence of the construction and the uniqueness statement (f).
\end{proof}

\subsection{Construction of periods}\label{sec:periods}

We keep the notation of the previous section and assume $\Pi$ is a regular algebraic representation of $\GL_n(\Adeles_F)$ with same infinitesimal character as the rational representation $M$ of $G$. By Theorem \ref{thm:globalglnrationality} we know the existence of a $\QQ_K(\Pi)$-rational structure
$$
\iota:\quad \Pi_{\QQ_K(\Pi)}\;\to\;\Pi^{(K)},
$$
which we interpret as a $\QQ_K(\Pi)$-rational intertwining map, where the right hand side denotes the complex $(\lieg,K)$-module unterlying the global representation. Matsushima's Formula states that for each degree $t_0\leq t\leq q_0$ in the cuspidal range we have a natural isomorphism
$$
\iota_q:\;\;
H^q(\lieg,K_\infty;\Pi\otimes M^\vee)\;\cong\;
H_{\rm cusp}^{q}(\GL_n(F)\backslash{}\GL_n(\Adeles_F)/K_\infty;M^{\vee}\otimes_{\QQ_M,\tau}\CC)[\Pi_{(\infty)}]
$$
The right hand side carries a natural $\QQ_K(\Pi)$-rational structure, and by the rationality of $(\lieg,K_\infty)$-cohomology, $\iota$ induces a natural $\QQ_K(\Pi)$-structure on the left hand side.

For notational simplicity we write
$$
m_q\;:=\;\dim H^q(\lieg,K_\infty;\Pi_\infty)
$$
for the multiplicity of $\Pi$ in cuspidal cohomology of degree $q$.

By Schur's Lemma, the $\QQ(\Pi)$-rational structures on $\left(\Pi^{(\infty)}\right)^{m_q}$ are parametrized by cosets $[g]$ in $\GL_{m_q}(\CC)/\GL_{m_q}(\QQ_K(\Pi))$. In particular, choosing a $\QQ_K(\Pi)$-rational basis of the relative Lie algebra cohomology of $\Pi\otimes M^\vee$ and a $\QQ_K(\Pi)$-rational basis on the $\Pi^{(\infty)}$-isotypical component in cuspidal cohomology, the transformation matrix $\Omega_{q}(\Pi,\iota)\in\GL_{m_q}(\CC)$ of $\iota_q$ with respect to these bases, may be interpreted as a {\em period matrix} associated to the pair $(\Pi,\iota)$. It enjoys the following properties.

\begin{theorem}\label{thm:periods}
For each irreducible cuspidal regular algebraic $\Pi$ as above and each $t_0\leq q\leq q_0$ in the cuspidal range and each rational structure $\iota:\Pi_{\QQ_K(\Pi)}\to\Pi$ on $\Pi$ there is a period matrix $\Omega_q(\Pi,\iota)\in\GL_{m_q}(\CC)$ with the following properties:
\begin{itemize}
\item[(a)] $\Omega_q(\Pi,\iota)$ is the transformation matrix transforming the rational structure $H^q(\lieg,K_\infty; \iota)$ on $(\lieg,K_\infty)$-cohomology into the natural $\QQ_K(\Pi)$-structure on cuspidal cohomology.
\item[(b)] The double coset $\GL_{m_q}(\QQ_K(\Pi))\Omega_q(\Pi,\iota)\GL_{m_q}(\QQ_K(\Pi))$ depends only on the pair $(\Pi,\iota)$ and the degree $q$.
\item[(c)] For each $c\in\CC^\times$ we have the relation
$$
\Omega_q(\Pi,c\cdot\iota)\;=\;c\cdot\Omega_q(\Pi,\iota).
$$
\item[(d)] The ratio
$$
\frac{1}
{\Omega_{t_0}(\Pi,\iota)}\cdot
\Omega_q(\Pi,\iota)\;\in\;\GL_{m_q}(\CC)
$$
is independent of $\iota$.
\end{itemize}
\end{theorem}

\begin{remark}
In (d) we may as well divide by the top degree period and obtain the same conclusion.
\end{remark}

In certain cases there are natural choices for $\iota$. If $n$ is even and admits a Shalika model, then we may use a natural {\em global} $\QQ(\Pi)$-structure on the Shalika model of $\Pi$ to normalize $\iota$ accordingly, i.e.\ in such a way that a $\QQ(\Pi)$-rational vector produces the correctly normalized $L$-function via the Friedberg-Jacquet integral representation. That this is indeed possible is shown in \cite{januszewski2015pre2}. This then should be used as normalization in the top degree $q=q_0$. A similar approach is available for Rankin-Selberg $L$-functions, where one may normalize the period in bottom degree inductively, cf.\ \cite{januszewski2015pre} (the results for $n\geq 4$ are conditional in this case, cf.\ loc.\ cit.\ Conjecture 5.5).

These normalizations are then be compatible with Deligne's Conjecture \cite{deligne1979}. Any motivic interpretation of the period matrices $\Omega_t(\Pi,\iota)$ for $t_0<q<q_0$ remains a challenge.

\addcontentsline{toc}{section}{References}
\bibliographystyle{plain}

\end{document}